%% file: Visualizing_Rank_Invariants.tex
\pdfoutput=1
\documentclass[11pt]{article}
\input{./VRI_Header}

\title{\textbf{Interactive Visualization of 2-D Persistence Modules}}
\author{Michael Lesnick\thanks{Columbia University, New York, NY, USA; \texttt{mlesnick@alumni.stanford.edu}} \and Matthew Wright\thanks{St.\ Olaf College, Northfield, MN, USA; \texttt{wright5@stolaf.edu}}}
\begin{document}

\date{}
\maketitle

\begin{abstract}
The goal of this work is to extend the standard persistent homology pipeline for exploratory data analysis to the 2-D persistence setting, in a practical, computationally efficient way.  To this end, we introduce RIVET, a software tool for the visualization of 2-D persistence modules, and present mathematical foundations for this tool.  RIVET provides an interactive visualization of the barcodes of 1-D affine slices of a 2-D persistence module $M$.  It also computes and visualizes the dimension of each vector space in $M$ and the bigraded Betti numbers of $M$.  At the heart of our computational approach is a novel data structure based on planar line arrangements, on which we can perform fast queries to find the barcode of any slice of $M$.  We present an efficient algorithm for constructing this data structure and establish bounds on its complexity.
\end{abstract}

\tableofcontents

\input{./VRI_Intro}             %Section 1
\input{./VRI_Preliminaries}     %Section 2
\input{./VRI_L_Arrangements}    %Section 3
\input{./VRI_L_Structures}      %Section 4
\input{./VRI_Computing_the_Arrangement}       %Section 5
\input{./VRI_Computing_Data_At_Faces}  %Section 6
\input{./VRI_Time_Complexity}  %Section 7
\input{./VRI_Optimizations}    %Section 8
\input{./VRI_Runtimes}         %Section 9

\section{Conclusion}

In this paper, we have introduced RIVET, a practical tool for visualization of 2-D persistence modules.  RIVET provides an interactive visualization of the barcodes of 1-D affine slices of a 2-D persistence module, as well as visualizations of the dimension function and bigraded Betti numbers of the module.  We have presented a mathematical theory for our visualization paradigm, centered around a novel data structure called an augmented arrangement.  We have also introduced and analyzed an algorithm for computing augmented arrangements, and described several strategies for improving the runtime of this algorithm in practice.  

In addition, we have presented timing data from preliminary experiments on the computation of augmented arrangements.  Though we have yet to incorporate several key optimizations into our code, the results demonstrate that our current implementation already scales well enough to be used to study bifitrations with millions of simplices.  With more implementation work, we expect RIVET to  scale well enough to be used in many of the same settings where 1-D persistence is currently used for exploratory data analysis.

From here, there are several natural directions to pursue.  Beyond continuing to improve our implementation of RIVET, we would like to:
\begin{itemize}
\item Apply RIVET to the exploratory analysis of scientific data.
\item Develop statistical foundations for our data analysis methodology.
\item Adapt the RIVET paradigm in the setting of 0-D homology to develop a tool for hierarchical clustering and interactive visualization of bidendrograms \cite{carlsson2010multiparameter}.
\item Extend the RIVET methodology to other generalized persistence settings, such cosheaves of vector spaces over $\R^2$ \cite{curry2013sheaves}, or cosheaves of 1-D persistence modules over $\R$.
\end{itemize}
We hope that RIVET will prove to be a useful addition to the existing arsenal of TDA tools.  
Regardless of how it ultimately fares in that regard, however, we feel that the broader program of developing practical computational tools for multidimensional persistence is a promising direction for TDA, and we hope that this work can draw attention to the possibilities for this.  We believe that there is room for a diverse set of approaches.  %In the future, we hope to see more activity by the TDA community on the practical, computational aspects of multidimensional persistence.

\subsection*{Acknowledgements}
This paper has benefited significantly from conversations with John Carlsson about point-line duality and discrete optimization.  We also thank Ulrich Bauer, Magnus Botnan, and Dmitriy Morozov, and Francesco Vaccarino for helpful discussions.  
The bulk of the work presented in this paper was carried out while the authors were postdoctoral fellows at the Institute for Mathematics and its Applications, with funds provided by the National Science Foundation.  Some of the work was completed while Mike was visiting Raul Rabadan's lab at Columbia University.  Thanks to everyone at the IMA and Columbia for their support and hospitality.  

\appendix 
\section{Appendix}
\input{./Interface_Details}

\input{./VRI_Algorithmic_Details.tex}

\bibliographystyle{abbrv}
{\small \bibliography{VRI_references} }

{\small \printnomenclature[0.5in] }	%notation index

\end{document}

%% file: VRI_Header.tex
\usepackage[letterpaper]{geometry}
\usepackage{bm}
\usepackage[colorlinks=true,linkcolor=blue,anchorcolor=blue,citecolor=blue,urlcolor=blue,plainpages=false,pdfpagelabels,bookmarks=false]{hyperref}
\hypersetup{linktocpage}
\usepackage{amsmath, amsthm, amsfonts, amssymb,sansmath}
\numberwithin{equation}{section}
\usepackage[capitalize]{cleveref}
\usepackage{graphicx}
\usepackage{algorithm}
\usepackage[noend]{algpseudocode}
\usepackage{comment}
\usepackage{xcolor}
\usepackage[all]{xy}
\usepackage{xspace}
\usepackage{enumerate}
\usepackage{tikz}
\usepackage{tikz-cd}
\usepackage[textsize=scriptsize,disable]{todonotes}
\usepackage[font=small,labelfont=bf]{caption}
\usepackage{fixltx2e}
\usepackage{etex}
\usepackage[labelfont=rm]{subcaption}
\usepackage{mdwlist}
\usepackage{rotating} %for sidewaysfigure

\usepackage[nottoc]{tocbibind}	%print an entry for the references
\usepackage{tocloft}		%this package allows customization of the table of contents, as follows:
\usepackage{array}
\usepackage{booktabs}
\usepackage{multirow}
\usepackage{ragged2e}
\newcolumntype{M}[1]{>{\RaggedRight\hspace{0pt}}m{#1}}

\usepackage[refpage,noprefix,intoc]{nomencl}
\makenomenclature

\makenomenclature

\usepackage[numbers,sort,compress]{natbib}

\let\originalparagraph\paragraph
\renewcommand{\paragraph}[2][.]{\originalparagraph{#2#1}}
\newtheorem{theorem}{Theorem}[section]
\newtheorem{proposition}[theorem]{Proposition}
\newtheorem{lemma}[theorem]{Lemma} 
\newtheorem{corollary}[theorem]{Corollary}

\theoremstyle{definition}
\newtheorem{example}[theorem]{Example}

\theoremstyle{definition}
\newtheorem{remark}[theorem]{Remark}

\newcommand\category[1]{\ensuremath{\mathbf{#1}}}
\DeclareMathOperator{\union}{\cup}
\DeclareMathOperator{\intrs}{\cap}
\DeclareMathOperator{\im}{im}

\DeclareMathOperator{\ceil}{cl}

\DeclareMathOperator{\coker}{coker}
\DeclareMathOperator{\obj}{obj}

\DeclareMathOperator{\rank}{rank}
\DeclareMathOperator{\rk}{rank}

\DeclareMathOperator{\id}{Id}
\DeclareMathOperator{\free}{Free}

\DeclareMathOperator{\push}{push}

\DeclareMathOperator{\dimf}{dimf}
\DeclareMathOperator{\Tor}{Tor}

\DeclareMathOperator{\nmod}{\kvect^{\RCat^n}}
\DeclareMathOperator{\gr}{gr}

\DeclareMathOperator{\crit}{crit}
\DeclareMathOperator{\Rips}{Rips}
\DeclareMathOperator{\BRips}{B-Rips}

\DeclareMathOperator{\supp}{supp}

\DeclareMathOperator{\comp}{\circ}

\DeclareMathOperator{\lub}{\textsc{lub}}
\DeclareMathOperator{\pairs}{Pairs}
\DeclareMathOperator{\ess}{Ess}
\DeclareMathOperator{\PE}{PE}

\newcommand{\cell}{\mathcal{A}}
\newcommand{\kVect}[0]{\category{Vect}}
\newcommand{\kvect}[0]{\category{Vect}}

\newcommand{\op}[0]{\mathrm{op}}
\newcommand{\indx}[0]{\mathrm{\zeta}}

\newcommand{\B}[1]{\mathcal B\ifthenelse{\equal{#1}{}}{}{(#1)}}
\newcommand{\Bi}[1]{{\mathcal B}_i\ifthenelse{\equal{#1}{}}{}{(#1)}}
\newcommand{\Bopen}[1]{U\ifthenelse{\equal{#1}{}}{}{(#1)}}

\newcommand{\barc}[1]{\mathcal B\ifthenelse{\equal{#1}{}}{}{(#1)}}

\newcommand{\A}[0]{\mathcal A}
\newcommand{\Arm}[0]{\mathrm A}
\newcommand{\Brm}[0]{\mathrm B}
\newcommand{\Crm}[0]{\mathrm C}

\newcommand{\C}[0]{\mathcal C}
\newcommand{\D}[0]{\mathcal D}
\newcommand{\e}[0]{E}

\newcommand{\F}[0]{\mathcal F}
\newcommand{\G}[0]{\mathcal G}

\renewcommand{\H}[0]{\mathcal H}
\newcommand{\I}[0]{\mathcal I}

\renewcommand{\L}[0]{\mathcal L}
\newcommand{\M}[0]{\mathcal M}

\renewcommand{\O}[0]{\mathcal O}

\newcommand{\pth}[0]{\Gamma}
\renewcommand{\P}[0]{\mathcal T}
\newcommand{\fir}[0]{FI-rep\@\xspace}
\newcommand{\firs}[0]{FI-reps\@\xspace}
\newcommand{\PCal}[0]{\mathcal P}
\newcommand{\Pgeqr}[0]{\mathcal{P}^{\geq r}}
\newcommand{\R}[0]{\mathbb R}
\newcommand{\Rc}[0]{\mathcal R}
\newcommand{\Sc}[0]{\mathcal S}
\newcommand{\Qc}[0]{\mathcal Q}
\renewcommand{\S}[0]{\mathcal A^\bullet}

\newcommand{\vine}[0]{\mathrm{vine}}

\renewcommand{\u}[2]{\mathsf{lifts}_{#1}\ifthenelse{\equal{#2}{}}{}{[#2]}}
\MakeRobust{\u}
\newcommand{\g}[2]{\mathsf{grades}_{#1}\ifthenelse{\equal{#2}{}}{}{[#2]}}
\newcommand{\sig}[2]{\mathsf{sig}_{#1}\ifthenelse{\equal{#2}{}}{}{[#2]}}
\newcommand{\sigInv}[2]{\mathsf{sigInv}_{#1}\ifthenelse{\equal{#2}{}}{}{[#2]}}
\MakeRobust{\sigInv}
\newcommand{\curr}[0]{\mathsf{y}}
\newcommand{\topsort}[0]{\mathsf{sortOneElement}}

\newcommand{\Nbb}[0]{\mathbb N}
\newcommand{\W}[0]{\mathcal W}

\newcommand{\Y}[0]{\mathcal Y}
\newcommand{\Z}[0]{\mathbb Z}

\newcommand{\cupinf}[1]{#1\union \infty}

\newcommand{\CoEx}[0]{\mathrm{E}}
\newcommand{\dual}[0]{\mathcal{D}}
\newcommand{\DiscreteMod}[0]{Q}

\newcommand{\ord}{\mathrm{ord}}

\newcommand{\ordfun}[0]{\Omega}
\newcommand{\orderedCoords}[0]{o}
\newcommand{\mapUpdate}[2]{\mathsf{mapUpdate}_{#1}\ifthenelse{\equal{#2}{}}{}{(#2)}}

\newcommand{\low}{\Lambda}

\newcommand{\partialLevelSet}[2]{\mathsf{levSet}_{#1}(#2)}
\newcommand{\xiSuppMat}[0]{\mathsf{tPtsMat}}
\newcommand{\anchors}[0]{\mathsf{anchors}}
\newcommand{\ancset}[0]{A}

\newcommand{\columns}[0]{{\sf columns}\@\xspace}
\newcommand{\rows}[0]{{\sf rows}\@\xspace}

\newcommand{\sbx}{\mathrm{box}(S)}
\newcommand{\switch}{\mathrm{sw}}
\newcommand{\sep}{\mathrm{sep}}

\newcommand{\XiSp}[0]{S}

\newcommand{\frontier}{{\sf Frontier}\@\xspace}

\newcommand{\length}{\mathrm{length}}

\newcommand{\lift}[0]{\mathrm{lift}}
\newcommand{\blift}[0]{\mathrm{lift}}
\newcommand{\MatLeft}[0]{{D_1}}
\newcommand{\MatRight}[0]{{D_2}}

\newcommand{\PCat}[0]{\category{P}}

\newcommand{\RCat}[0]{\category{R}}
\newcommand{\ZCat}[0]{\category{Z}}
\newcommand{\LCat}[0]{\category{L}}
\newcommand{\Simp}[0]{\category{Simp}}

\newcommand{\An}[0]{A_n}

\newcommand{\pfd}{p.f.d.\@\xspace}

\newcommand{\fl}{\mathrm{fl}}

\usetikzlibrary{calc}
\usetikzlibrary{decorations.pathreplacing}
\usetikzlibrary{positioning,chains,fit,shapes}
\tikzset{>=stealth}

%% file: VRI_Intro.tex
%VRI_Intro.tex

\section{Introduction}

\subsection{Overview}
Topological data analysis (TDA) is a relatively new branch of statistics whose goal is to apply topology to develop tools for studying the coarse-scale, global, non-linear, geometric features of data. Persistent homology, one of the central tools of TDA, provides invariants of data, called barcodes, by associating to the data a filtered topological space $\F$ and then applying standard topological and algebraic techniques to $\F$.  In the last 15 years, persistent homology has been widely applied in the study of scientific data \cite{carlsson2014topological,edelsbrunner2010computational}, and has been the subject of extensive theoretical work \cite{bauer2014induced,chazal2009proximity,chazal2012structure,cohen2007stability,zomorodian2005computing,carlsson2009zigzag}.  

For many data sets of interest, such as point cloud data with noise or non-uniformities in density, a single filtered space is not a rich enough invariant to encode the structure of interest in our data \cite{carlsson2010multiparameter,carlsson2009theory,chazal2011geometric}. This motivates the consideration of multidimensional persistent homology, which in its most basic form associates to the data a topological space simultaneously equipped with two or more filtrations.  Multi-D persistent homology yields algebraic invariants of data far more complex than in the 1-D setting; new methodology is thus required for working with these invariants in practice.

In the TDA community, it is widely appreciated that there is a need for practical data analysis tools for handling multidimensional persistent homology.  However, whereas the community has been quick to develop fast algorithms and good publicly available software for 1-D persistent homology, it has been comparatively slow to extend these to the multi-D setting.  Indeed, to date there is, to the best of our knowledge, no publicly available software which extends the usual persistent homology pipeline for exploratory data analysis to the multidimensional persistence.  

This work seeks to address this gap in the case of 2-D persistence: building on ideas presented in \cite{carlsson2009theory} and \cite{cerri2013betti}, we introduce a practical tool for working with 2-D persistent homology in exploratory data analysis applications, and develop mathematical and algorithmic foundations for this tool.  Our tool can be used in much the same way that 1-D persistent homology is used in TDA, but offers the user significantly more information and flexibility than standard 1-D persistent homology does.  

Our tool, which we call RIVET (the Rank Invariant Visualization and Exploration Tool), is \emph{interactive}--it allows the user to dynamically navigate a two-parameter collection of persistence barcodes derived from a 2-D persistence module.   This is in contrast to previous visualization tools for persistent homology, which have presented static displays.  Because the invariants considered in this paper are larger and more complex than the standard 1-D persistence invariants, it is essential for the user have some nice way of ``browsing" the invariant on the computer screen.  We expect that as TDA moves towards the use of richer invariants in practical applications, interactive visualization paradigms will play an increasingly prominent role in the TDA workflow.

It should be possible to extend our approach for 2-D persistence to 3-D, at a computational cost.  However, there are a number of practical challenges in this, not the least of which is designing and implementing a suitable graphical user interface.  
Since there is already plenty to keep us busy in 2-D, we restrict attention to the 2-D case in this paper.

In the remainder of this section, we review multidimensional persistent homology,  introduce the RIVET visualization paradigm, and provide an overview of RIVET's mathematical and computational underpinnings.  Given the length of this paper, we expect that some readers will be content to limit their first reading to this introduction. 
We invite you to begin this way.

\paragraph{Availability of the RIVET Software}
We plan to make our RIVET software publicly available at \url{http://rivet.online} within the next few months.  In the meantime, a demo of RIVET can be accessed through the website.

\subsection{Multidimensional Filtrations and Persistence Modules}\label{Sec:Multi_D_Persistent_Homology}
We start our introduction to RIVET by defining multidimensional filtrations and persistence modules, and reviewing the standard 1-D persistent homology pipeline for TDA.  
Here and throughout, we freely use basic language from category theory.  An accessible introduction to such language in the context of persistence theory can be found in \cite{bubenik2014metrics}.  

\paragraph{Notation}
For categories, $\C$ and $\D$, let $\D^\C$ denote the category whose objects are functors $\C\to \D$ and whose objects are natural transformations.

For $\C$ a poset category, $\F:\C\to \D$ a functor, and $c\in \obj(c)$, let $\F_c=\F(c)$, and for $c\leq d\in \obj(c)$, let $\F(c,d):\F_c \to \F_d$ denote the image under $\F$ of the unique morphism in $\hom_\C(c,d)$.  
Let $\Simp$ denote the category of simplicial complexes and simplicial maps.  For a fixed field $K$, let $\kvect$ denote the category of $K$-vector spaces and linear maps.

Define a partial order on $\R^n$ by taking \[(a_1,\ldots,a_n)\leq (b_1,\ldots b_n)\] if and only if $a_i\leq b_i$ for all $i$, and let $\RCat^n$ denote the corresponding poset category.\nomenclature[R]{$\RCat^n$}{poset category of $\R^n$}

For $S$ a set, we let $\cupinf{S}$ denote the set $S\cup \{\infty\}$.

\paragraph{Finite Multidimensional Filtrations}

Define an \emph{n-D filtration} to be a functor $\F:\RCat^n\to \Simp$ such that for all $a\leq b$, $\F(a,b):\F_a\to \F_b$ is an inclusion.  We will usually refer to a 2-D filtration as a \emph{bifiltration}.  $n$-D filtrations are the basic topological objects of study in multidimensional persistence.  

In the computational setting, we of course work with filtrations that are specified by a finite amount of data.  Let us now introduce language and notation for such filtrations.  

We say an $n$-D filtration \emph{stabilizes} if there exists $a_0\in \R^n$ such that $\F_a=\F_{a_0}$ whenever $a\geq a_0$.  We write $\F_{\max}=\F_{a_0}$.  We say a simplex $s$ in $\F_{\max}$ \emph{appears a finite number of times} if there is a finite set $A\in \R^n$ such that 
\begin{enumerate}
\item for each $a\in A$, $s\in \F_a$, and
\item for each $b\in \R^n$ with $s\in \F_b$, there is some $a\in A$ with $a\leq b$.  
\end{enumerate}
If $s$ appears a finite number of times, then a minimal such $A$ is unique; we denote it $A(s)$, and call it the set of \emph{grades of appearance} of $s$.  

We say an $n$-D filtration $\F$ is \emph{finite} if 
\begin{enumerate}
\item $\F$ stabilizes,
\item $\F_{\max}$ is finite, and
\item each simplex $s\in \F_{\max}$ appears a finite number of times.  
\end{enumerate}
For $\F$ finite, we define $|\F|$, the \emph{size of $\F$}, by \[|\F|=\sum_{s\in \F_{\max}} |A(s)|.\]
In the computational setting, we can represent a finite $n$-D filtration $\F$ in memory by storing the simplicial complex $\F_{\max}$, along with the set $A(s)$ for each $s\in \F_{\max}$.

\paragraph{Multidimensional Persistence Modules}
Define an \emph{n-D persistence module} to be a functor $M:\RCat^n\to \kvect$.\nomenclature[M]{$M$}{multidimensional persistence module}
We say $M$ is \emph{pointwise finite dimensional (\pfd)} if $\dim(M_a)< \infty$ for all $a\in \R^n$.\nomenclature[pfd]{\pfd}{pointwise finite dimensional}
As we explain in \cref{Sec:MultidimensionalPersistenceModules}, the category $\kvect^{\RCat^n}$ of persistence modules is isomorphic to a category of multi-graded modules over suitable rings.  In view of this, we may define presentations of $n$-D persistence modules; see \cref{Sec:Free_Modules_Presentations} for details.

For $i\geq 0$, let $H_i:\Simp\to \kvect$ denote the $i^{\mathrm{th}}$ simplicial homology functor with coefficients in $K$.\nomenclature[Hi]{$H_i$}{$i^{\mathrm{th}}$ homology functor}
For $\F$ a finite $n$-D filtration, there exists a finite presentation for $H_i\F$, which implies in particular that $H_i\F$ is \pfd; see \cref{Sec:Free_Modules_Presentations} for details on presentations.

\paragraph{Barcodes of 1-D Persistence Modules}
\cite{crawley2012decomposition} shows that \pfd 1-D persistent homology modules decompose in an essentially unique way into indecomposable summands, and that the isomorphism classes of these summands are parameterized by (non-empty) intervals in $\R$.

Thus, we may associate to each \pfd 1-D persistence module $M$ a multiset $\B{M}$ of intervals in $\R$, which records the isomorphism classes of indecomposable summands of $M$.  \label{BarcodeDefRef} We call $\B{M}$ the \emph{barcode} of $M$; in general, we refer to any multiset of intervals as a barcode.

The barcode of a finitely presented 1-D persistence module consists of a finite set of intervals of the form $[a,b)$, with $a\in \R$, $b\in \cupinf{\R}$.  A barcode of this form can be represented as a \emph{persistence diagram}, i.e., a multiset of points in $(a,b)\in\R\times (\cupinf{\R})$ with $a\leq b$.  The right side of \cref{fig:barcode_example} depicts a barcode, together with its corresponding persistence diagram.

\begin{figure}[ht]
\begin{center}
	%point cloud circle
	\begin{tikzpicture}
        	\draw[white!20] (0,0) circle (1.55);
		\draw[white!20] (0,0) circle (2.45);
        
             \foreach [count=\i from 0] \p in {(46:2.3),(54:1.6),(67:2.4),(80:1.9),(99:2.2),(122:2.2),(134:2.1),(137:1.8),(157:2.2),(160:1.7),(174:2.3),(179:1.8),(210:2.0),(232:1.7),(235:2.2),(255:1.7),(271:1.8),(275:2.0),(286:2.1),(294:1.6),(298:2.3),(308:1.8),(323:2.0),(352:1.7)}
			\fill \p circle (1.5pt);
   \end{tikzpicture}
   \hspace{0.5in}
   %barcode and persistence diagram
   \begin{tikzpicture}
       % define styles
	\tikzstyle{h1} = [green!60!black,line width=2.4pt]
	\tikzstyle{axis} = [black!70]
	\tikzstyle{ref} = [black!30]
	
	% parameters
	\def\s{1.8}  % scale for barcode

	% draw h1 bars and dots
	\foreach [count=\i from 0] \a/\b in {1.72/3.02, 0.88/0.92, 0.68/0.78, 0.62/0.72}
	{
        \draw[ref] (-1.2 + 0.18*\i,\a*\s) -- (\a*\s,\a*\s) -- (\a*\s,\b*\s) -- (-1.2 + 0.18*\i,\b*\s);
        \draw[h1] (-1.2 + 0.18*\i,\a*\s) -- (-1.2 + 0.18*\i,\b*\s);
        \fill[green!60!black] (\a*\s,\b*\s) circle (1.8pt);
       }
       
       % draw box
	\draw[axis] (0.6,0.6) rectangle +(5.5,5.5);
	\draw[axis] (0.6,0.6) -- +(5.5,5.5);
	
	% draw axis
	\draw[axis,->] (-0.5,0.6) -- +(0,5.5);
   \end{tikzpicture}
\end{center}
	\caption{A point cloud circle (left) and its $1^{\mathrm{st}}$ persistence barcode (right), obtained via the Vietoris-Rips construction.  A barcode can be visualized either by directly plotting each interval (green bars, oriented vertically) or a via persistence diagram (green dots, right).  The single long interval in the barcode encodes the presence of a cycle in the point cloud.}
	\label{fig:barcode_example}
\end{figure}
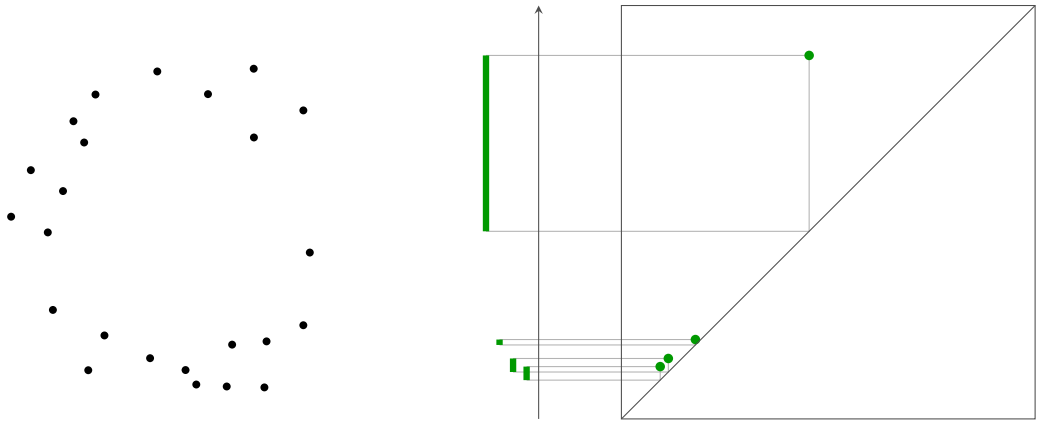

\subsection{Persistence Barcodes of Data}\label{Sec:Intro_Multi_D_PH}
The standard 1-D persistent homology pipeline for data analysis associates a barcode $\Bi{P}$ to a data set $P$, for each $i\geq 0$.  We regard each interval of each barcode $\Bi{P}$ as a topological feature of the data, and we interpret the length of the interval as a measure of robustness of the feature.  The pipeline for constructing $\Bi{P}$ proceeds in three steps:  

\begin{enumerate}
\item We associate to the data set $P$ a finite 1-D filtration $\F(P)$.  

\item We apply $H_i$ to obtain a finitely presented 1-D persistence module $H_i \F(P)$. 

\item We take $\Bi{P}=\B{H_i\F(P)}$.
\end{enumerate}

This pipeline for the construction of barcodes of data is quite flexible: in principle, we can work with a data set $P$ of any kind, and may consider any number of choices of the filtration $\F(P)$.  Different choices of $\F(P)$ topologically encode different aspects of the structure of our data.

The barcodes $\Bi{P}$ are readily computed in practice; see \cite{zomorodian2005computing} or \cref{sec:ComputationPersistenceBarcodes} for details.  

\paragraph{Persistence Barcodes of Finite Metric Spaces}
Here is one standard choice of $P$ and $\F(P)$ in the TDA pipeline: Let $P$ be a finite metric space, and let $\F(P)=\Rips(P)$, the \emph{Vietoris-Rips filtration} of $P$, be defined as follows.
For $t\geq 0$, $\Rips(P)_t$ is the maximal simplicial complex with 0-skeleton $P$, and 1-simplices the pairs $[p,q]$ with $d(p,q)\leq 2t$; for $t<0$, $\Rips(P)_t$ is the empty simplicial complex.

Informally, the long intervals in the barcodes $\Bi{P}$ correspond to cycles in the data; see \cref{fig:barcode_example} for an illustration.
\paragraph{Stability}
The following well-known result shows that these barcode invariants of finite metric spaces are robust to certain perturbations of the data:

\begin{theorem}[Stability of Barcodes of Finite Metric Spaces \cite{chazal2009gromov}]\label{Thm:GH_Stability} 
For all finite metric spaces $P$, $Q$ and $i\geq 0$,
\[d_{GH}(P,Q)\geq d_B(\Bi{P},\Bi{Q}),\]
where $d_{GH}$ denotes Gromov-Hausdorff distance and $d_B$ denotes the bottleneck distance on barcodes; see \textup{\cite{chazal2009gromov}} for definitions.  
\end{theorem}

See also \cite{chazal2014persistence} for a generalization of this theorem to compact metric spaces.  

\subsection{Multidimensional Persistent Homology}\label{Sec:Multi_D_PH_Intro}
In many cases of interest, a 1-D filtration is not sufficient to capture the structure of interest in our data.  In such cases, we are naturally led to associate to our data an $n$-D filtration, for some $n\geq 2$.  As in the 1-D case, applying homology with field coefficients to an $n$-D filtration yields an $n$-D persistence module.  The question then arises of whether we can associate to this $n$-D filtration an $n$-dimensional generalization of a barcode in a way that is useful for data analysis.

In what follows, we motivate the study of multidimensional persistence by describing one natural way that bifiltrations arise in the study of finite metric spaces; other ways that bifiltrations arise in TDA applications are discussed, for example,  in \cite{carlsson2009theory}, \cite{lesnick2014theory}, and \cite{chazal2011geometric}.  We then discuss the algebraic difficulties involved in defining a multidimensional generalization of barcode.
\paragraph{Bifiltrations of Finite Metric Spaces}
In spite of \cref{Thm:GH_Stability}, which tells us that the barcodes $\Bi{P}=\Bi{H_i\Rips(P)}$ of a finite metric space $P$ are well behaved in a certain sense, these invariants have a couple of important limitations.  First, they are highly unstable to the addition and removal of outliers; see \cref{fig:barcode_problems} for an illustration.  Second, and relatedly, when $P$ exhibits non-uniformities in density, the barcodes $\Bi{P}$ can be insensitive to interesting structure in the high density regions of $P$; see \cref{fig:barcode_problems}.

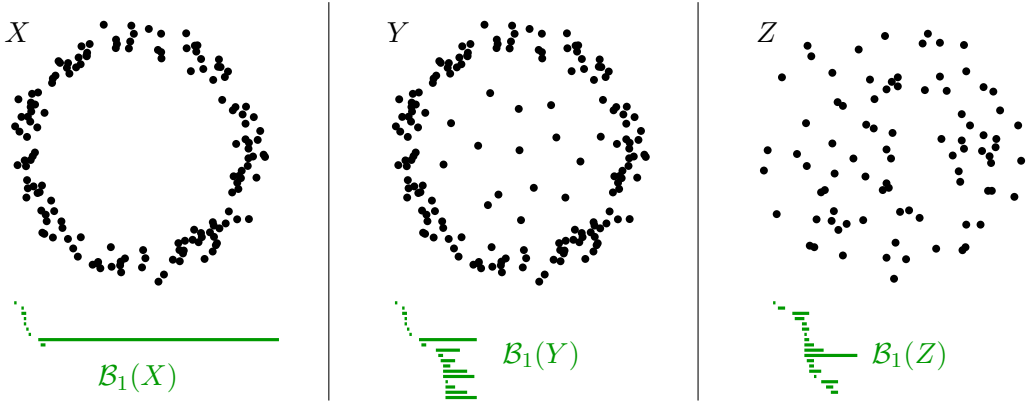
\begin{figure}[ht]
	\begin{center}
		\begin{tikzpicture}
			\tikzstyle{h1} = [green!60!black,line width=1.2pt]  % style for barcodes
			\def\s{3.2}  % scale for barcodes
			\def\v{-0.07}  % vertical separation of bars
			\def\yoff{-2}  % y-offset for barcodes
			\def\xoff{1.87}  % y-offset for barcodes
			\matrix[column sep=.25in] {
				%point-cloud circle with outliers
				\node at (-1.6,1.6) {$X$};
				\foreach \p in {(257:1.56), (126:1.40), (141:1.47), (62:1.61), (312:1.43), (102:1.42), (294:1.47), (131:1.61), (117:1.49), (95:1.39), (68:1.71), (44:1.39), (15:1.48), (282:1.66), (130:1.47), (139:1.50), (310:1.61), (220:1.66), (263:1.42), (156:1.74), (327:1.44), (5:1.31), (240:1.60), (272:1.53), (357:1.47), (211:1.50), (25:1.43), (101:1.47), (157:1.55), (52:1.41), (346:1.31), (49:1.64), (291:1.55), (302:1.67), (250:1.55), (46:1.58), (92:1.58), (159:1.52), (304:1.44), (57:1.29), (254:1.31), (295:1.44), (65:1.53), (170:1.60), (182:1.40), (61:1.63), (202:1.60), (93:1.69), (137:1.50), (236:1.54),
						(198:1.54), (155:1.51), (321:1.50), (252:1.62), (326:1.59), (306:1.27), (273:1.30), (221:1.65), (340:1.42), (261:1.52), (317:1.48), (32:1.32), (343:1.38), (85:1.65), (123:1.40), (202:1.41), (166:1.38), (151:1.61), (180:1.37), (157:1.35), (248:1.61), (187:1.50), (155:1.57), (79:1.61), (149:1.55), (164:1.49), (298:1.33), (162:1.54), (358:1.53), (317:1.36), (198:1.46), (316:1.40), (262:1.61), (23:1.56), (118:1.46), (294:1.32), (279:1.74), (215:1.36), (184:1.43), (225:1.60), (217:1.52), (353:1.47), (184:1.51), (2:1.47), (231:1.48), (186:1.61), (10:1.65), (118:1.60), (114:1.58), (358:1.69),
						(288:1.50), (209:1.32), (347:1.38), (131:1.58), (131:1.39), (7:1.44), (225:1.42), (163:1.62), (24:1.38), (348:1.61), (312:1.51), (101:1.53), (250:1.57), (60:1.55), (350:1.40), (41:1.53), (359:1.67), (42:1.61), (6:1.60), (311:1.60), (307:1.65), (329:1.72), (307:1.41), (20:1.48), (5:1.48), (172:1.49), (346:1.41), (199:1.35), (194:1.52), (76:1.29), (261:1.47), (214:1.48), (94:1.48), (274:1.41), (122:1.49), (202:1.44), (308:1.36), (128:1.51), (236:1.36), (105:1.76), (337:1.36), (17:1.61), (302:1.37), (56:1.50), (47:1.56), (77:1.37), (168:1.67), (118:1.45), (222:1.37), (79:1.39), (293:1.41)} 
					\fill \p circle (1.5pt);
					
				%H_1 barcode
				\node[green!60!black] at (0,-3) {$\B{}_1(X)$};
				\foreach [count=\i from 0] \a/\b in {0.07/0.08, 0.10/0.11, 0.10/0.12, 0.11/0.12, 0.11/0.12, 0.12/0.13, 0.13/0.14, 0.17/1.17, 0.18/0.20}
					\draw[h1] (\a*\s - \xoff, \i*\v + \yoff) -- (\b*\s - \xoff, \i*\v + \yoff);
			&
				\draw (0,-3.3)--(0,2);
			&
				% dense circle with low-density interior
				\node at (-1.6,1.6) {$Y$};
				\foreach \p in {(257:1.56), (126:1.40), (141:1.47), (62:1.61), (312:1.43), (102:1.42), (294:1.47), (131:1.61), (117:1.49), (95:1.39), (68:1.71), (44:1.39), (15:1.48), (282:1.66), (130:1.47), (139:1.50), (310:1.61), (220:1.66), (263:1.42), (156:1.74), (327:1.44), (5:1.31), (240:1.60), (272:1.53), (357:1.47), (211:1.50), (25:1.43), (101:1.47), (157:1.55), (52:1.41), (346:1.31), (49:1.64), (291:1.55), (302:1.67), (250:1.55), (46:1.58), (92:1.58), (159:1.52), (304:1.44), (57:1.29), (254:1.31), (295:1.44), (65:1.53), (170:1.60), (182:1.40), (61:1.63), (202:1.60), (93:1.69), (137:1.50), (236:1.54),
						(198:1.54), (155:1.51), (321:1.50), (252:1.62), (326:1.59), (306:1.27), (273:1.30), (221:1.65), (340:1.42), (261:1.52), (317:1.48), (32:1.32), (343:1.38), (85:1.65), (123:1.40), (202:1.41), (166:1.38), (151:1.61), (180:1.37), (157:1.35), (248:1.61), (187:1.50), (155:1.57), (79:1.61), (149:1.55), (164:1.49), (298:1.33), (162:1.54), (358:1.53), (317:1.36), (198:1.46), (316:1.40), (262:1.61), (23:1.56), (118:1.46), (294:1.32), (279:1.74), (215:1.36), (184:1.43), (225:1.60), (217:1.52), (353:1.47), (184:1.51), (2:1.47), (231:1.48), (186:1.61), (10:1.65), (118:1.60), (114:1.58), (358:1.69),
						(288:1.50), (209:1.32), (347:1.38), (131:1.58), (131:1.39), (7:1.44), (225:1.42), (163:1.62), (24:1.38), (348:1.61), (312:1.51), (101:1.53), (250:1.57), (60:1.55), (350:1.40), (41:1.53), (359:1.67), (42:1.61), (6:1.60), (311:1.60), (307:1.65), (329:1.72), (307:1.41), (20:1.48), (5:1.48), (172:1.49), (346:1.41), (199:1.35), (194:1.52), (76:1.29), (261:1.47), (214:1.48), (94:1.48), (274:1.41), (122:1.49), (202:1.44), (308:1.36), (128:1.51), (236:1.36), (105:1.76), (337:1.36), (17:1.61), (302:1.37), (56:1.50), (47:1.56), (77:1.37), (168:1.67), (118:1.45), (222:1.37), (79:1.39), (293:1.41), 
						(0.00,0.58), (0.01,0.03), (0.43,0.63), (0.5,0.2), (1.09,0.21), (0.82,-0.12), (0.62,-0.60), (0.03,-0.90), (0.2,-0.47), (-1.0,-0.16),(-0.9,0.39), (-0.54,0.09),(-0.38,0.79), (-0.41,-0.70), (-0.27,-0.56)} 
					\fill \p circle (1.5pt);
				%H_1 barcode
				\node[green!60!black] at (0.3,-2.7) {$\B{}_1(Y)$};
				\foreach [count=\i from 0] \a/\b in {0.07/0.08, 0.10/0.11, 0.10/0.12, 0.11/0.12, 0.11/0.12, 0.12/0.13, 0.13/0.14, 0.17/0.41, 0.18/0.20, 0.24/0.34, 0.25/0.27, 0.26/0.32, 0.27/0.30, 0.27/0.37, 0.27/0.40, 0.28/0.29, 0.28/0.32, 0.28/0.37, 0.28/0.41}
					\draw[h1] (\a*\s - \xoff, \i*\v + \yoff) -- (\b*\s - \xoff, \i*\v + \yoff);
			&
				\draw (0,-3.3)--(0,2);
			&
				%low-density disk
				\node at (-1.6,1.6) {$Z$};
				\foreach \p in {(1.50,0.56), (0.72,0.18), (-0.66,0.67), (-0.32,-0.90), (0.06,0.94), (0.02,-0.47), (0.05,-0.08), (-0.97,-1.26), (-0.71,-0.78), (1.10,0.36), (-1.07,-0.18), (-0.54,-0.95), (-0.65,-1.03), (0.87,-0.04), (1.28,0.87), (-0.59,0.62), (-1.59,0.03), (0.99,0.51), (0.99,0.16), (-0.70,-0.28), (-0.23,0.97), (-0.03,1.55), (-1.03,-1.24), (-0.88,-0.53), (-0.73,0.12), (1.05,1.08), (1.36,-0.05), (-0.01,0.84), (-0.45,-0.08), (1.74,0.36), (0.47,1.45), (-0.06,-0.50), (0.43,-0.87), (0.97,-0.39), (-0.01,0.04), (1.28,0.23), (-0.59,-1.37), (0.65,-1.28), (0.14,1.00), (-0.38,1.37), (1.24,-0.96), (1.02,1.43), 
						(-0.01,-1.40), (-0.01,-0.41), (0.09,-1.68), (1.02,0.05), (1.34,-0.51), (-1.06,1.42), (0.50,0.83), (0.79,0.82), (-0.74,-0.91), (1.26,0.02), (1.39,-0.51), (-1.20,-0.02), (-0.73,1.06), (0.29,-0.77), (0.07,0.26), (-1.08,0.39), (1.44,0.07), (1.05,-0.96), (1.45,0.18), (1.04,-1.26), (1.78,-0.11), (0.70,1.01), (-0.17,0.71), (0.17,-0.79), (0.95,-0.25), (0.11,-1.20), (0.54,1.58), (1.00,-1.30), (-0.18,0.32), (-0.94,-0.89), (-0.83,-0.50), (-0.28,0.37), (1.38,-0.14), (-0.29,-0.32), (1.69,-0.59), (-1.64,-0.24), (0.14,0.88), (-1.47,-0.82), (0.70,1.46), (0.25,-1.42), (-1.01,1.28), (1.41,0.81), (-1.4,1.0)}
					\fill \p circle (1.5pt);
				%H_1 barcode
				\node[green!60!black] at (0.3,-2.7) {$\B{}_1(Z)$};
				\foreach [count=\i from 0] \a/\b in {0.11/0.12, 0.13/0.16, 0.19/0.26, 0.20/0.24, 0.23/0.25, 0.23/0.26, 0.24/0.25, 0.24/0.26, 0.24/0.28, 0.24/0.32, 0.24/0.46, 0.25/0.28, 0.26/0.28, 0.26/0.31, 0.28/0.29, 0.31/0.38, 0.33/0.36, 0.35/0.38}
					\draw[h1] (\a*\s - \xoff, \i*\v + \yoff) -- (\b*\s - \xoff, \i*\v + \yoff);
			\\};
		\end{tikzpicture}
	\end{center}
	\caption{Barcodes are unstable with respect to addition of outliers, and can be insensitive to interesting structure in high density regions of the data.  Thus, though the point clouds $X$ and $Y$ share a densely sampled circle and differ only by the addition of a few low-density outliers, $\protect\B{}_1(X)$ and $\protect\B{}_1(Y)$ are quite different from one another: $\protect\B{}_1(X)$ has a long interval not appearing in $\protect\B{}_1(Y)$.  In contrast, the point cloud $Z$ contains no densely sampled circle, but the longest intervals of $\protect\B{}_1(Y)$ and $\protect\B{}_1(Z)$ are of similar length.}	
	\label{fig:barcode_problems}
\end{figure}

To address these issues, \cite{carlsson2009theory} proposed that we associate a bifiltration to $P$, and study the persistent homology of this bifiltration.  We describe here both the proposal of  \cite{carlsson2009theory}, which depends on a choice of bandwidth parameter, and a simple parameter-free variant which, to the best of our knowledge, has not appeared elsewhere.

The construction of the bifiltration proposed in \cite{carlsson2009theory} depends on a choice of \emph{codensity function} $\gamma:P\to \R$, a function on $P$ whose value is high at dense points and low at outliers \cite{carlsson2009theory,wasserman2004all}.  For example, \cite{carlsson2009theory} proposes to take $\gamma$ to be a $K$-nearest neighbors density function; in general, the choice of a density function depends on a choice of a bandwidth parameter $K$.

Given $\gamma$, we may define a 2-dimensional filtration $\Rips(\gamma)$ by taking
\[\Rips(\gamma)_{(a,b)}=\Rips(\gamma^{-1}(-\infty,a])_b.\]
If $a\geq a'$ and $b\leq b'$, then $\Rips(\gamma)_{(a,b)}\subseteq \Rips(\gamma)_{(a',b')}$.  Thus, the collection of all such simplicial complexes, together with these inclusions, yields a functor \[\Rips(\gamma):\RCat^{\op}\times \RCat\to \Simp.\]  $\RCat^{\op}\times \RCat$ is naturally isomorphic to $\R^2$---for example there is an isomorphism sending each object $(a,b)$ to $(-a,b)$---so upon the identification of these two categories, we may regard $\Rips(\gamma)$ as a bifiltration.  Note that the definition of $\Rips(\gamma)$ in fact makes sense for \emph{any} function $\gamma:P\to \R$.  As discussed in \cite{carlsson2009theory}, there are numerous possibilities for interesting choices of $\gamma$ in our study of data, aside from density estimates.

To obtain a parameter-free variant of $\Rips(\gamma)$, for $(a,b)\in \R^2$ we first define the graph $G_{a,b}$ to be the subgraph of the 1-skeleton of $\Rips(P)_{b}$ consisting of vertices whose degree is at least $a$.  We then define $\BRips(P)_{(a,b)}$ to be the maximal simplicial complex with 1-skeleton $G_{a,b}$.  As above, upon the identification of $\RCat^{\op}\times \RCat$ with $\RCat^2$, the collection of simplicial complexes $\{\BRips(P)_{(a,b)}\}_{(a,b)\in \R^2}$ defines a bifiltration $\BRips(P)$.

$H_i \BRips(P)$ is a richer algebraic invariant than $H_i\Rips(P)$, and in particular is sensitive to interesting structure in high density regions of $P$ in a way that $H_i \Rips(P)$ is not.  

\originalparagraph{Barcodes of Multi-D Persistence Modules?}
We now explain the algebraic difficulties with defining the barcode of a $n$-D persistence module for $n>1$, closely following a discussion in \cite{lesnick2014theory}.

As for the case $n=1$, for $n> 1$, finitely presented $n$-D persistence modules also decompose in an essentially unique way into indecomposable summands; this follows easily from a standard formulation of the Krull-Schmidt theorem \cite{atiyah1956krull}.  However, it is a consequence of standard quiver theory results, as described for example in \cite{derksen2005quiver}, that the set of isomorphism classes of indecomposable multi-D persistence modules is, in contrast to the 1-D case, extremely complicated.\footnote{For example, there is a fully faithful functor from the category of representations of the wild quiver $E_9$ to $\kvect^{\RCat^2}$, which maps indecomposables to indecomposables; see also \cite{carlsson2009theory} for a study of the possible isomorphism types of a multidimensional persistence module.}  In particular, for $n\geq 2$ the dimension of a vector space in a finitely presented indecomposable $n$-D persistence module can be arbitrarily large.  Thus, while in principle we could define the barcode of a multi-D persistence module to be its multiset of isomorphism classes of indecomposables, as in the 1-D case, for $n>1$ this invariant will typically not be useful for data visualization and exploration in the way that the 1-D barcode is.

In general, it seems that for the purposes of TDA, there is no entirely satisfactory way of defining the barcode of an $n$-D persistence module for $n>1$, even if we consider invariants which are incomplete (i.e., invariants which can take the same value on two non-isomorphic modules.)

\paragraph{Three Simple Invariants of a Multidimensional Persistence Module}
Nevertheless, it is possible to define simple, useful, computable invariants of a multidimensional persistence module.  Our tool RIVET computes and visualizes three such invariants of a 2-D persistence module $M$:
\begin{enumerate}
\item The \emph{dimension function of $M$}, i.e., the function which maps $a\in \R^2$ to $\dim M_a$.
\item The \emph{fibered barcode} of $M$, i.e., the collection of barcodes of 1-D affine slices of $M$.
\item The (multigraded) Betti numbers of $M$. 
\end{enumerate}
The dimension function of $M$ is a simple, intuitive, and easily visualized invariant, but is unstable and provides no information about persistent features in $M$.  The next two subsections introduce the fibered barcode and the multigraded Betti numbers.

\subsection{The Rank Invariant and Fibered Barcodes}\label{sec:rank_inv_fibered_barcodes}

\paragraph{The Rank Invariant}
For $n\geq 1$, let $\H^n\subseteq \R^n\times \R^n$ denote the set of pairs $(s,t)$ with $s\leq t$, let $\Nbb$ denote the non-negative integers,\nomenclature[N]{$\Nbb$}{non-negative integers} and let $M$ be an $n$-D persistence module.  Following \cite{carlsson2009theory}, we define \[\rk(M):\H^n\to \Nbb,\] the \emph{rank invariant of $M$}, by $\rk(M)(a,b)=\rank\, M(a,b)$.\nomenclature[rk]{$\rk(M)$}{rank invariant of $M$}
Using the structure theorem for 1-D persistence modules \cite{crawley2012decomposition}, it's easy to check that for $M$ a \pfd 1-D persistence module $M$, $\rk(M)$ and $\B{M}$ determine each other \cite{carlsson2009theory}.

For $M$ an $n$-D persistence module, $n\geq 2$, $\rk(M)$ does not encode the isomorphism class of $M$; see example \cref{Ex:Rank_Invariant_Not_Complete}.  Nevertheless, the rank invariant does capture interesting ``first order information" about the structure of a persistence module.

\cite{cerri2013betti} observed that if $M$ is a \pfd $n$-D persistence module, $\rk(M)$ carries the same data as a $(2n-2)$-parameter family of barcodes, each obtained by restricting $M$ to an affine line in $\R^n$; we call this parameterized family of barcodes the \emph{fibered barcode} of $M$.  In particular, if $M$ is a 2-D persistence module, the fibered barcode of $M$ is a 2-parameter family of barcodes.  
In what follows, we give the definition of a fibered barcode of a 2-D persistence module.

\paragraph{The Space of Affine Lines in 2-D with Non-Negative Slope}
Let $\bar \L$ denote the collection of all (unparameterized) lines in $\R^2$ with non-negative (possibly infinite) slope, and let $\L\subset \bar \L$ denote the collection of all (unparameterized) lines in $\R^2$ with non-negative, finite slope. \nomenclature[L]{$\bar \L$}{space of affine lines in $\R^2$ with non-negative slope} \nomenclature[L]{$\L$}{space of affine lines in $\R^2$ with non-negative, finite slope} 

Note that $\bar\L$ is a submanifold (with boundary) of an affine Grassmannian of dimension 2.   With the induced topology, $\bar\L$ is homeomorphic to $[0,1]\times \R$.  In this sense, it is appropriate to think of $\bar \L$ as a 2-parameter family of lines.  

A standard point-line duality, to be described in \cref{sec:arr_def}, provides an identification between $\L$ and the half-plane $[0,\infty)\times \R$; we will make extensive use of this point-line duality later in the paper.  The duality does not extend to $\bar \L$ (i.e., to vertical lines) in a natural way.

\paragraph{Definition of the Fibered Barcode}
For $L\in \bar \L$, let $\LCat$ denote the associated full subcategory of $\RCat^2$.  The inclusion $L\hookrightarrow \R^2$ induces a functor $i_L:\LCat\to \RCat^2$.  For $M$ a \pfd 2-D persistence module, we define $M^L=M\circ i_L$.  

Define an \emph{interval} $I$ in $L$ to be a non-empty subset of $L$ such that whenever $a<b<c\in L$ and $a,c\in I$, we also have that $b\in I$.  As $\LCat$ is isomorphic to $\RCat$, the structure theorem for \pfd 1-D persistence modules yields a definition of the barcode $\B{M^L}$ as a collection of intervals in $L$.

We define $\B{M}$, the \emph{fibered barcode of $M$}, to be the map which sends each line $L\in \bar\L$ to the barcode $\B{M^L}$.  \nomenclature[B]{$\B{M}$}{barcode of a 1-D persistence module, page  \pageref{BarcodeDefRef}; fibered barcode of a 2-D persistence module}  

\begin{proposition}[\cite{cerri2013betti}]\label{Rank_Vs_Fibered_Barcode}
$\B{M}$ and $\rk(M)$ determine each other.
\end{proposition}

\begin{proof}
For each $a\leq b\in \R^2$, there exists a unique line $L\in \bar \L$ such that $a,b\in L$.  Clearly then, $\rk(M)$ and the collection of 1-D rank invariants $\{\rk(M^L)\mid L\in \bar \L\}$ determine each other.  As noted above, for $N$ a \pfd 1-D persistence module, $\B{N}$ and $\rk(N)$ determine each other.  The result follows.
\end{proof}

\paragraph{Stability of the Fibered Barcode}
Adapting arguments introduced in \cite{cerri2013betti} and \cite{cerri2011new}, a recent note by Claudia Landi \cite{landi2014rank} establishes that $\B{M}$ is \emph{stable} in two senses.   The results of \cite{landi2014rank} in fact hold for persistent homology modules of arbitrary dimension.  However, in this paper we are primarily interested in the 2-D  case.

In 2-D, Landi's first stability result says that for $L$ a line which is neither horizontal nor vertical, the map $M\to \B{M^L}$ is Lipschitz continuous with respect to the interleaving distance on 2-D persistence modules and the \emph{bottleneck distance} on barcodes.  The Lipschitz constant $c$ depends on the slope of $L$; $c=1$ when the slope of $L$ is 1, and $c$ grows larger as the slope of $L$ deviates from 1,  tending towards infinity as the slope of $L$ approaches 0 or $\infty$.  We refer to this stability result as the \emph{external stability} of $\B{M}$.  

 \cite{landi2014rank} also presents an \emph{internal} stability result for $\B{M}$ which tells us that when $M$ is finitely presented, the map $L\to\B{M^L}$ is continuous, in a suitable sense, at lines $L$ which are neither horizontal nor vertical.  In fact the result says something stronger, which put loosely, is that the closer the slope of the line $L$ is to 1, the more stable $\B{M^L}$ is to perturbations of $L$.  
 
In sum, the stability results of \cite{landi2014rank} tell us that for $M$ a 2-D persistence module and $L$ a line which is neither horizontal or vertical, the barcode $\B{M^L}$ is robust to perturbations both of $M$ and $L$; the more diagonal $L$ is, the more robust $\B{M^L}$ is.

\subsection{Multigraded Betti Numbers}
We next briefly introduce (multigraded) Betti numbers (called bigraded Betti numbers in the case of 2-D persistence); see \cref{Sec:Betti_Numbers} for a precise definition and examples.

For $M$ a finitely presented $n$-D persistence module and $i\geq 0$, the $i^{\rm{th}}$ (graded) Betti number of $M$ is a function $\xi_i(M):\R^n\to \Nbb$. \nomenclature[xi]{$\xi_i(M)$}{$i^{\rm{th}}$ (graded) Betti number of $M$}
It follows from the \emph{Hilbert Basis Theorem}, a classical theorem of commutative algebra, that $\xi_i(M)$ is identically $0$ for $i>n$, so $\xi_i(M)$ is only of interest for $i\leq n$.  

We will be especially interested in $\xi_0(M)$ and $\xi_1(M)$ in this paper.  For $a\in \R^n$, $\xi_0(M)(a)$ and $\xi_1(M)(a)$ are the number of generators and relations, respectively, at index $a$ in a minimal presentation for $M$; see \cref{Sec:MinPresentations}.   $\xi_2(M)$ has an analogous interpretation in terms of a minimal resolution of $M$.

Neither of the invariants $\B{M}$ nor $\{\xi_i(M) \mid i=0,1,\ldots,n\}$ determines the other, but the invariants are intimately connected.  This connection in the case that $n=2$ plays a central role in the present work.

One of the main mathematical contributions of this project is a fast algorithm for computing the multigraded Betti numbers of a 2-D persistence module $M$; our algorithm is described in the companion paper \cite{lesnick2014betti}.
\subsection{The RIVET Visualization Paradigm}\label{Sec:Intro_Visualization}

\paragraph{Overview}
We propose to use the fibered barcode in exploratory data analysis in much the same way that 1-D barcodes are typically used.  In particular, this requires us to have a good way of visualizing the fibered barcode.  Though discretizations of fibered barcodes have been used in shape matching applications \cite{cerri2014comparing}, to the best of our knowledge, there is no prior work on visualization of fibered barcodes.

This work introduces a paradigm, called RIVET, for the interactive visualization of the fibered barcode of a 2-D persistence module, and presents an efficient computational framework for implementing this paradigm.  Our paradigm also provides for the visualization of the dimension function and bigraded Betti numbers of the module.  The visualizations of the three invariants complement each other, and work in concert: Our visualizations of the dimension function and bigraded Betti numbers provide a coarse, global view of the structure of the persistence module, while our visualization of the fibered barcodes, which focuses on a single barcode at a time, provides a sharper but more local view.  

We now give a brief description of our RIVET visualization paradigm.  Additional details are in \cref{Sec:Interface_Details}.

Given a 2-D persistence module $M$, RIVET allows the user to interactively select a line $L\in \R^2$ via a graphical interface; the software then displays the barcode $\B{M^L}$.  As the user moves the line $L$ by clicking and dragging, the displayed barcode is updated in real time.  

The RIVET interface consists of two main windows, the \emph{Line Selection Window} (left) and the \emph{Persistence Diagram Window} (right).  \cref{fig:RIVET_screenshots} shows screenshots of RIVET for a single choice of $M$ and four different lines $L$.

\begin{sidewaysfigure}%[th]
	\begin{center}
		\includegraphics[width=3.7in]{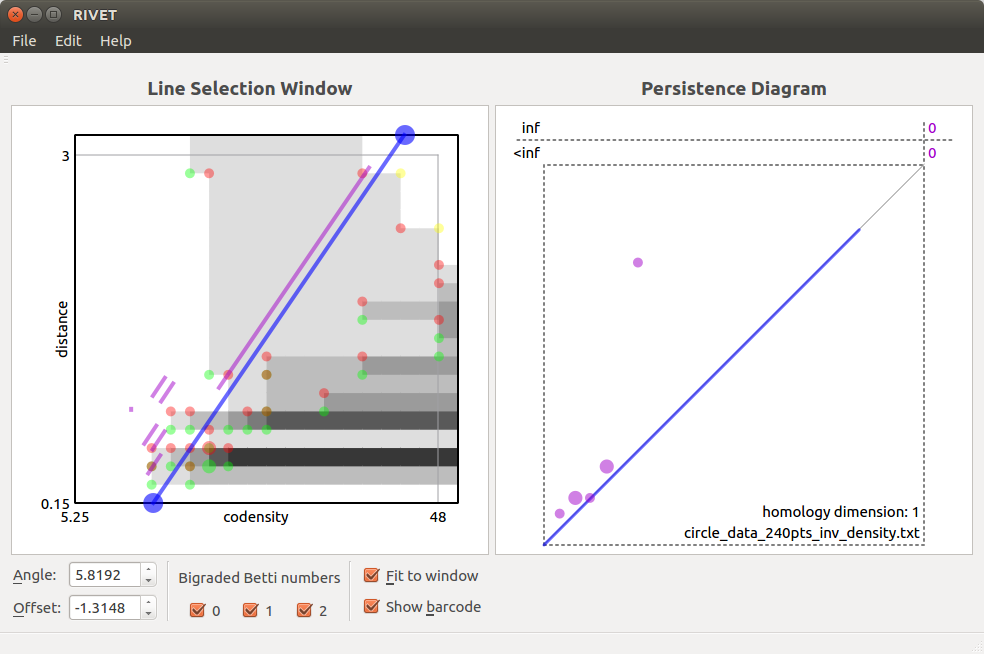} \hspace{0.2in}
		\includegraphics[width=3.7in]{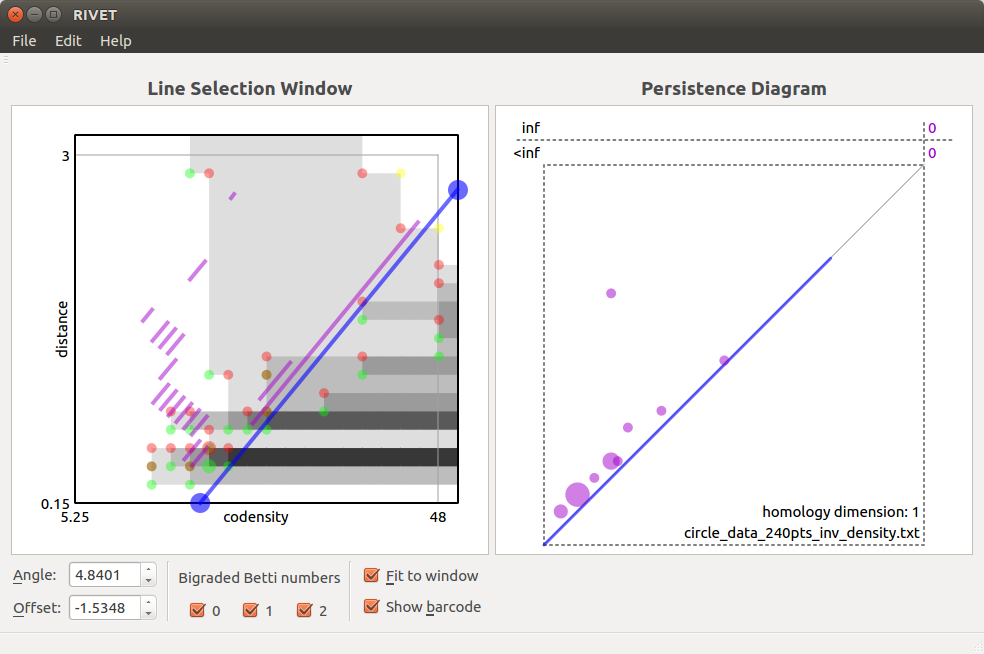}
		
		\vspace{0.2in}
		
		\includegraphics[width=3.7in]{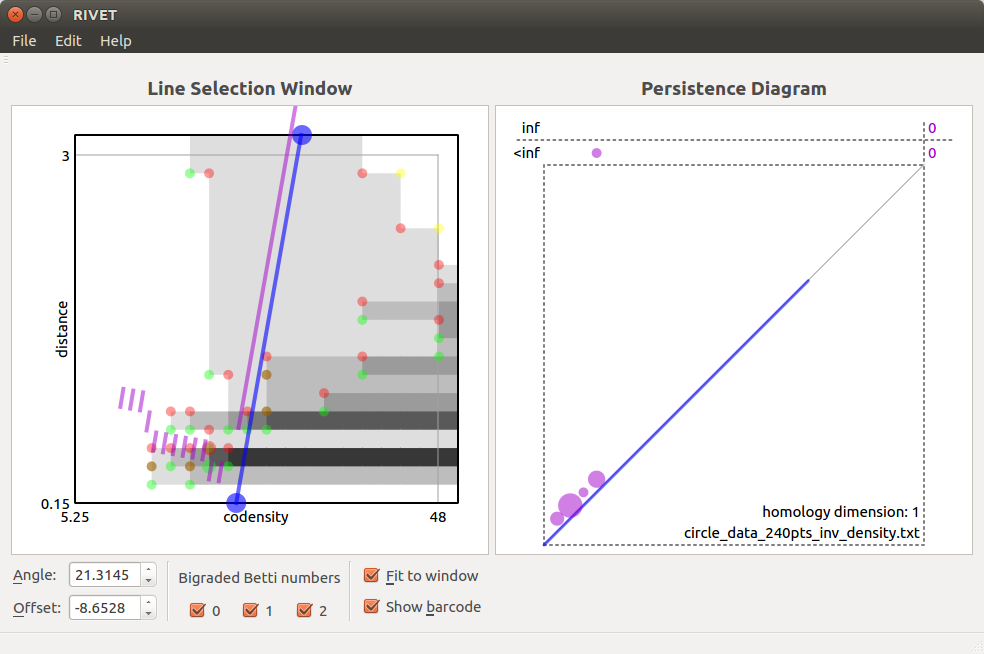} \hspace{0.2in}
		\includegraphics[width=3.7in]{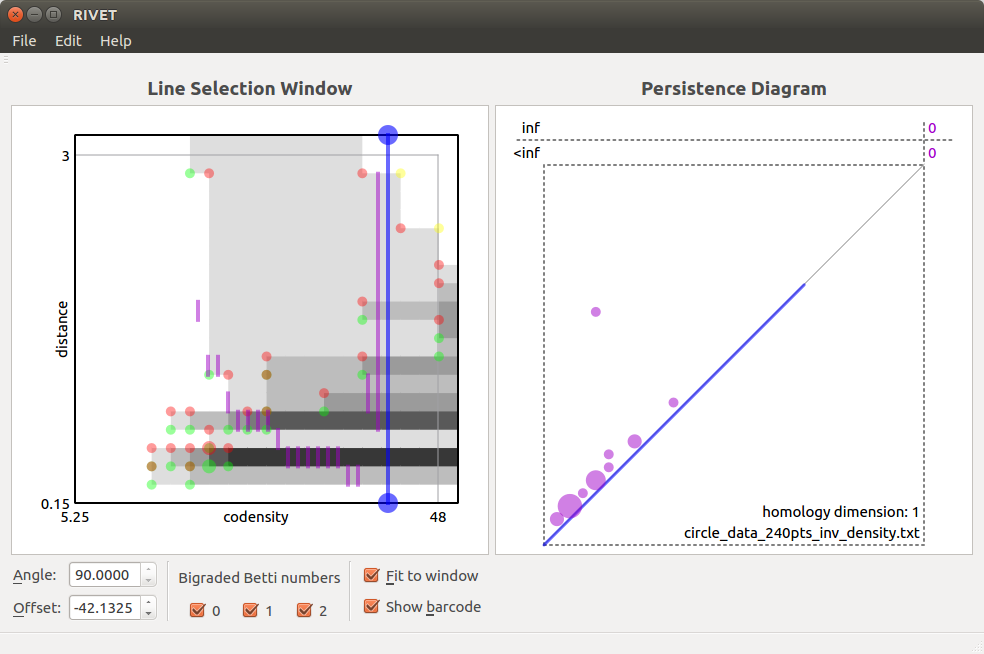}
    \end{center}
	\caption{Screenshots of RIVET for a single choice of 2-D persistence module $M$ and four different lines $L$.  RIVET provides visualizations of the dimension of each vector space in $M$ (greyscale shading); the $0^{\text{th}}$, $1^{\text{st}}$, and $2^{\text{nd}}$ bi-graded Betti numbers of $M$ (green, red, and yellow dots); and the barcodes of the 1-D slices $M^L$, for each $L$ (in purple).}
    \label{fig:RIVET_screenshots}
\end{sidewaysfigure}

\paragraph{The Line Selection Window}
For a given finitely presented persistence module $M$, the Line Selection Window plots a rectangle in $\R^2$ containing the union of the supports of functions $\xi_i(M)$, $i\in \{0,1,2\}$.  

The greyscale shading at a point $a$ in this rectangle represents $\dim M_a$:  $a$ is unshaded when $\dim M_a=0$, and larger $\dim M_a$ corresponds to darker shading.  Scrolling the mouse over $a$ brings up a popup box which gives the precise value of $\dim M_a$.

Points in the supports of $\xi_0(M)$, $\xi_1(M)$, and $\xi_2(M)$ are marked with green, red, and yellow dots, respectively.  
The area of each dot is proportional to the corresponding function value.  The dots are translucent, so that, for example, overlaid red and green dots appear brown on their intersection.  This allows the user to read the values of the Betti numbers at points which are in the support of more than one of the functions.

The line selection window contains a blue line of non-negative slope, with endpoints on the boundary of the displayed region of $\R^2$.  This line represents a choice of $L \in \bar \L$.  The intervals in the barcode $\B{M^L}$ are displayed in purple, offset from the line $L$ in the perpendicular direction.  

\paragraph{The Persistence Diagram Window}
In the Persistence Diagram Window (right), a persistence diagram representation of $\B{M^L}$ is displayed.  To represent $\B{M^L}$ via a persistence diagram, we need to choose an order-preserving parameterization $\gamma_L:\R\to L$ of $L$, which we regard as a functor; we then display $\B{M^L\circ \gamma_L}$.  Our choice of the parameterizations $\gamma_L$ is described in \cref{Sec:Interface_Details}.

As with the multi-graded Betti numbers, the multiplicity of a point in the persistence diagram is indicated by the area of the corresponding dot. 

\paragraph{Interactivity}
The user can click and drag the blue line in the left window with the mouse, thereby changing the choice of $L \in \bar \L$.  Clicking the blue line away from its endpoints and dragging moves the line in the direction perpendicular to its slope, while keeping the slope constant.  The clicking and dragging an endpoint of the line moves that endpoint while keeping the other fixed; this allows the user to change the slope of the line.

As the line moves, the both the interval representation of $\B{M^L}$ in the left window and its persistence diagram representation in the right window are updated in real time.

\subsection{Querying the Fibered Barcode}\label{Sec:Computational_Underpinnings}
Our algorithm for fast computation of Betti numbers of 2-D persistence modules, described in \cite{lesnick2014betti}, performs an efficient computation of the dimension function of $M$ as a subroutine.  Thus, in explaining the computational underpinnings of RIVET, we will focus on the RIVET's interactive visualization of the fibered barcode.  

Because our visualization paradigm needs to update the plot of $\B{M^L}$ in real time as we move the line $L$, it must be able to very quickly access $\B{M^L}$ for any choice of $L\in \bar \L$. In this paper, we introduce an efficient data structure $\S(M)$, the \emph{augmented arrangement} of $M$, on which we can perform fast queries to determine $\B{M^L}$, for $L\in \bar \L$. \nomenclature[AAM]{$\S(M)$}{augmented arrangement of $M$}
We present a theorem which guarantees that our query procedure correctly recovers $\B{M^L}$, and describe an efficient algorithm for computing $\S(M)$.

\paragraph{Structure of the Augmented Arrangement}
$\S(M)$ consists of a line arrangement $\cell(M)$ in $[0,\infty)\times \R$ (that is, a cell decomposition of $[0,\infty)\times \R$ induced by a set of intersecting lines), together with a collection $\P^e$ of pairs $(a,b)\in \R^2\times (\R^2\cup\infty)$ stored at each 2-cell $e$ of $\cell(M)$.  
We call $\P^e$ the \emph{barcode template} at $e$. \nomenclature[Te]{$\P^e$}{barcode template at cell $e$}
As we explain in \cref{Sec:Discrete_Barcodes}, $\P^e$ is defined in terms of the barcode of a discrete 1-D persistence module derived from $M$.
\nomenclature[AM]{$\cell(M)$}{line arrangement}
\nomenclature[e]{$e$}{cell in $\cell(M)$}

\paragraph{Queries of $\S(M)$}
We now briefly describe how we query $\S(M)$ for the barcodes $\B{M^L}$.  Further details are given in \cref{Sec:QueryingMath}.

As noted above, a standard point-line duality, described in \cref{sec:arr_def}, provides an identification of $\L$ with $[0,\infty)\times \R$.  For simplicity, let us restrict attention for now to the generic case where $L$ lies in a 2-cell $e$ of $\cell(M)$; the general case is similar and is treated in  \cref{Sec:QueryingMath}.    
To obtain $\B{M^L}$, for each pair $(a,b)\in \P^{e}$, we ``push" the points of each pair $(a,b)\in \P^e$ onto the line $L$, by moving $a$ and $b$ upwards or rightwards in the plane along horizontal and vertical lines; this gives a pair of points $\push_L(a),\push_L(b)\in L$; if $b=\infty$, we take $\push_L(b)=\infty$.  

Our \cref{Thm:QueriesMain}, the central mathematical result underlying the RIVET paradigm, tells us that
\[\B{M^L}=\{\,[\push_L(a),\,\push_L(b)) \mid (a,b)\in \P^{e}\,\};\]
see \cref{fig:Barcode_Templates} for an illustration.  Thus, to obtain the barcode of $\B{M^L}$ it suffices to identify the cell $e$ and then compute $\push_L(a)$ and $\push_L(b)$ for each $(a,b)\in \P^{e}$.

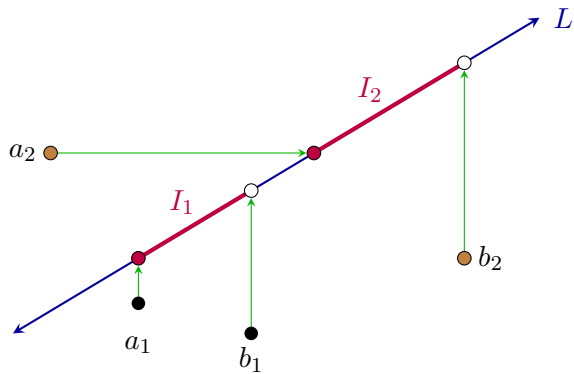
\begin{figure}[ht]
  \begin{center}
    \begin{tikzpicture}
      \draw[<->,thick,color=blue!60!black] (0,0) -- (7,4.2);
       \node[right,color=blue!60!black] at (7.05,4.2) {$L$};   
      
      \draw[ultra thick,color=purple] (1.667,1) -- (3.167,1.9);
      \draw[ultra thick,color=purple] (4,2.4)-- (6,3.6);
      \node[above left,color=purple] at (2.55,1.45) {$I_1$};   
      \node[above left,color=purple] at (5.05,2.95) {$I_2$};   
      
      \coordinate (a1) at (1.667,0.4);
      \coordinate(pa1) at (1.667,1);
      \coordinate (b1) at (3.167,0.0);
      \coordinate (pb1) at (3.167,1.9);

      \draw[->,green!70!black] (a1) -- (1.667,0.9);
       \draw[->,green!70!black] (b1) -- (3.167,1.8);         
      \fill (a1) circle (0.09);
      \fill (b1) circle (0.09);
      \draw[fill=purple] (pa1) circle (0.09);
      \draw[fill=white] (pb1) circle (0.09);
      \node[below] at (1.667,0.1) {$a_1$};
      \node[below] at (3.167,-0.05) {$b_1$};
      
      \coordinate (a2) at (.5,2.4);
      \coordinate(pa2) at (4,2.4);
      \coordinate(b2) at (6,1);
      \coordinate(pb2) at (6,3.6);
      \draw[->,green!70!black] (a2) -- (3.9,2.4);  
      \draw[->,green!70!black] (b2) -- (6,3.5);       
      \draw[fill=brown] (a2) circle (0.09);
       \draw[fill=brown] (b2) circle (0.09);
      \draw[fill=purple] (pa2) circle (0.09);
       \draw[fill=white] (pb2) circle (0.09);
      \node[left] at (.45,2.4) {$a_2$};
      \node[right] at (6.05,1) {$b_2$};
      
    \end{tikzpicture}
  \end{center}
  \caption{An illustration of how we recover the barcode $\protect\B{M^L}$ from the barcode template $\protect\P^e$ by ``pushing" the points of each pair in the barcode template onto $L$.  In this example, $\protect\P^e=\{(a_1,b_1),(a_2,b_2)\}$, and $\protect\B{M^L}$ consists of two disjoint intervals $I_1$, $I_2$.}  
  \label{fig:Barcode_Templates}
\end{figure}

\subsection{Complexity Results for Computing, Storing, and Querying The Augmented Arrangement}\label{Sec:Intro_Complexity}
Once $\S(M)$ has been computed, computing $\B{M^L}$ via a query to $\S(M)$ is far more efficient than computing $\B{M^{L}}$ from scratch; for typical persistence modules arising from data, the query of $\B{M^L}$ can be performed in real time, as desired.
At the same time, the cost of computing and storing $\S(M)$ is reasonable.  The following two theorems provide a theoretical basis for these claims.

For $M$ a 2-D persistence module and $i\in \{0,1,2\}$, let \[\supp \xi_i(M)=\{a\in \R^n\mid \xi_i(M)(a)>0\}.\]
Let $\kappa =\kappa_x \kappa_y$, for $\kappa_x$ and $\kappa_y$ the number of unique $x$ and $y$ coordinates, respectively, of points in $\supp \xi_0(M) \cup \supp \xi_1(M)$; we call $\kappa$ the \emph{coarseness of $M$}.  \nomenclature[ka]{$\kappa$}{coarseness of a 2-D persistence module}

\begin{theorem}[computational cost of querying the augmented arrangement]\label{Thm:Query_Cost}\mbox{}
\begin{enumerate}[(i)]
\item For $L\in \L$ lying in a 2-cell of $\cell(M)$, we can query $\S(M)$ for $\B{M^L}$ in time $O(\log \kappa+|\B{M^L}|)$, where $|\B{M^L}|$ denotes the number of intervals in $\B{M^L}$. 
\item For all other lines $L\in \bar \L$, we can query $\S(M)$ for $\B{M^L}$ in  time $O(\log \kappa+|\B{M^{L'}}|)$, for $L'$ some arbitrarily small perturbation of $L$. 
\end{enumerate}
\end{theorem}

\begin{theorem}\label{SimpleAugArrComplexity}\mbox{}
%Fix $i\geq 0$.  
For $\F$ a bifiltration of size $m$ and $M=H_i(\F)$,
\begin{enumerate}[(i)]
\item $\S(M)$ is of size $O(m\kappa^2)$,
\item Our algorithm computes $\S(M)$ from $\F$ using $O(m^3 \kappa+(m+\log \kappa)\kappa^2)$ elementary operations and $O(m^2+m\kappa^2)$ storage.%.
\end{enumerate}
\end{theorem}

We prove \cref{Thm:Query_Cost} in \cref{Sec:Queries} and \cref{SimpleAugArrComplexity} in \cref{Sec:Complexity_Results}.  

To keep our exposition brief in this introduction, we have assumed in the statement of \cref{SimpleAugArrComplexity} that $M$ is a persistent homology module of a bifiltration.  However, our algorithm for computing augmented arrangements does handle purely algebraic input; see \cref{Sec:FI_Reps}.  We give more general complexity bounds in the algebraic setting in \cref{Sec:Complexity_Results}.  

\paragraph{Coarsening and the Interpretation of \cref{SimpleAugArrComplexity}}
For a fixed choice of $L$, $\B{M^L}$ is of size $O(m)$, and the time required to compute $\B{M^L}$ via an application of the standard persistence algorithm is $O(m^3)$.  Thus  \cref{SimpleAugArrComplexity} indicates that, as one would expect, computation and storage of $\S(M)$ is more expensive than computation and storage of $\B{M^L}$ for some fixed $L$.  

For $m$ as above, $|\kappa|$ is $O(m^2)$ in the worst case.  Thus, in the worst case, our bounds on the size and time to compute $\S(M)$ grow like $m^5$, which on the surface may appear problematic for practical applications.  However, as we explain in \cref{Sec:Coarsening}, we can always employ a simple coarsening procedure to approximate $M$ by a module $M'$ for which $\kappa$ is a small constant (say, $\kappa\leq 400$).  $\S(M')$ encodes $\B{M'}$ exactly, and so in view of Landi's external stability result \cite{landi2014rank}, $\S(M)$ encodes $\B{M}$ approximately.  More details on coarsening are given in \cref{Sec:Coarsening}.

\subsection{Computation of the Augmented Arrangement}\label{Sec:Intro_Computation}
Our algorithm for computing $\S(M)$ decouples into three main parts:
\begin{enumerate}
\item Computing $\xi_0(M)$ and $\xi_1(M)$,
\item Constructing the line arrangement $\cell(M)$,  
\item Computing the barcode template $\P^e$ at each 2-cell $e$ of $\cell(M)$.  
\end{enumerate}
We next say a few words about each of these.

\paragraph{Computing Bigraded Betti Numbers}
As noted above, one of the main mathematical contributions underlying RIVET is a fast algorithm for computing the bigraded Betti numbers of a 2-D persistence module $M$.  Not only does RIVET provide a visualization of the Betti numbers, but it also makes essential use of the Betti numbers in constructing $\S(M)$.

For $M$ a persistent homology module of a bifiltration with $n$ simplices, our algorithm for computes the bigraded Betti numbers of $M$ in $O(m^3)$ time;  see \cite{lesnick2014betti} and \cref{sec:ComputingBettiNumbers}.    

In \cref{Sec:Runtimes} of this paper, we present preliminary experimental results on the performance of our algorithm for computing bigraded Betti numbers.  These indicate that the cost of the algorithm is very reasonable in practice. 

\paragraph{Computing the Line Arrangement}
The second phase of computation constructs the line arrangement $\cell(M)$ underlying $\S(M)$.  Line arrangements have been the object of intense study by computational geometers for decades, and there is well-developed machinery for constructing and working with line arrangements in practice \cite{deberg2008computation}.  Our algorithms for constructing and querying $\S(M)$ leverage this machinery; see \cref{Sec:Representing_And_Querying_Aug_Arrangement} and \cref{Sec:Computing_Line_Arrangement}.
 
\paragraph{Computing the Barcode Templates}
The third phase of our computation of $\S(M)$ computes the barcode templates $\P^e$ stored at each 2-cell $e\in \cell(M)$.  
As noted above, each $\P^e$ is defined in terms of the barcode $\B{M^e}$ of a certain 1-D persistence module $M^e$.  To compute each $\P^e$ for each 2-cell $e$, we need to compute each $\B{M^e}$.  This is the most expensive part of the computation of $\S(M)$, both in theory and in practice. In \cref{Sec:Computing_Additional_Data_at_Faces}, we introduce our core algorithm for this, based on the vineyard algorithm for updating persistent homology computations \cite{cohen2006vines}.  In \cref{Sec:Barcodes_Scratch}, we present a modification of this algorithm which is much faster in practice. 

In fact, as explained in \cref{Sec:Parallel_Comp}, our algorithm for computing barcode templates is embarrassingly parallelizable.  

\paragraph{Computational Experiments}
\cref{Sec:Runtimes} presents preliminary results on the performance of our algorithm for computing augmented arrangements.  As explained there, our present implementation of RIVET is not yet fully optimized, and our timing results should be regarded as loose upper bounds on what can be achieved using the algorithms of this paper.

Still, the results demonstrate that even with our current code, computation of an augmented arrangement of a bifiltration containing millions of simplices is feasible on a standard personal computer, provided we employ some modest coarsening. 
Thus, the current implementation already performs well enough to be used in the analysis of modestly sized real world data sets.  
  With more implementation work, including the introduction of parallelization, we expect RIVET to scale well enough to be useful in many of the same settings where 1-D persistence is currently used for exploratory data analysis.

\subsection{Outline}
We conclude this introduction with an outline of the remainder of the paper.

 \cref{Sec:Preliminaries} reviews basic algebraic facts about persistence modules, their (minimal) presentations, and graded Betti numbers.  We also discuss the connection between $\R^n$-indexed persistence modules and their $\Z^n$-indexed discretizations.  
\cref{Sec:Arrangements} defines the augmented arrangement $\S(M)$ of a 2-D persistence module $M$.  \cref{Sec:Selecting_A_Coface,Sec:Query_Theorem} give our main result on querying $\S(M)$ for the barcodes $\B{M^L}$; this is \cref{Thm:QueriesMain}.
In \cref{Sec:Representing_And_Querying_Aug_Arrangement}, we describe how $\S(M)$ is stored in memory, and apply \cref{Thm:QueriesMain} to give an algorithm for querying $\S(M)$.

The remaining sections introduce our algorithm for computing $\S(M)$: First, \cref{Sec:Computing_Line_Arrangement}  specifies how persistence modules are represented as input to our algorithm, and explains our algorithm for computing $\cell(M)$.  \cref{Sec:Computing_Additional_Data_at_Faces} explains our core algorithm for computing the barcode templates $\P^e$; this completes the specification of our algorithm for computing $\S(M)$, in its basic form.  \cref{Sec:Complexity_Results} analyzes the time and space complexity of our algorithm for computing $\S(M)$.  \cref{Sec:Speedups} describes several practical strategies to speed up the computation of $\S(M)$.
\cref{Sec:Runtimes} presents our preliminary timing results for the computation of $\S(M)$.

\cref{Sec:Interface_Details} expands on the introduction to the RIVET interface given in \cref{Sec:Intro_Visualization}, providing additional details.

%% file: VRI_Preliminaries.tex
%VRI_Preliminaries.tex

\section{Algebra Preliminaries}\label{Sec:Preliminaries}
In this section, we present the basic algebraic definitions and facts we will need to define and study augmented arrangements of 2-D persistence modules.

\subsection{A Module-Theoretic Description of Multi-D Persistence Modules}\label{Sec:MultidimensionalPersistenceModules}
In \cref{Sec:Multi_D_Persistent_Homology} we defined a $n$-D persistence module to be an object of the functor category $\nmod$.  Here we give a module-theoretic description of $\nmod$.

\paragraph{The Ring $P_n$}\label{MonoidRings}
Let the ring $P_n$ be the analogue of the usual polynomial ring $K[x_1,\ldots,x_n]$ in $n$ variables, where exponents of the indeterminates in $P_n$ are allowed to take on arbitrary values in $[0,\infty)$ rather than only values in $\Nbb$.  For example, if $K=\R$, then $1+x_2+x_1^\pi+\frac{2}{5}x_1^3x_2^{\sqrt 2}$ is an element of $P_2$.

More formally, $P_n$ can be defined as a {\it monoid ring} over the monoid $([0,\infty)^n,+)$ \cite{lang2002algebra}.

For $a=(a_1,\ldots,a_n)\in [0,\infty)^n$, we let $x^a$ denote the monomial $x_1^{a_1}x_2^{a_2}\cdots x_n^{a_n}\in P_n$. 
Let $I\subset P_n$ be the ideal generated by the set $\{x_i^\alpha\mid \alpha>0, 1\leq i\leq n\}$.  

\paragraph{$\R^n$-graded $P_n$-Modules}
Since the field $K$ is a subring of $P_n$, any $P_n$-module comes naturally equipped with the structure of a $K$-vector space.  

We define an \emph{$\R^n$-graded} $P_n$-module to be a $P_n$-module $M$ with a direct sum decomposition as a $K$-vector space $M \simeq \bigoplus_{a \in {\R}^n} M_a$ such that the action of $P_n$ on $M$ satisfies $x^{b}(M_a) \subset M_{a+b}$ for all $a\in \R^n,$ $b\in [0,\infty)^n$.

The $\R^n$-graded $P_n$-modules form a category whose morphisms are the module homomorphisms $f:M\to N$ such that $f(M_a)\subset N_a$ for all $a \in \R^n$.  There is an obvious isomorphism between this $\nmod$ and this category, so that we may identify the two categories.  Henceforth, we will refer to $\R^n$-graded $P_n$-modules as ($n$-D) persistence modules, or $n$-modules for short.

\begin{remark}
As a rule, the familiar definitions and constructions for modules make sense in the category $\nmod$.  For example, as the reader may check, we can define submodules, quotients, direct sums, tensor products, projective/injective resolutions, and Tor functors in $\nmod$.  As we next explain, we also can define free $n$-modules and presentations of $n$-modules. 
\end{remark}

\subsection{Free \texorpdfstring{$n$}{n}-Modules and Presentations}\label{Sec:Free_Modules_Presentations}
 
\paragraph{$n$-graded Sets}
Define an {\it $n$-graded set} to be a pair $\W=(W,\gr_\W)$ for some set $W$ and function $\gr_\W:W\to\R^n$.  Formally, we may regard $\W$ as the set of pairs \[\{(w,\gr_\W(w))\mid w\in W\},\] and we'll sometimes make use of this representation.  
 We'll often abuse notation and write $\W$ to mean the set $W$.  Also, when $\W$ is clear from context we'll abbreviate $\gr_\W$ as $\gr$.
\nomenclature[gr]{$\gr_\W$}{grade function of the $n$-graded set $\W$}
        
We say a subset $\Y$ of an $n$-module $M$ is \emph{homogeneous} if \[\Y\subseteq \bigcup_{a\in \R^n} M_a.\]  Clearly, we may regard $\Y$ as an $n$-graded set. %so that we have a well defined grade function $\gr:Y\to \R^n$.     

\paragraph{Shifts of Modules}
For $M$ an $n$-module and $v\in \R^n$, we define $M(v)$ to be the $n$-module such that for $a\in \R^n$, $M(v)_a=M_{a+v}$, and for $a\leq b\in \R^n$, $M(v)(a,b)=M(a+v,b+v)$.

For example, when $n=2$, $M(1,1)$ is obtained from $M$ by shifting all vector spaces of $M$ down by one and to the left by one.  

\paragraph{Free $n$-Modules}
The usual notion of a free module extends to the setting of $n$-modules as follows: For $\W$ an $n$-graded set, let 
$\free[\W]=\bigoplus_{w\in \W} P_n(-\gr(w))$.  We identify $\W$ with a set of generators in $\free[\W]$, in the obvious way.  A {\it free $n$-module} $F$ is an $n$-module such that $F\simeq \free[\W]$ for some $n$-graded set $\W$.  Equivalently, we can define a free $n$-module as an $n$-module which satisfies a certain universal property; see \cite{carlsson2009theory}.

For $\Y$ a homogeneous subset of a free $n$-module $F$, let $\langle \Y \rangle$ denote the submodule of $F$ generated by $\Y$.

\paragraph{Matrix Representations of Morphisms of Free Modules}
Let $\W,\W'$ be finite $n$-graded sets with ordered underlying sets $W=\{w_1,\ldots,w_l\}$, $W'=\{w'_1,\ldots,w'_m\}$.  For a morphism $f:\free[\W]\to \free[\W']$, we can represent $f$ by a $|W|\times |W'|$ matrix $[f]$ with coefficients in $K$:  If $\gr(w'_i)\leq \gr(w_j)$, we define $[f]_{ij}$ to be the unique solution to \[ [f]_{ij}\, w'_i =x^{\gr(w'_i)-\gr(w_j)}\rho_i \circ f(w_j),\] where $\rho_i: \free[\W]\to \free[\{w_i\}]$ is the projection. If $\gr(w'_i)\not\leq \gr(w_j)$, we define $[f]_{ij}=0$.

\paragraph{Presentations of $n$-Modules}
A {\it presentation} of an $n$-module $M$ is a pair $(\W,\Y)$, where $\W$ is an $n$-graded set and $\Y \subseteq \free[\W]$ is a homogeneous, with $M\simeq \free[\W]/\langle \Y \rangle$.  We denote the presentation $(\W,\Y)$ as $\langle \W|\Y \rangle$.  If there exists a presentation $\langle \W|\Y \rangle$ for $M$ with $\W$ and $\Y$ finite, then we say $M$ is {\it finitely presented}.

Note that the inclusion $\Y\hookrightarrow \free[\W]$ induces a morphism $\free[\Y]\to \free[\W]$; we denote this morphism as $\Phi_{\langle \W|\Y\rangle}$.

\begin{example}\label{Ex:Rank_Invariant_Not_Complete}
Consider the 2-D persistence modules
\begin{align*}
M&\simeq\langle(a,(1,0)), (b,(0,1)), (c,(1,1))\mid x_2a-x_1b\rangle \\
N&\simeq\langle(a,(1,0)), (b,(0,1))\mid \emptyset \rangle.
\end{align*}
The induced linear maps $M_{(1,0)}\oplus M_{(0,1)}\to M_{(1,1)}$ and $N_{(1,0)}\oplus N_{(0,1)}\to N_{(1,1)}$ do not have equal ranks.  Hence $M$ and $N$ are not isomorphic.  However, $\rk(M)=\rk(N)$.  This shows that the rank invariant does not completely determine the isomorphism type of a persistence module.   
\end{example}

\subsection{Minimal Presentations of \texorpdfstring{$n$}{n}-Modules}\label{Sec:MinPresentations}
Let $M$ be an $n$-module.  We say a presentation $\langle \W|\Y \rangle$ of $M$ is \emph{minimal} if $\langle\Y\rangle \subseteq I\free[\W]$ and $\ker \Phi_{\langle \W|\Y\rangle}\subseteq I\free[\Y].$

The following proposition is a variant of \cite[Lemma 19.4]{eisenbud1995commutative}, and is proved in the same way.  It makes clear that minimal presentations are indeed minimal, in a reasonable sense.  

\begin{proposition}\label{MinPresentations}
A finite presentation $\langle \W|\Y\rangle$ is minimal if and only if $\W$ descends to a minimal set of generators for $\coker \Phi_{\langle \W|\Y\rangle}$ and $\Y$ is a minimal set of generators for $\langle \Y\rangle$.
\end{proposition}

\begin{remark}
It follows immediately from \cref{MinPresentations} that every finitely presented $n$-D persistence module has a minimal presentation.  
\end{remark}

\subsection{Graded Betti Numbers of Persistence Modules}\label{Sec:Betti_Numbers}
For $M$ an $n$-module, define $\dimf(M):\R^n\to \Nbb$, the \emph{dimension function of $M$}, by 
$\dimf(M)(a)=\dim(M_a)$.

For $i\geq 0$, define \[\xi_i(M)=\dimf(\Tor_i(M,P_n/I P_n)).\]  The functions $\xi_i(M)$ are called the \emph{(graded) Betti numbers} of $M$.    Betti numbers of multigraded $K[x_1,\cdots,x_n]$-modules are defined analogously; these are discussed in many places, e.g., in \cite{carlsson2009theory}.  In our study of augmented arrangements and fibered barcodes, we will only need to consider $\xi_0(M)$ and $\xi_1(M)$.
  
We omit the straightforward proof the following result:
\begin{proposition}\label{MinPresAndBetti}
If $\langle \W|\Y\rangle$ is a minimal presentation for $M$, then for all $a\in \R^n$, \[\xi_0(M)(a)=|\gr_\W^{-1}(a)|,\qquad\xi_1(M)(a)=|\gr_\Y^{-1}(a)|.\]
\end{proposition}

\begin{example}
The presentations of the modules $M$ and $N$ given in \cref{Ex:Rank_Invariant_Not_Complete} are minimal.  Using this, it's easy to see that
\begin{align*}
	\xi_0(M)(a) &=
	\begin{cases}
		1 &\textup{ if } a\in \{(1,0),(0,1),(1,1)\}\\
		0 &\textup{ otherwise},
	\end{cases} &
	\xi_1(M)(a) &= 
	\begin{cases}
		1 &\textup{ if } a\in \{(1,1)\}\\
		0 &\textup{ otherwise},
	\end{cases} \\
	\xi_0(N)(a) &= 
	\begin{cases}
		1 &\textup{ if } a\in \{(1,0),(0,1)\}\\
		0 &\textup{ otherwise}.
	\end{cases} &
	\xi_1(N) &= \xi_2(N)=\xi_2(M)=0.
\end{align*}
\end{example}

\begin{example}
For $M\simeq\langle(a,(0,0))\mid x_1a,x_2a\rangle$,
\[
\xi_0(M)(a)=
\begin{cases}
1 &\textup{ if } a\in \{(0,0)\}\\
0 &\textup{ otherwise},
\end{cases}
\qquad
\xi_1(M)(a)= 
\begin{cases}
1 &\textup{ if } a\in \{(1,0),(0,1)\}\\
0 &\textup{ otherwise},
\end{cases}
\]
\[
\xi_2(M)(a)= 
\begin{cases}
1 &\textup{ if } a\in \{(1,1)\}\\
0 &\textup{ otherwise}.
\end{cases}
\]
\end{example}

\paragraph{Grades of Influence}
For $M$ a persistence module and $a\in \R^n$ let \[\I_M(a)=\{y\leq a \mid \xi_0(y)>0\textup{ or }\xi_1(y)>0\}.\]

\begin{lemma}\label{Lem:GradesOfInfluenceAndIsomorphisms}
For $M$ finitely presented and $a,b\in \R^n$ with $\I_M(a)=\I_M(b)$, $M(a,b)$ is an isomorphism.  
\end{lemma}

\begin{proof}
Let $\langle \W | \Y\rangle $ be a minimal presentation for $M$.  Let $M'$ be the $n$-module with $M'_a=\free[\W]_a/\langle \Y\rangle_a$, and $M'(a,a')$ the map induced by the inclusion $\free[\W](a,a')$.  Clearly, $M'$ isomorphic to $M$.  Using \cref{MinPresAndBetti}, it's easy to see that for $a,b\in \R^n$ with $\I_M(a)=\I_M(b)$, the map $\free[\W](a,b)$ is an isomorphism sending $\langle \Y\rangle_a$ isomorphically to $\langle \Y\rangle_b$.  Hence $M'(a,b)$ is an isomorphism.  Since $M$ and $M'$ are isomorphic, $M(a,b)$ is an isomorphism as well.  
\end{proof}

\subsection{Continuous Extensions of Discrete Persistence Modules}\label{Sec:Continuous_Extensions}
In the computational setting, the persistence modules we encounter are always finitely presented.  It turns out that finitely presented persistence modules are, in a sense, \emph{essentially discrete}; we now explain this.  

A \emph{discrete (or $\Z^n$-indexed) persistence module} is a functor $\ZCat^n\to \kvect$, where $\ZCat^n$ is the poset category of $\Z^n$.\nomenclature[Z]{$\ZCat^n$}{Poset category of $\Z^n$}  
Let $\An=K[x_1,\ldots,x_n]$ denote the ordinary polynomial ring in $n$ variables.  In analogy with the $\R^n$-indexed case, we can regard a discrete persistence module as a multi-graded $\An$-module, in the obvious way.

All of the basic definitions and machinery we've described above for $\R^n$-indexed persistence modules can be defined for discrete persistence modules in essentially the same way.  In particular, we may define the Betti numbers $\xi_i(\DiscreteMod):\Z^n\to \Nbb$ of a discrete persistence module $\DiscreteMod$.

\paragraph{Grid Functions}
\nomenclature[G]{$\G$}{$n$-D grid function}
For $n\geq 1$, we define an \emph{n-D grid} to be a function $\G:\Z^n\to \R^n$, given by \[\G(z_1,\dots,z_n)=(\G_1(z_1),\ldots,\G_n(z_n))\] for some non-decreasing functions $\G_i:\Z\to \R$ with \[\lim_{z\to -\infty} \G_i(z)=-\infty\qquad\textup{and}\qquad  \lim_{z\to \infty}\G_i(z)=\infty\] for each $i$.

\nomenclature[fl]{$\fl_\G$}{floor function for $\im \G$}
We define $\fl_\G(a)$ to be the maximum element of $\im(\G)$ ordered before $a$ in the partial order on $\R^n$.  That is, for $\G$ a 1-D grid, we let $\fl_\G:\R\to \im(\G)$ be given by \[\fl_\G(t)=\max \{s\in \im(\G)\mid s\leq t\},\] and for $\G$ an $n$-D grid function, we define $\fl_\G:\R^n\to \im(\G)$ by 
\[\fl_\G(a_1,\ldots,a_n)=(\fl_{\G_1}(a_1),\ldots,\fl_{\G_n}(a_n)).\]

\paragraph{Continuous Extensions of Discrete Persistence Modules}
\nomenclature[EG]{$\CoEx_{\G}$}{continuous extension functor}
For $\G$ an $n$-D grid, we define a functor $\CoEx_\G:\kVect^{\ZCat^n}\to \kVect^{\RCat^n}$ as follows:
\begin{enumerate}
\item For $\DiscreteMod$ a $\Z^n$-indexed persistence module and $a,b\in \R^n$, \[\CoEx_\G(\DiscreteMod)_a=\DiscreteMod_y,\qquad  \CoEx_\G(\DiscreteMod)(a,b)=\DiscreteMod(y,z),\] where 
\begin{align*}
y&=\max\, \{w\in \Z^n\mid \G(w)=\fl_\G(a)\},\\
z&=\max\, \{w\in \Z^n\mid \G(w)=\fl_\G(b)\}.
\end{align*}
\item The action of $\CoEx_\G$ on morphisms is the obvious one.
\end{enumerate}

We say that an $n$-module $M$ is a \emph{continuous extension of $\DiscreteMod$ along $\G$} if $M\cong\CoEx_\G(\DiscreteMod)$.

\begin{proposition}\label{Prop:Finitely_Presented_Implies_Discrete}
Any finitely presented $n$-module $M$ is a continuous extension of a finitely generated discrete persistence module along some $n$-D grid.
\end{proposition}

\begin{proof}
Let $\G:\Z^n\to \R^n$ be any $n$-D grid such that $\supp \xi_0(M)\cap\supp \xi_1(M)\subseteq \im \G$.  We regard $\G$ as a functor $\ZCat^n\to \RCat^n$ in the obvious way.  
Using \cref{Lem:GradesOfInfluenceAndIsomorphisms}, it's easy to check that $M$ is a continuous extension of $M\circ \G$ along $\G$.  Further, $M\circ \G$ is finitely generated; a finite presentation for $M$ induces one for $M\circ \G$ of the same size.
\end{proof}

\paragraph{Betti Numbers of Continuous Extensions}

\begin{proposition}\label{Prop:BettiNumsContinuousExtensions}
Suppose an $n$-module $M$ is a continuous extension of $\DiscreteMod$ along an injective $n$-D grid $\G$.  Then for all $i$, 
\begin{equation*}
\xi_i(M)(a)=
\begin{cases}
\xi_i(\DiscreteMod)(z) &\textup{if } a=\G(z)\textup{ for some } z\in \Z^n,\\
0 &\textup{otherwise.}
\end{cases}
\end{equation*}
\end{proposition}

RIVET exploits \cref{Prop:BettiNumsContinuousExtensions} to compute the Betti numbers of finitely presented $\R^2$-indexed persistence modules, by appealing to local formulae for the Betti numbers of $\Z^2$-indexed persistence modules; see \cref{sec:ComputingBettiNumbers}.

\begin{proof}[Proof of \cref{Prop:BettiNumsContinuousExtensions}]
Let \[\cdots \xrightarrow{\alpha^3} F^2 \xrightarrow{\alpha^2} F^1\xrightarrow{\alpha^1} F^0\]
be a free resolution of $\DiscreteMod$.  It's easy to see that $\CoEx_\G$ preserves exactness, so 
\[\cdots \xrightarrow{\CoEx_\G(\alpha^3)} \CoEx_\G(F^2)\xrightarrow{\CoEx_\G(\alpha^2)} \CoEx_\G(F^1)\xrightarrow{\CoEx_\G(\alpha^1)} \CoEx_G(F^0)\]
is a free resolution for $M$.  Write $G^i=\CoEx_\G(F^i)$ and $\beta^i=\CoEx_\G(\alpha^i).$

By definition, $\xi_i(M)=\dimf(\mathcal H^i)$, where $\mathcal H^i$ is the $i^{\mathrm{th}}$ homology module of the following chain complex:  
\[\cdots \xrightarrow{\beta_3 \otimes \id} G^2 \otimes P_n /I \xrightarrow{\beta_2 \otimes\id} G^1\otimes P_n/I\xrightarrow{\beta_1 \otimes \id} G^0\otimes P_n/I.\]

We have two functors $\nmod\to \nmod$ acting respectively on objects by 
\begin{align*}
N\to N\otimes P_n/I,\\ 
N\to N/IN,
\end{align*}
with the action on morphisms defined in the obvious way; note that these are naturally isomorphic.  Thus the above chain complex is isomorphic to the chain complex
\[\cdots \xrightarrow{\bar \beta^3 } G^2/IG^2\xrightarrow{\bar \beta^2} G^1/IG^1\xrightarrow{\bar \beta^1} G^0/IG^0,\]
where $\bar \beta^i$ is the map induced by $\beta^i$.  Since $G^i$ is a continuous extension along $\G$, it is clear that if $a\not\in \im \G$, then $(G^i/IG^i)_a=0$.  Hence, if $a\not\in \im \G$ then $\xi_i(M)(a)=0$, as claimed.  

It remains to consider the case that $a\in \im \G$.  Let $J$ denote the maximal homogeneous ideal of $\An$.  
$\xi_i(\DiscreteMod)=\dimf(\mathcal K^i)$, for $\mathcal K^i$ the $i^{\mathrm{th}}$ homology module of the chain complex
\[\cdots \xrightarrow{\bar \alpha^3 } F^2/J F^2\xrightarrow{\bar \alpha^2} F^1/J F^1\xrightarrow{\bar\alpha^1} F^0/J F^0,\]
where $\bar \alpha^i$ is the map induced by $\alpha^i$.
If $a=\G(z)$, then by the way we have defined the functor $\CoEx_\G$, it is clear that we have 
isomorphisms $\{\psi_i:F^i_z\to G^i_a\}_{i\geq 0}$ sending $J F^i_z$ isomorphically to $IG^i_a$, such that the following diagram commutes:
\[
\xymatrix{
\cdots \ar[r]
&F^2_z \ar[r]^{\alpha^2_z}\ar[d]^{\psi_2} &F^1_z\ar[r]^{\alpha^1_z}\ar[d]^{\psi_1} &F^0_z\ar[d]^{\psi_0}\\
\cdots \ar[r]
&G^2_a \ar[r]^{\beta^2_a} &G^1_a\ar[r]^{\beta^1_a} &G^0_a.
}
\]
Taking quotients, we obtain a commutative diagram:
\[
\xymatrix{
\cdots \ar[r]
&(F^2/J F^2)_z \ar[r]^{\bar \alpha^2_z}\ar[d]^{\simeq} &(F^1/J F^1)_z\ar[r]^{\bar\alpha^1_z}\ar[d]^{\simeq} &(F^0/J F^0)_z\ar[d]^{\simeq}\\
\cdots \ar[r]
&(G^2/I G^2)_a \ar[r]^{\bar\beta^2_a} &(G^1/I G^1)_a\ar[r]^{\bar\beta^1_a} &(G^0/I G^0)_a.
}
\]

It follows that ${\mathcal K}^i_z\simeq {\mathcal H}^i_a$, so $\xi_i(M)_a=\xi_i(\DiscreteMod)_z$ as desired.
\end{proof}

\paragraph{Barcodes of Discrete Persistence Modules}
\begin{sloppypar}
We discussed the barcodes of ($\R$-indexed) 1-D persistence modules in \cref{Sec:Multi_D_Persistent_Homology}.  The structure theorem of \cite{webb1985decomposition} tells us that the barcode $\B{\DiscreteMod}$ of a discrete 1-D persistence module $\DiscreteMod$ is also well defined, provided $\dim(V_a)<\infty$ for all $a$ less than some $a_0\in \Z$.  When $\DiscreteMod$ is finitely generated, $\B{Q}$ is a finite multiset of intervals $[a,b)$ with $a<b\in \Z\cup \infty$.
\end{sloppypar}

\paragraph{Barcodes Under Continuous Extension}
We omit the easy proof of the following:

\begin{proposition}\label{BarcodesContinuousExtensions}
For $\DiscreteMod$ a finitely generated discrete 1-D persistence module and $M$ a continuous extension of $\DiscreteMod$ along $\G$,
\[\B{M}=\{[\G(a),\G(b))\mid [a,b)\in \B{\DiscreteMod},\, \G(a)<\G(b) \},\]
where we define $\G(\infty)=\infty$.  
\end{proposition}

\begin{remark}\label{Rmk:Extensions_On_L}
As already noted in \cref{sec:rank_inv_fibered_barcodes}, for $\LCat$ the poset category corresponding to a line $L\in \bar L$, $\LCat$ is isomorphic to $\RCat$.  Hence, by adapting the definitions given above in the $\R$-indexed setting,  we can define the grid function $\G:\Z\to L$, a function $\fl_\G: L\to \im G$, and the functor $\CoEx_\G:\kVect^{\ZCat}\to \kVect^{\LCat}$.  As in the $\R$-indexed case, we say $M:\LCat\to \kvect$ is a continuous extension of a $\Z$-indexed persistence module $\DiscreteMod$ if $M\cong \CoEx_\G(\DiscreteMod)$.  Clearly then, \cref{BarcodesContinuousExtensions} also holds for continuous extensions in the $L$-indexed setting.  In \cref{Sec:Query_Theorem}, we will use \cref{BarcodesContinuousExtensions} in the $L$-indexed setting to prove our main result on queries of augmented arrangements.
\end{remark}  

%% file: VRI_L_Arrangements.tex
%VRI_L_Arrangments

\section{Augmented Arrangements of 2-D Persistence Modules}\label{Sec:Arrangements}
In this section we define the augmented arrangement $\S(M)$ associated to a finitely presented 2-D persistence module $M$. 

First, in \cref{sec:arr_def} we define the line arrangement $\cell(M)$ associated to $M$.  Next, in \cref{Sec:CriticalLines} we present a characterization of the 1-skeleton of $\cell(M)$.  Finally, using this characterization, in \cref{Sec:Discrete_Barcodes} we define the barcode template $\P^e$ stored at a 2-cell $e$ of $\cell(M)$.

The augmented arrangement $\S(M)$ is defined to be the arrangement $\cell(M)$, together with the additional data $\P^e$  at each $2$-cell $e$.

\subsection{Definition of \texorpdfstring{$\cell(M)$}{A(M)}}\label{sec:arr_def}
\nomenclature[S]{$\XiSp$}{union of the supports of the $0^{\mathrm{th}}$ and $1^{\mathrm{st}}$ bigraded Betti numbers of $M$}
Let $\XiSp=\supp \xi_0(M)\cup\supp \xi_1(M)$.  To keep our exposition simple, we will assume that each element of $\XiSp$ has non-negative $x$-coordinate.  Using the shift construction described in \cref{Sec:Free_Modules_Presentations}, we can always translate the indices of $M$ so that this assumption holds, so there is no loss of generality in this assumption.

\paragraph{Point-Line Duality}
Recall the definitions of $\L$ and $\bar \L$ from \cref{sec:rank_inv_fibered_barcodes}.  As mentioned there, a standard point-line duality gives a parameterization of $\L$ by the half-plane $[0,\infty) \times \R$.  We now explain this. 
Define dual transforms $\dual_\ell$ and $\dual_p$ as follows: \nomenclature[Dl]{$\dual_\ell$, $\dual_p$}{point-line duality transforms}
\begin{align*}
	\dual_\ell : \L &\to [0,\infty) \times \R  & 	\dual_p : [0,\infty) \times \R &\to \L \\
	y=ax+b &\mapsto (a,-b) &			(c,d) &\mapsto y=cx-d
\end{align*}
This duality does not extend naturally to vertical lines, i.e.\ lines in $\bar \L-\L$.  

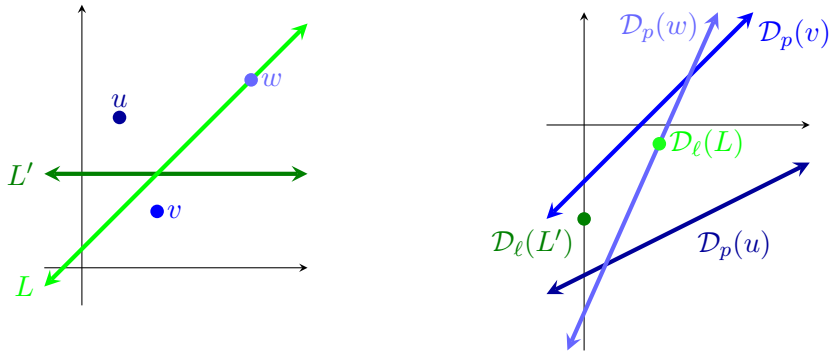
\begin{figure}[ht]
  \begin{center}
    \begin{tikzpicture}
        \draw[->] (-0.5,0) -- (3,0);
        \draw[->] (0,-0.5) -- (0,3.5);
        
        \draw[<->,ultra thick,color=green!50!black] (-0.5,1.25) -- (3,1.25);
        \node[left,color=green!50!black] at (-0.5,1.25) {$L'$};
        \draw[<->,ultra thick,color=green] (-0.5,-0.25) -- (3,3.25);
        \node[left,color=green] at (-0.5,-0.25) {$L$};
          
        \coordinate (p1) at (0.5,2);
        \fill[blue!60!black] (p1) circle (0.09);
        \node[above,color=blue!60!black] at (p1) {$u$};
      
        \coordinate (p2) at (1,0.75);
        \fill[blue] (p2) circle (0.09);
        \node[right,color=blue] at (p2) {$v$};
        
        \coordinate (p3) at (2.25,2.5);
        \fill[blue!60!white] (p3) circle (0.09);
        \node[right,color=blue!60!white] at (p3) {$w$};
        
        \node at (0,-1) {};
      \end{tikzpicture}
      \hspace{.8in}
      \begin{tikzpicture}
        \draw[->] (-0.5,0) -- (3,0);
        \draw[->] (0,-3) -- (0,1.5);
        
        \draw[<->,ultra thick,color=blue!60!black] (-0.5,-2.25) -- (3,-0.5);
        \node[color=blue!60!black] at (2,-1.6) {$\dual_p(u)$};
        \draw[<->,ultra thick,color=blue] (-0.5,-1.25) -- (2.25,1.5);
        \node[color=blue] at (2.8,1.2) {$\dual_p(v)$};
        \draw[<->,ultra thick,color=blue!60!white] (-0.222,-3) -- (1.778,1.5);
        \node[color=blue!60!white] at (1,1.3) {$\dual_p(w)$};
        
        \coordinate (q1) at (0,-1.25);
        \fill[green!50!black] (q1) circle (0.09);
        \node[below left,color=green!50!black] at (q1) {$\dual_\ell(L')$};
        
        \coordinate (q2) at (1,-0.25);
        \fill[green!90] (q2) circle (0.09);
        \node[right,color=green] at (q2) {$\dual_\ell(L)$};
    \end{tikzpicture}
  \end{center}
  \caption{Illustration of point-line duality}
  \label{fig:duality}
\end{figure}

The following lemma, whose proof we omit, is illustrated in \cref{fig:duality} by the point $w$ and line $L$.
\begin{lemma}\label{Lem:description_of_lcm_lines}
The transforms $\dual_\ell$ and $\dual_p$ are inverses and preserve incidence, in the sense that for $w\in [0,\infty) \times \R$ and $L\in \L$, $w\in L$ if and only if $\dual_\ell(L)\in \dual_p(w)$.  

\end{lemma}

\paragraph{Line Arrangements in $[0,\infty)\times \R$}
A \emph{cell} is a topological space homeomorphic to $\R^n$ for some $n\geq 0$.  We define a \emph{cell complex} on $[0,\infty)\times \R$ to be a decomposition of $[0,\infty)\times \R$ into a finite number of cells, so that the topological boundary of each cell lies in the union of cells of lower dimension.  By standard point-set topology, each cell in a cell complex on $[0,\infty)\times \R$ has dimension at most 2.  According to the definition, a cell complex on $[0,\infty)\times \R$ is not a CW-complex, as some cells will necessarily be unbounded.  %It can be checked that the line $x=0$ lies in the 1-skeleton of any cell decomposition of $[0,\infty)\times \R$.

By a \emph{line arrangement} in $[0,\infty)\times \R$, we mean the cell complex on $[0,\infty)\times \R$ induced by a set $W$ of lines in $[0,\infty)\times \R$.  In this cell complex, the 1-skeleton consists of the union of all lines in $W$, together with the line $x=0$.

\paragraph{Definition of \texorpdfstring{$\cell(M)$}{A(M)}}
For $W=\{(a_1,b_1), (a_2,b_2), \ldots, (a_l,b_l)\}$ a finite subset of $\R^2$, let $\lub(S)$, the \emph{least upper bound of $W$} be given by \[\lub(S)=(\max_{i}\, a_i,\,\max_{i}\, b_i).\]
\nomenclature[lub]{$\lub$}{least upper bound}
For example, $\lub((3,5),(7,4))=(7,5)$.

Say that a pair of distinct elements $u,v\in \XiSp$ is \emph{weakly incomparable} if one of the following is true:
\begin{itemize*}
\item $u$ and $v$ are incomparable with respect to the partial order on $\R^2$,
\item $u$ and $v$ share either their first or second coordinate. 
\end{itemize*}
Call $\alpha\in [0,\infty)\times \R$ an \emph{anchor} if $\alpha=\lub(u,v)$ for  $u,v\in \XiSp$ weakly incomparable.

We define $\cell(M)$ to be the line arrangement in $[0,\infty)\times \R$ induced by the set of lines
\[ \{\dual_p(\alpha)\mid \alpha \textup{ is an anchor}\}. \]  
In view of \cref{Lem:description_of_lcm_lines}, then, the 1-skeleton $\cell^1(M)$ of $\cell(M)$ is given by
\begin{equation}\label{Eq:Description_Of_One_Skel}
\cell^1(M)=\{\dual_\ell(L)\mid L\in \L \textup{ contains an anchor} \}\cup (\{0\}\times \R).
\end{equation}  
It is clear that $\cell(M)$ is completely determined by $\XiSp$.

\paragraph{0-cells of $\cell(M)$}
Note that for $\alpha\ne \beta$ two anchors, $\dual_p(\alpha)$ and $\dual_p(\beta)$ intersect in $[0,\infty)$ if and only if there exists some $L\in \L$ containing both $\alpha$ and $\beta$; such a line $L$ exists if and only if $\alpha$ and $\beta$ are comparable and have distinct $x$-coordinate.  

\paragraph{Size of $\cell(M)$}
We now bound the number of cells in $\cell(M)$ of each dimension.  As in \cref{Sec:Intro_Complexity}, let $\kappa =\kappa_x \kappa_y$, for $\kappa_x$ and $\kappa_y$ the number of unique $x$ and $y$ coordinates, respectively, of points in $\XiSp$.  
 Clearly, the number of anchors for $\XiSp$ is bounded above by $\kappa$.  Hence the number of lines in $\cell(M)$ is also bounded above by $\kappa$.  
 
 Precise bounds on the number of vertices, edges, and faces in an arbitrary line arrangement are well known, and can be computed by simple counting arguments.  These bounds tell us that the number of vertices, edges, and faces in $\cell(M)$ is each not greater than $\kappa^2$.  

\subsection{Characterization of the 1-Skeleton of \texorpdfstring{$\cell(M)$}{A(M)}}\label{Sec:CriticalLines}
We next give our alternate description of the 1-skeleton of $\cell(M)$.  To do so, we first define the set $\crit(M)$ of \emph{critical lines} in $\L^{\circ}$.\nomenclature[L]{$\L^\circ$}{space of affine lines in $\R^2$ with finite, positive slope}  Here $\L^{\circ}$ denotes the topological interior of $\L$, i.e.\ the set of affine lines in $\R^2$ with positive, finite slope.  

\paragraph{The Push Map}
Note that for each $L\in \bar \L$, the partial order on $\R^2$ restricts to a total order on $L$.  This extends to a total order on $\cupinf{L}$ by taking $v<\infty$ for each $v\in L$.  For $L\in \bar\L$, define the \emph{push map} 
$\push_L:\R^2 \to \cupinf{L}$ by taking \[\push_L(a)=\min\{v\in L\mid a\leq v\}.\]
\nomenclature[push]{$\push_L$}{push map}
Note that:
\begin{itemize} 
\item $\infty\in \im \push_L$ if and only if $L$ is horizontal or vertical.
\item For $a=(a_1,a_2)\in \R^2$ and $\push_L(u)=(v_1,v_2)\in L$, either $a_1=v_1$ or $a_2=v_2$; see \cref{fig:push_map}.
\item For $r<s\in \R^2$, $\push_L(r)\leq \push_L(s)$.
\end{itemize}

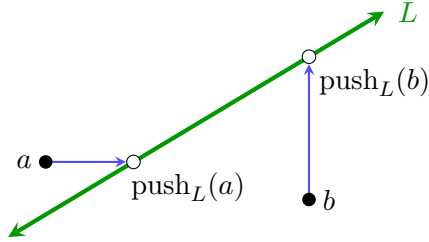
\begin{figure}[ht]
  \begin{center}
    \begin{tikzpicture}
      \draw[<->,ultra thick,color=green!60!black] (0,0) -- (5,3);
       \node[right,color=green!60!black] at (5.05,3) {$L$};   
      
      \coordinate (u) at (0.5,1);
      \coordinate(pu) at (1.667,1);
      \draw[->,thick,blue!70!white] (u) -- (1.567,1);      
      \fill (u) circle (0.09);
      \draw[fill=white] (pu) circle (0.09);
      \node[left] at (0.45,1) {$a$};
      \node[below right] at (1.5,1) {$\push_L(a)$};
      
      \coordinate (v) at (4,0.5);
      \coordinate(pv) at (4,2.4);
      \draw[->,thick,blue!70!white] (v) -- (4,2.3);      
      \fill (v) circle (0.09);
      \draw[fill=white] (pv) circle (0.09);
      \node[right] at (4.05,0.5) {$b$};
      \node[below right] at (pv) {$\push_L(b)$};
    \end{tikzpicture}
  \end{center}
  \caption{Illustration of the push map for lines of positive, finite slope}
  \label{fig:push_map}
\end{figure}

\paragraph{Continuity of Push Maps}
For any $a\in \R^2$, the maps $\{\push_L\}_{L\in \bar \L}$ induce a map $\push_{a}:\bar \L \to \R^2 \cup \infty$, defined by $\push_{a}(L)=\push_L(a)$.  

Recall that we consider $\bar \L$ as a topological space, with the topology the restriction of topology on the affine Grassmannian of 1-D lines in $\R^2$.  

\begin{lemma}\label{Lem:ContinuityOfPush}
For each $a\in \R^2$, $\push_{a}$ is continuous on $\L^{\circ}$.  
\end{lemma}

\begin{proof}
Note that for any $a=(a_1,a_2)$ and $L\in \L^{\circ}$, $\push_{a}(L)$ is the unique intersection of $L$ with \[\{(a_1,y)\mid y\geq a_2\} \cup \{(x,a_2)\mid x\geq a_1\}.\]  From this, the result follows readily.
\end{proof}

\paragraph{Critical Lines}
For $L\in \L^\circ{}$, $\push_L$ induces a totally ordered partition $\XiSp^L$ of $\XiSp$: elements of the partition are restrictions of levelsets of $\push_L$ to $\XiSp$, and the total order on $\XiSp^L$ is the pullback of the total order on $L$.  
This partition is illustrated in \cref{fig:partition}.

\begin{figure}[ht]
  \begin{center}
    \begin{tikzpicture}
      \draw[<->,ultra thick,color=green!60!black] (0,0) -- (5,3);
       \node[right,color=green!60!black] at (5.05,3) {$L$};   
      
      \coordinate (u) at (0.5,0.8);
      \coordinate (pu) at (1.333,0.8);
      \draw[->,thick,blue!70!white] (u) -- (1.233,0.8);
      \fill (u) circle (0.09);
      \draw[fill=white] (pu) circle (0.09);
      \draw[dashed,rounded corners=4pt] (0.3,0.6) rectangle (0.7,1);
      \node[left] at (0.3,0.8) {$\XiSp^L_1$};
      
      \coordinate (v1) at (0.2,1.5);
      \coordinate (v2) at (1.3,1.5);
      \coordinate (pv) at (2.5,1.5);
      \draw[->,thick,blue!70!white] (v1) -- (2.4,1.5);
      \fill (v1) circle (0.09);
      \fill (v2) circle (0.09);
      \draw[fill=white] (pv) circle (0.09);
      \draw[dashed,rounded corners=4pt] (0,1.3) rectangle (1.5,1.7);
      \node[left] at (0,1.5) {$\XiSp^L_2$};
      
      \coordinate (w1) at (3,0.5);
      \coordinate (w2) at (3,1.1);
      \coordinate (pw) at (3,1.8);
      \draw[->,thick,blue!70!white] (w1) -- (3,1.7);
      \fill (w1) circle (0.09);
      \fill (w2) circle (0.09);
      \draw[fill=white] (pw) circle (0.09);
      \draw[dashed,rounded corners=4pt] (2.8,0.3) rectangle (3.2,1.3);
      \node[below] at (3,0.3) {$\XiSp^L_3$};
      
      \coordinate (x1) at (3.8,0.4);
      \coordinate (x2) at (3.8,1.5);
      \coordinate (px) at (3.8,2.28);
      \draw[->,thick,blue!70!white] (x1) -- (3.8,2.18);
      \fill (x1) circle (0.09);
      \fill (x2) circle (0.09);
      \draw[fill=white] (px) circle (0.09);
      \draw[dashed,rounded corners=4pt] (3.6,0.2) rectangle (4,1.7);
      \node[below] at (3.8,0.2) {$\XiSp^L_4$};
      
      \coordinate (y1) at (0.5,2.6);
      \coordinate (y2) at (1.5,2.6);
      \coordinate (y3) at (3.5,2.6);
      \coordinate (py) at (4.333,2.6);
      \draw[->,thick,blue!70!white] (y1) -- (4.233,2.6);
      \fill (y1) circle (0.09);
      \fill (y2) circle (0.09);
      \fill (y3) circle (0.09);
      \draw[fill=white] (py) circle (0.09);
      \draw[dashed,rounded corners=4pt] (0.3,2.4) rectangle (3.7,2.8);
      \node[left] at (0.3,2.6) {$\XiSp^L_5$};
    \end{tikzpicture}
  \end{center}
  \caption{The totally ordered partition $\XiSp^L$ of $\XiSp$.  The $i^{\mathrm{th}}$ element of the partition is labeled as $\XiSp^L_i$.}
  \label{fig:partition}
\end{figure}
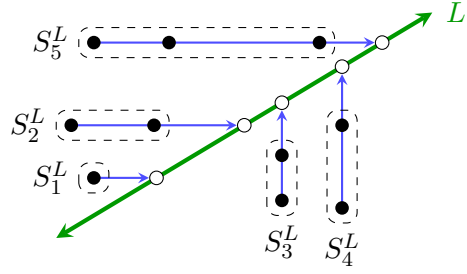

We call $L\in \L^{\circ}$ \emph{regular} if there is an open ball $B \in \L^\circ$ containing $L$ such that $\XiSp^{L}=\XiSp^{L'}$ for all $L'\in B$.  We call $L\in \L^{\circ}$ \emph{critical} if it is not regular.  Let $\crit(M)$ denote the set of critical lines in $\L^{\circ}$.

\begin{theorem}[Characterization of the 1-Skeleton of $\cell(M)$] \label{Thm:ArrangementCharacterization}
The 1-skeleton of $\cell(M)$ is exactly \[\dual_\ell(\crit(M)) \cup \left(\{0\}\times \R\right).\]
\end{theorem}
\begin{figure}[ht]
  \begin{center}
    \begin{tikzpicture}
      \draw[<->,ultra thick,color=green!60!black] (1,-.3) -- (5.0,2.1);
       \node[below,color=green!60!black] at (4.9,1.9) {$L'$};   
      
      \coordinate (u) at (1.5,1);
      \coordinate(pu) at (3.167,1);
      \draw[->,thick,blue!70!white] (u) -- (3.067,1);      
      \fill (u) circle (0.09);
      \draw[fill=white] (pu) circle (0.09);
      \node[left] at (1.45,1) {$u$};

      \coordinate (v) at (4,-.5);
      \coordinate(pv) at (4,1.5);
      \draw[->,thick,blue!70!white] (v) -- (4,1.4);      
      \fill (v) circle (0.09);
      \draw[fill=white] (pv) circle (0.09);
      \node[right] at (4.05,-0.5) {$v$};

      \coordinate(LUB) at (4,1);
      \fill (LUB) circle (0.09);
      \draw[fill=orange] (LUB) circle (0.09);
      \node[right,color=orange] at (4.2,1) {$\lub(u,v)$};

    \end{tikzpicture}
    \hspace{.25in}
        \begin{tikzpicture}
      \draw[<->,ultra thick,color=green!60!black] (2.5,-.3) -- (6.5,2.1);
       \node[below,color=green!60!black] at (6.4,1.9) {$L''$};   
      
      \coordinate (u) at (1.5,1);
      \coordinate(pu) at (4.667,1);
      \draw[->,thick,blue!70!white] (u) -- (4.567,1);      
      \fill (u) circle (0.09);
      \draw[fill=white] (pu) circle (0.09);
      \node[left] at (1.45,1) {$u$};

      \coordinate (v) at (4,-.5);
      \coordinate(pv) at (4,0.6);
      \draw[->,thick,blue!70!white] (v) -- (4,0.5);      
      \fill (v) circle (0.09);
      \draw[fill=white] (pv) circle (0.09);
      \node[right] at (4.05,-0.5) {$v$};

      \coordinate(LUB) at (4,1);
      \fill (LUB) circle (0.09);
      \draw[fill=orange] (LUB) circle (0.09);
      \node[above,color=orange] at (3.4,1.1) {$\lub(u,v)$};

    \end{tikzpicture}
  \end{center}
  \caption{An illustration of the geometric idea behind \cref{Thm:ArrangementCharacterization}: For $u,v,\in \XiSp$ and a line $L'$ lying just above $\lub(u,v)$, $\push_{L'}(u)<\push_{L'}(v)$, whereas for a line $L''$ lying just below $\lub(u,v)$, $\push_{L''}(v)<\push_{L''}(u)$.  Thus any line passing through $\lub(u,v)$ is critical.}
  \label{fig:Criticality_Illustration}
\end{figure}
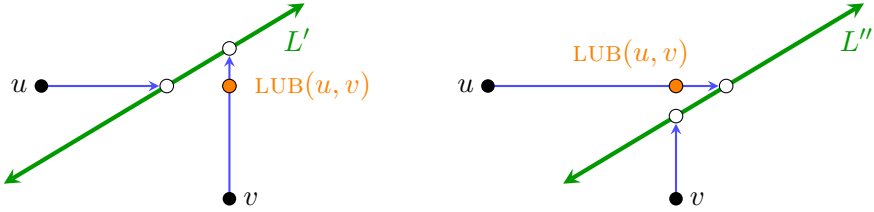

\begin{proof}
In view of the description of the 1-skeleton of $\cell(M)$ given by~\cref{Eq:Description_Of_One_Skel}, proving the proposition amounts to showing that $L\in \L^{\circ}$ is critical if and only if $L$ contains some anchor $w$.

Suppose $L\in \L^\circ$ contains an anchor $w$, and let $u,v\in \XiSp$ be weakly incomparable, with $w=\lub(u,v)\in L$.  Then we must have that $\push_L(u)=\push_L(v)=w$.  Further, it's easy to see that we can find an arbitrarily small perturbation of $L$ so that either $\push_L(u)<\push_L(v)$ or $\push_L(v)<\push_L(u)$.  Thus $L$ is critical; see \cref{fig:Criticality_Illustration} for an illustration in the case that $u$ and $v$ are incomparable.

To prove the converse, assume that $L\in \L^\circ$ does not contain any anchor, and consider distinct $u,v\in \XiSp$ with $\push_L(u)=\push_L(v)$.  Note that $u,v$ must lie either on the same horizontal line or the same vertical line; otherwise $u$ and $v$ would be incomparable and we would have $\push_L(u)=\push_L(v)=\lub(u,v)\in L$.  Assume without loss of generality that $u<v$ and $u,v$ lie on the same horizontal line $H$.  
Then we must also have that $\push_L(u)=\push_L(v)$ lies on $H$. 
Since $v=\lub(u,v)$, $v$ is an anchor. However, $L$ does not contain any anchor, so we must have $v<\push_L(v)$.  

Any sufficiently small perturbation $L'$ of $L$ will also intersect $H$ at a point $p$ to the right of $v$, so that we have $\push_{L'}(u)=\push_{L'}(v)=p$.  Thus for all $L'$ in a neighborhood of $L$, $\push(u)=\push(v)$.  

In fact, since $\XiSp$ is finite, we can choose a single neighborhood $\mathcal{N}$ of $L$ in $\L^\circ$ such that for any $u,v\in \XiSp$ with $\push_L(u)=\push_L(v)$, we have $\push_{L'}(u)=\push_{L'}(v)$ for all $L'\in \mathcal{N}$.  By \cref{Lem:ContinuityOfPush}, choosing $\mathcal{N}$ to be smaller if necessary, may further assume that if $\push_L(u)\ne \push_L(v)$, then $\push_{L'}(u)\ne \push_{L'}(v)$ for all $L'\in \mathcal{N}$.  Thus, the partition $\XiSp^{L'}$ is independent of the choice of $L'\in \mathcal{N}$.  Moreover, by \cref{Lem:ContinuityOfPush} again, the total order on $\XiSp^{L'}$ is also independent of the choice of $L' \in\mathcal{N}$.  Therefore, $L$ is regular. 
\end{proof}

\begin{corollary}\label{Cor:2Cells}
If the duals of $L, L'\in \L$ are contained in the same 2-cell in $\cell(M)$, then $\XiSp^L=\XiSp^{L'}$.
\end{corollary}

\begin{proof}
Each 2-cell is connected and open, so this follows from \cref{Thm:ArrangementCharacterization}.
\end{proof}

\begin{remark}
In fact, \cref{Cor:2Cells} can be strengthened to show that for any $L,L'\in L^\circ$, the duals of $L$ and $L'$ lie in the same cell of $\cell(M)$ if and only if $\XiSp^L=\XiSp^{L'}$.  However, we will not need this stronger result.
\end{remark}

\subsection{The Barcode Templates \texorpdfstring{$\P^e$}{}}\label{Sec:Discrete_Barcodes}
Using \cref{Cor:2Cells}, we now define the barcode templates $\P^e$ stored at each 2-cell $e$ of $\cell(M)$.  This will  complete the definition of the augmented arrangement $\S(M)$.

By \cref{Cor:2Cells}, we can associate to each 2-cell $e$ in $\cell(M)$ a totally ordered partition $\XiSp^e$  of $\XiSp$.  
Let $\XiSp^e_i$ denote the $i^\mathrm{th}$ element of the partition.  
\nomenclature[Se]{$\XiSp^e$}{totally ordered partition of $\XiSp$}

\paragraph{A 1-D Persistence Module at the 2-Cell $e$}
We next use $\XiSp^e$ to define a discrete (i.e., $\Z$-indexed) persistence module $M^e$.  If $M\cong 0$, we take $M^e=0$.  Assume then that $M\not\cong 0$.  First, we define a map $\indx^e:\Nbb\setminus\{0\}\to \R^2$ by \nomenclature[z]{$\indx^e$}{map from positive integers to template points at cell $e$}
\[\indx^e(z)=
\begin{cases}\lub(\cup_{i=1}^z \XiSp^e_i) &\textup{if }1\leq z< |\XiSp^e|,\\
\lub(\XiSp) &\textup{if }z\geq  |\XiSp^e|.
\end{cases}
\]

We define $\PCal^e:=\im \indx^e$, and call this the set of \emph{template points} at cell $e$.  Note that the restriction of $\indx^e$ to $\{1,\ldots,|S^e|\}$ is an injection, so that $|\PCal^e|=|S^e|$.  
\nomenclature[Pe]{$\PCal^e$}{set of template points at cell $e$}

Let $\PCat$ denote the poset category of positive integers.  $\indx^e(y)\leq \indx^e(z)$ whenever $y< z$, so $\indx^e$ induces a functor $\PCat$ to $\RCat^2$, which we also denote $\indx^e$.  Finally, the functor $M\circ \indx^e:\PCat \to \kvect$ extends to a functor $M^e:\ZCat\to \kvect$ (i.e., a $\Z$-indexed persistence module) by taking $M^e_z=0$ whenever $z\leq 0$.

\paragraph{The Definition of Barcode Templates}
Clearly, $M^e$ is \pfd, so it has a well-defined barcode $\B{M^e}$, and it's easy to see that $\B{M^e}$ consists of intervals $[a,b)$, with $a<b\in \{1,\ldots,|\PCal^e|, \infty\}$.  Let us write  $\indx(\infty)=\infty$.

We define $\P^e$ to be a collection of pairs of points in $\PCal^e\times (\PCal^e\cup \infty)$, as follows:
\[\P^e=\{(\indx^e(a),\indx^e(b))\mid [a,b)\in \B{M^e}\}.\]
This completes the definition of $\S(M)$.  

\begin{remark}
$\S(M)$ is completely determined by the fibered barcode $\B{M}$ and the set $\XiSp$.  Indeed, $\cell(M)$ is completely determined by $\XiSp$, and using \cref{Rank_Vs_Fibered_Barcode}, it's easy to see that $\P^e$ is completely determined by $\B{M}$ and $\XiSp$.  The main result of the next section, \cref{Thm:QueriesMain}, shows that conversely, $\S(M)$ completely determines $\B{M}$, and does so in a simple way.
\end{remark}

%% file: VRI_L_Structures.tex
%VRI_L_Structures
\section{Querying the Augmented Arrangement}\label{Sec:QueryingMath}
In the previous section, we defined the augmented arrangement $\S(M)$ of a finitely presented 2-D persistence module $M$.  We now explain how $\S(M)$ encodes the fibered barcode $\B{M}$.  That is, for a given $L\in \bar \L$, we explain how to recover $\B{M^L}$ from $\S(M)$.  The main result of this section, \cref{Thm:QueriesMain}, is the basis of our algorithm for querying $\S(M)$.  We describe the computational details of our query algorithm in \cref{Sec:Representing_And_Querying_Aug_Arrangement}. 

Recall that the procedure for querying $\S(M)$ for a barcode $\B{M^L}$ was discussed for the case of generic lines $L$  in \cref{Sec:Computational_Underpinnings}.  More generally, for any $L\in \bar \L$, querying $\S(M)$ for the barcode $\B{M^L}$ involves two steps.  First, we choose a 2-cell $e$ in $\cell(M)$; if $L\in \L$, then $e$ is a 2-dimensional coface of the cell in containing $\dual_\ell(L)$.  Second, we obtain the intervals of $\B{M^L}$ from the pairs of $\P^e$ by pushing the points in each pair $(a,b)\in \P^e$ onto the line $L$, via the map $\push_L$ of \cref{Sec:CriticalLines}.  

We now describe in more detail the first step of selecting the 2-D coface $e$.

\subsection{Selecting a 2-D Coface \texorpdfstring{$e$ of $L$}{}}\label{Sec:Selecting_A_Coface}
For $L\in \bar \L$, we choose the 2-D coface $e$ of $L$ as follows:
\begin{itemize}
\item If $L\in \L^\circ$, then there exists a 2-cell of $\cell(M)$ whose closure contains $\dual_\ell(L)$; we take $e$ to be any such 2-cell.  
\item If $L$ is horizontal, then the 2-D cofaces of the cell containing $\dual_\ell(L)$ are ordered vertically in $[0,\infty)\times \R$; we take $e$ to be the \emph{bottom} coface.  (Note that $\dual_\ell(L)$ has only one 2-D coface unless $L$ contains an anchor.)
\item For $L$ vertical, say $L$ is the line $x=a$, let $L_0$ be the line in the arrangement $\cell(M)$ of maximum slope, amongst those having slope less than or equal to $a$, if a unique such line exists; if there are several such lines, take $L_0$ to be the one with the largest $y$-intercept.  If such $L_0$ exists, it contains a unique unbounded 1-cell in $a$. We take $e$ to be the 2-cell lying directly above this 1-cell.  If such a line $L_0$ does not exist, then we take $e$ to be the bottom unbounded 2-cell of $\cell(M)$; since we assume all $x$-coordinates of $\XiSp$ to be non-negative, this cell is uniquely defined.
\end{itemize}
The selection of cofaces for several lines is illustrated in \cref{fig:coface}.

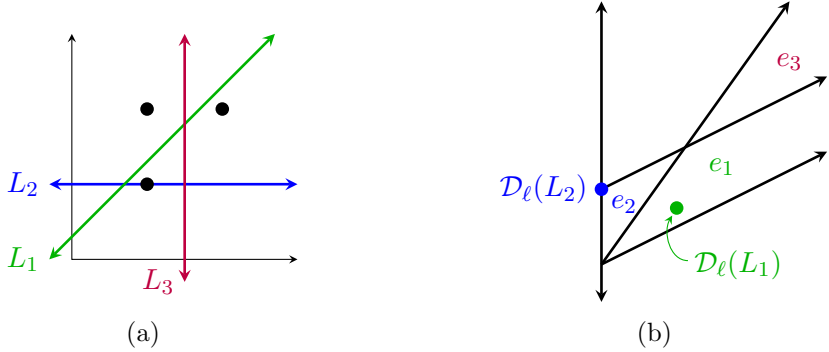
\begin{figure}[ht]
  \begin{center}
    \subcaptionbox{}{
    \begin{tikzpicture}
        \draw[<->] (0,3) -- (0,0) -- (3,0);
        \draw[<->,line width=1pt,color=blue] (-0.3,1) -- (3,1);
        \node[left,color=blue] at (-0.3,1) {$L_2$};
        \draw[<->,line width=1pt,color=green!70!black] (-0.3,0) -- (2.7,3);
        \node[left,color=green!70!black] at (-0.3,0) {$L_1$};
        \draw[<->,line width=1pt,color=purple] (1.5,-0.3) -- (1.5,3);
        \node[below left,color=purple] at (1.5,0) {$L_3$};    
        \foreach \p in {(1,1),(1,2),(2,2)}
          { \fill \p circle (0.09); }
      \end{tikzpicture}
}
 \hspace{.8in}
           \subcaptionbox{}{
      \begin{tikzpicture}
        \draw[<->,line width=1pt] (0,-2.5) -- (0,1.5);
        
        \draw[->,line width=1pt] (0,-2) -- (3,-0.5);
        \draw[->,line width=1pt] (0,-2) -- (2.5,1.5);
        \draw[->,line width=1pt] (0,-1) -- (3,0.5);
        
        \fill[blue] (0,-1) circle (0.09);
        \node[left,color=blue] at (-0.05,-1) {$\dual_\ell(L_2)$};
        \node[color=blue] at (0.3,-1.2) {$e_2$};
        
        \fill[green!70!black] (1,-1.25) circle (0.09);
        \node[right,color=green!70!black] at (1.1,-2) {$\dual_\ell(L_1)$};
        \node[color=green!70!black] at (1.6,-0.7) {$e_1$};
        \draw[->,color=green!70!black] (1.15,-1.95) to [out=180,in=240] (0.94,-1.35);
        
        \node[color=purple] at (2.5,0.7) {$e_3$};
    \end{tikzpicture}
         }
  \end{center}
   \caption{Three anchors are drawn as black dots in (a); the corresponding line arrangement $\cell(M)$ is shown in (b).  For each line $L_i$ in (a), the dual point $\dual_\ell(L_i)$ and the corresponding 2-cell $e_i$, chosen as in \cref{Sec:Selecting_A_Coface}, are shown in (b) in the same color.}
  \label{fig:coface}
\end{figure}

\subsection{The Query Theorem}\label{Sec:Query_Theorem}
Here is the main mathematical result underlying RIVET:
\begin{theorem}[Querying the Augmented Arrangement]\label{Thm:QueriesMain}
For any line $L\in \bar \L$ and $e$ a 2-cell chosen as in \cref{Sec:Selecting_A_Coface}, the barcode obtained by restricting $M$ to $L$ is:
\[\B{M^L}=\{\,[\push_L(a),\,\push_L(b)) \mid (a,b)\in \P^e,\ \push_L(a)<\push_L(b)\,\}.\]
\end{theorem}

Note that if $L$ is such that $\dual_\ell(L)$ lies in a 2-cell of $\cell(M)$, then $\push_L(a)<\push_L(b)$ for all $(a,b)\in \P^e,$ so the theorem statement simplifies for such $L$.  In general, however, it is possible to have $\push_L(a)=\push_L(b)$.   

\begin{proof}[Proof of \cref{Thm:QueriesMain}]
We prove the result for the case $L\in \L^\circ$; the proof for $L$ horizontal or vertical is similar, but easier, and is left to the reader.   The result holds trivially if $M\cong 0$.  Assume then that $M\not\cong 0$.  

Let $e$ be a 2-D coface of the cell containing $\dual_p(L)$.  To keep notation simple, we will write $\push=\push^L$ and $\indx=\indx^e$.

Keeping \cref{Rmk:Extensions_On_L} in mind, we define a 1-D grid $\G:\Z\to L$.  To do so, we first define the restriction of $\G$ to $\{1,\ldots,|\PCal^e|\}$ by taking \[\G(z)=\push\comp \indx(z);\] note that this is non-decreasing. We choose an arbitrary extension of this  to a 1-D grid $\G:\Z\to L$.

By \cref{BarcodesContinuousExtensions}, to finish the proof it suffices to show that $M^L$ is a continuous extension of $M^e$ along $\G$; i.e., that there exists an isomorphism $\alpha:\CoEx_\G(M^e)\to M^L$.

Given $t\in L$, let \[z =\max\, \{w\in \Z\mid \G(w)=\fl_\G(t)\}.\]  Note that we have \[\CoEx_\G(M^e)_t=M^e_z=\begin{cases} M_{\indx(z)}&\textup{if }z\geq 1,\\ 0&\textup{if }z< 1.\end{cases}\]  Note also that 
\begin{equation}\label{eq:gamma_and_G}
\G(z)\leq t<\G(z+1).
\end{equation}
To define the maps $\alpha_t:\CoEx_\G(M^e)_t\to M^L_t$, we will consider separately the three cases 
\[z<1,\qquad 1\leq z< |\PCal^e|,\qquad \text{and} \qquad z\geq |\PCal^e|.\]

For $a\in \R^n$, recall the definition of $\I_M(a)$ from immediately above \cref{Lem:GradesOfInfluenceAndIsomorphisms}, and note that  $\G(1)=\push\comp \indx(1)$ is the minimal element $a$ of $L$ (with respect to the partial order on $\R^2$) such that $\I_M(a)\ne \emptyset$. Hence by \cref{eq:gamma_and_G}, if $z<1$ then $\I_M(t)=\emptyset$, so $M^L_{t}=0$.  Further, if $z<1$, then $\CoEx_\G(M^e)_t=0$.  Thus for $z<1$, we (necessarily) take the isomorphism $\alpha_t:\CoEx_\G(M^e)_t\to M^L_t$ to be the zero map.

For $1\leq z< |\PCal^e|$, \cref{eq:gamma_and_G} and the definition of $\G$ give that 
\[\push\comp \indx(z)\leq t<\push\comp \indx(z+1).\]
This implies \[\indx(z)\leq t\not\geq \indx(z+1),\] so $\I_M(\indx(z))=\I_M(t)$.  Thus $M(\indx(z),t)$ is an isomorphism by \cref{Lem:GradesOfInfluenceAndIsomorphisms}.  Since $\CoEx_\G(M^e)_t=M_{\indx(z)}$, we may regard the map $M(\indx(z),t)$ as an isomorphism $\alpha_t:\CoEx_\G(M^e)_t\to  M_t=M^L_t$.  

For $z\geq |\PCal^e|$, we have that $\indx(z)=\lub(\XiSp)$, so $\I_M(\indx(z))=\XiSp$.  Further, the chain of inequalities \[\indx(z)=\lub(\XiSp) \leq \push\comp \lub(\XiSp)=\G(|\PCal^e|)  \leq \G(z)\leq t\] gives that $\indx(z)\leq t$, and further that $\XiSp=\I_M(\indx(z))\subseteq \I_M(t)\subseteq \XiSp$, so in fact \[\I_M(\indx(z))=\I_M(t)=\XiSp.\]  Then by \cref{Lem:GradesOfInfluenceAndIsomorphisms} again, $M(\indx(z),t)$ is an isomorphism, which as above, can be interpreted as an isomorphism $\alpha_t:\CoEx_\G(M^e)_t\to M^L_t$.  

We have now defined isomorphisms $\alpha_t:\CoEx_\G(M^e)_t\to M^L_t$ for all $t\in L$.  Clearly, these isomorphisms commute with internal maps in $\CoEx_\G(M^e)$ and $M^L$, so they define an isomorphism $\alpha:\CoEx_\G(M^e)\to M^L$, as desired.
\end{proof}

\subsection{Computational Details of Queries}\label{Sec:Queries}
\label{Sec:Representing_And_Querying_Aug_Arrangement}\label{Sec:Data_Structure}

We next explain the computational details of storing and querying $\S(M)$.  We also give the complexity analysis of our query algorithm, proving \cref{Thm:Query_Cost}.  

\paragraph{DCEL Representation of $\cell(M)$}
As noted in the introduction, we represent the line arrangement $\cell(M)$ that underlies $\S(M)$ using the DCEL data structure \cite{deberg2008computation}, a standard data structure for representing line arrangements in computational geometry.  The DCEL consists of a collection of vertices, edges, and 2-cells, together with a collection of pointers specifying how the cells fit together to form a decomposition of $[0,\infty)\times \R$.

\paragraph{Representing the Barcode Templates}
To represent the augmented arrangement $\S(M)$, we store the barcode template $\P^e$ at each 2-cell $e$ in the DCEL representation of $\cell(M)$.  Recall that $\P^e$ is a multiset; this means a pair $(a,b)$ may appear in $\P^e$ multiple times.  We thus store $\P^e$ as a list of triples $(a,b,k)$, where $k\in \Nbb$ gives the multiplicity of $(a,b)$ in $\P^e$.%and a term of the form $(a,b,\cdot)$ appears exactly most once in the list.    

\paragraph{Our Query Algorithm}
Given a line $L\in \bar\L$, the query of $\S(M)$ for $\B{M^L}$ proceeds in two steps.  The first step performs a search for the 2-cell $e$ of $\cell(M)$ specified in \cref{Sec:Selecting_A_Coface}.  
Once the 2-cell $e$ is selected, we obtain $\B{M^L}$ from $\P^e$ by applying $\push_L$ to the endpoints of each pair $(a,b) \in \P^e$.  

Let us describe our algorithm to find the 2-cell $e$ in detail.  In the case that $L\in \L$ (i.e., $L$ is not vertical), it suffices to find the cell of $\cell(M)$ containing $\dual_\ell(L)$.  In general, the problem of finding the cell in a line arrangement containing a given query point is known as the \emph{point location problem}; this is a very well studied problem in computational geometry.  When we need to perform many point location queries on an arrangement, or when we need to perform the queries in real time, it is standard practice to precompute a data structure on the which point locations can be performed very efficiently \cite{toth2004handbook}; this is the approach we take.  For $v$ the number of vertices in the arrangement, there are a number of different strategies which, in time $O(v\log v)$, compute a data structure of size $O(v)$ on which we can perform a point location query in time $O(\log v)$.  $\cell(M)$ has $O(\kappa^2)$ vertices, so computing such a data structure for $\cell(M)$ takes time $O(\kappa^2 \log \kappa)$.  The data structure is of size $O(\kappa^2)$, and the point location query takes time $O(\log \kappa)$.  

In the case that $L$ is vertical, $\dual_\ell(L)$ is not defined, and we need to take a different approach to find the 2-cell $e$.  We precompute a separate (simpler) search data structure to handle this case: Let $\Y$ denote the set of lines $J$ in $\cell(M)$ such that there is no other line in $\cell(M)$ with the same slope lying above $J$.  We compute a 1-D array which contains, for each $J\in \Y$, a pointer to the rightmost (unbounded) 1-cell of $\cell(M)$ contained in $J$, sorted according to slope.  Given $\cell(M)$, computing this array takes $O(n_l)$ time, where $n_l=O(\kappa)$ is the number of anchor lines.  Once the array has been computed, for any vertical line $L$, we can find the appropriate 2-cell $e$ via a binary search over the array.  This takes $\log n_l$ time.

We are now ready to prove our result from the introduction on the cost of querying $\S(M)$.

\begin{proof}[Proof of \cref{Thm:Query_Cost}]
From the discussion above, it is clear that  once we have precomputed the appropriate data structures, finding the cell $e$ takes $O(\log \kappa)$ time.  Each evaluation of $\push_L$ takes constant time, so computing $\B{M^L}$ from $\P^e$ takes total time $O(|\P^e|)$.  Thus, the total time to query $\S(M)$ for $\B{M^L}$ is $O(|\P^e|+\log \kappa)$.  If $\dual_\ell(L)\in e$, then $|\P^e|=|\B{M^L}|$; this gives \cref{Thm:Query_Cost}~(i).  If, on the other hand, $\dual_\ell(L)\not\in e$, then we may not have $|\P^e|=|\B{M^L}|$, but we do have $|\P^e|=|\B{M^{L'}}|$, for $L'$ an arbitrarily small perturbation of $L$ with $\dual_\ell(L')\in e$.  \cref{Thm:Query_Cost}~(ii) follows.
\end{proof}

%% file: VRI_Computing_the_Arrangement.tex
\section{Computing the Arrangement \texorpdfstring{$\cell(M)$}{A(M)}}\label{Sec:Computing_Line_Arrangement}

We now turn to the specification of our algorithm for computing $\cell(M)$.  First, in \cref{Sec:FI_Reps,Sec:Motivativing_Free_Implicit_Reps}, we specify the algebraic objects which serve as the input to our algorithm, and explain how these objects arise from bifiltrations.  

\subsection{Free Implicit Representations of Persistence Modules: The Input to our Algorithm}\label{Sec:FI_Reps}
For $k\geq 1$, define $[k]$ to be the set of integers $\{1,\ldots,k\}$, and let $[0]$ denote the empty set.\nomenclature[K]{$[k]$}{integers $\{1,\ldots,k\}$}  Define a \emph{free implicit representation (\fir)} of an $n$-D persistence module $M$ to be a 
quadruple $\Phi = (\gr_1,\gr_2,D_1,D_2)$, such that:
\nomenclature[FI]{\fir}{free implicit representation}
\nomenclature[Phi]{$\Phi$}{a free implicit representation}
\nomenclature[m]{$m_j$}{dimension of $\Phi$}
\begin{itemize}
\item For $i=1,2$, $\gr_i:[m_i]\to \R^n$ is a function, for some $m_i\geq 0$.
\item $D_1$ and $D_2$ are matrices with coefficients in $K$, of respective dimensions $m_0\times m_1$ and  $m_1\times m_2$, for some $m_0\geq 0$.  Note that either $\MatLeft$ or $\MatRight$ may be an empty matrix if $m_i=0$.
\item If $m_1=0$, then $M\cong 0$.  
\item If $m_1>0$, let $\gr_0:[m_0]\to \R^n$ denote the constant function mapping to the greatest lower bound of $\im \gr_1$.  Then, defining ordered $n$-graded sets $\W_i:=([m_i],\gr_i)$ for $i=0,1,2$, we require that (in the notation of \cref{Sec:Free_Modules_Presentations}), $D_1=[\partial_1]$ and $D_2=[\partial_2]$ are the matrix representations, respectively, of maps \[\partial_1:\free(\W_1)\to \free(\W_0),\quad \partial_2:\free(\W_2)\to \free(\W_1)\] such that $\partial_1\circ \partial_2=0$ and $M\cong \ker \partial_1/\im \partial_2$.
\end{itemize}
We refer to $(m_0,m_1,m_2)$, defined above, as the \emph{dimensions} of $\Phi$, and write $M\cong H\Phi$.  
Note that a presentation of $M$ is an \fir of $M$ in the degenerate case that $m_0=0$, so that $D_1$ is an empty matrix. 

Our algorithm for computing the augmented arrangement of a finitely presented 2-D persistence module $M$ takes as input an \fir of $M$. 

\paragraph{Storing a Free Implicit Representation in Memory}
We store the matrices $\MatLeft$ and $\MatRight$ in a column-sparse data structure, as used in the standard persistent homology algorithm \cite{zomorodian2005computing}.  Columns of $D_j$ are stored in an array of size $m_j$; the $i^{\mathrm{th}}$ column is stored as a list in position $i$ of the array.  We also store $\gr_j(i)$ at position $i$ of the same array.  

\subsection{Motivation: Free Implicit Representations From Finite Filtrations}\label{Sec:Motivativing_Free_Implicit_Reps}
We are interested in studying the $i^{\rm{th}}$ persistent homology module $H_i\F$ of a finite bifiltration $\F$ arising from data.  Our choice to represent a 2-D persistence module via an \fir is motivated by the fact that in practice, one typically has ready access to an \fir of $H_i\F$.  In contrast, we generally do not have direct access to a presentation of $H_i\F$ at the outset, and while there are known algorithms for computing one \cite{kreuzer2005computational}, they are computationally expensive \cite{carlsson2009computing}.

We'll now describe in more detail how \firs of persistent homology modules arise in practice.  

\paragraph{The Chain Complexes of Multi-filtrations}
Recall from \cref{Sec:Multi_D_Persistent_Homology} that an $n$-D filtration is a functor $\F:\RCat^n\to \Simp$ such that the map $\F_a\to \F_b$ is an inclusion whenever $a\leq b\in \R^n$.  $\F$ gives rise to a chain complex $C\F$ of $n$-D persistence modules, given by 
\[\cdots\xrightarrow{\partial_{j+2}} C_{j+1}\F \xrightarrow{\partial_{j+1}} C_j\F \xrightarrow{\partial_{j}} C_{j-1}\F\xrightarrow{\partial_{j-1}}\cdots,\]
where
\begin{itemize}
\item we define $C_j\F$ by taking the vector space $(C_{j}\F)_a$ to be generated by the $j$-simplices of $\F_a$, and taking the map $C_j\F(a,b)$ to be induced by the inclusion $\F_a\hookrightarrow \F_b$,
\item the morphism $\partial_j$ is induced by the $j^{\mathrm{th}}$ boundary maps of the simplicial complexes $\F_a$.
\end{itemize}
Note that $H_j \F\cong\ker \partial_j/\im \partial_{j+1}$.  

\paragraph{1-Critical and Multi-critical Filtrations} 
To explain how \firs arise from finite bifiltrations, it is helpful to first consider a special case: Following \cite{carlsson2009computing}, we define a \emph{1-critical filtration} $\F$ to be a finite $n$-D filtration where for each $s\in \F_{\max}$, $|A(s)|=1$; here, as in \cref{Sec:Multi_D_Persistent_Homology}, $A(s)$ denotes the set of grades of appearance of $s$.  If $\F$ is finite and not 1-critical, we say $\F$ is \emph{multi-critical}.  

Bifiltrations arising in TDA applications are often (but not always) 1-critical.  For example, for $P$ a finite metric space and $\gamma:P\to \R$ a function, the bifiltration $\Rips(\gamma)$ of \cref{Sec:Intro_Multi_D_PH} is 1-critical, but the filtration $\BRips(P)$ of that section generally is not 1-critical.  

It's easy to see that if $\F$ is 1-critical, then each $C_j\F$ is free; we have an obvious isomorphism between $C_j\F$ and $\free[\F_j]$, for $\F_j$ the $n$-graded set given by
\[\F_j\equiv \{(s,A(s))\mid s\textup { a }j\textup{-simplex in }\F_{\max}\}.\]  
Thus, choosing an order for each $\F_j$, the boundary map $\partial_j:C_j\F\to C_{j-1}\F$ can be represented with respect to $\F_j$, $\F_{j-1}$ by a matrix $[\partial_j]$ with coefficients in the field $K$, as explained in \cref{Sec:Free_Modules_Presentations}; $[\partial_j]$ is exactly the usual matrix representation of the $j^{\mathrm{th}}$ boundary map of $\F_{\max}$.

\paragraph{Free Implicit Representations of $H_i\F$ in the 1-Critical Case}
The ordered $n$-graded sets $\F_j$ and matrices $[\partial_j]$ determine the chain complex $C\F$, and hence each of the homology modules $H_j\F$ up to isomorphism.  In fact, for $\O_j:\F_j\to \{1,2,\ldots,|\F_j|\}$ the total order on $\F_j$, we have that \[(\gr_{\F_{j}}\comp \O_j^{-1},\gr_{\F_{j+1}}\comp \O_{j+1}^{-1},[\partial_j],[\partial_{j+1}],)\] is an \fir of $H_j\F$, for any $j\geq 0$.

\paragraph{Free Implicit Representations of $H_i\F$ in the Multi-Critical case}
For $\F$ a multi-critical filtration, the modules $C_j\F$ are not free.  Nevertheless, as explained in \cite{chacholski2012combinatorial}, there is an easy construction of an \fir
\[\Phi=(\gr_1,\gr_2,D_1,D_2)\] of $H_j \F$, generalizing the construction given above in the 1-critical setting.
Letting \[l_j=\sum_{\substack{s\text{ a $j$-simplex}\\ \text{in } \F_{\max}}} |A(s)|,\] 
the dimensions $(m_0,m_1,m_2)$ of $\Phi$ satisfy the following bounds when $\F$ is a bifiltration:
\begin{equation}\label{Eq:Multicritical_Bounds}
m_0 \leq l_{j-1}, \qquad m_1 \leq l_j,\qquad   m_2\leq l_{j+1}+l_j-1;
\end{equation}
see \cite{chacholski2012combinatorial} for details.

\paragraph{Computation of the Free Implicit Representation of a Bifiltration}
As explained in \cref{Sec:Multi_D_Persistent_Homology}, we can store a finite simplicial bifiltration in memory as a simplicial complex, together with a list of grades of appearance for every simplex.  Let $\F$ be a finite (one-critical or multi-critical) bifiltration of size $l$.  It's not hard to show that given this input, for any $i\geq 0$ we can compute the \fir of $H_i\F$ described above in time $O(l \log l)$.

\originalparagraph{One Homology Index at a Time, or All Homology Indices at Once?}
The standard algorithms for computing persistence barcodes of a 1-D filtration $\F$ compute $\B{H_i\F}$ for all $i$ up to a specified integer, in a single pass.  When one is interested in the barcodes at each homology index, this is more efficient than doing the computations one index at a time, because the computations of $\B{H_i\F}$ and $\B{H_{i+1}\F}$ share some computational work.

In contrast, the algorithm described in this paper, and implemented in the present version of RIVET, computes $\S(H_i\F)$ of a finite bifiltration $\F$, for a single choice of $i\geq 0$.  This approach allows us to save computational effort when we are only interested in a single homology module, and seems to be the more natural approach when working with multi-critical filtrations.

That said, within the RIVET framework, for a 1-critical bifiltration $\F$ one can also handle all persistence modules $H_i\F$, for $i$ up to a specified integer, in a single pass.\footnote{The natural way to do this is not to compute $\S(H_i\F)$ for each $1\leq i\leq l$, but rather to compute the single augmented line arrangement $\S(\bigoplus_{i=1}^l H_i\F)$, labeling the intervals of the discrete barcode at each 2-cell of $\cell(\bigoplus_{i=1}^l H_i\F)$ by homology degree, so that a query of $\S(\bigoplus_{i=1}^l H_i\F)$ provides the homology degree of each interval of $\B{\bigoplus_{i=1}^l (H_i\F)^L}$.  This ``labeled" variant of $\S(\bigoplus_{i=1}^l H_i\F)$ can be computed using essentially the same algorithm as presented in this paper for computation of a single augmented arrangement.

When $\F$ is not 1-critical, to compute $\S(\bigoplus_{i=1}^l H_i\F)$, we need to first replace $C(\F)$ by a chain complex of free 2-D persistence modules.  As noted in \cite{carlsson2009computing}, this can be done via a mapping telescope construction, though this may significantly increase the size of $C(\F)$.}

\subsection{Computation of Betti Numbers}\label{sec:ComputingBettiNumbers}
Our first step in the computation of $\cell(M)$ is to compute $\XiSp=\supp\xi_0(M) \cup \supp \xi_1(M)$.  Since RIVET also visualizes the Betti numbers of $M$ directly, we  choose to compute this by fully computing $\xi_0$ and $\xi_1$.  (In any case, we do not know of any algorithm for computing $\XiSp$ that is significantly more efficient than our algorithm for fully computing $\xi_0$ and $\xi_1$.)

\paragraph{Computing Betti Numbers of $\Z^2$-indexed Persistence Modules}
We can define a \fir $\Phi=(\gr_1,\gr_2,\MatLeft,\MatRight)$ of a $\Z^2$-indexed persistence module $\DiscreteMod$ just as we have for an $\R^2$-indexed persistence module above; in this case the functions $\gr_1$ and $\gr_2$ take values in $\Z^2$.  

In a companion article \cite{lesnick2014betti}, we show how to fully compute $\xi_0(\DiscreteMod)$, $\xi_1(\DiscreteMod)$, and $\xi_2(\DiscreteMod)$, given $\Phi$; the algorithm runs in time runs in time $O(m^3)$, for $(m_0,m_1,m_2)$ the dimensions of $\Phi$ and $m=m_0+m_1+m_2$.

One way to compute the bigraded Betti numbers of $M$, quite standard in computational commutative algebra, is to compute a free resolution for $M$ \cite{la1998strategies}.  However, this gives us more than we need for our particular application.  Instead of following this route, our algorithm computes the Betti numbers via carefully scheduled column reductions on matrices, taking advantage of a well-known characterization of Betti numbers in terms of the homology of Kozul complexes \cite{eisenbud2005geometry}.  

\paragraph{Computing Betti Numbers of $\R^2$-indexed Persistence Modules}
In fact, our algorithm for computing Betti numbers in the discrete setting can also be used to compute the bigraded Betti numbers of a finitely presented $\R^2$-indexed persistence module.  Indeed, as we will now explain, a \fir  $(\gr_1,\gr_2,\MatLeft,\MatRight)$ of an $\R^2$-indexed persistence module $M$ induces a discrete \fir $\Phi=(\gr_1' ,\gr_2',\MatLeft,\MatRight)$ and an injective 2-D grid $\G:\Z^2\to \R^2$ such that $M$ is a continuous extension of $M$ along $H\Phi$.  (The matrices $D_1, D_2$ are the same in the two free implicit representations).  Given this, we can compute the multigraded Betti numbers of $H\Phi$ using the algorithm of \cite{lesnick2014betti} and deduce the multigraded Betti numbers of $M$ from those of $H\Phi$, by \cref{Prop:BettiNumsContinuousExtensions}.

The construction of $\Phi$ and the 2-D grid $\G$ are simple: Let $\orderedCoords_x$ (respectively, $\orderedCoords_y$) denote the ordered set of unique $x$-coordinates ($y$-coordinates) of elements of $\im \gr_1 \cup \im \gr_2$.  Let $n_x=|\orderedCoords_x|$ and $n_y=|\orderedCoords_y|$.  Let $\ordfun_x:[n_x]\to \orderedCoords_x$ be the bijection sending $i$ to the $i^{\mathrm{th}}$ element of $\orderedCoords_x$; define $\ordfun_y:[n_y] \to \orderedCoords_y$ analogously.  We choose $\G: \Z^2\to \R^2$ to be an arbitrary extension of \[\ordfun_x\times \ordfun_y: [n_x]\times [n_y] \to \orderedCoords_x\times \orderedCoords_y.\] 
For $j=1,2$, we define $\gr_j'$ to be the $\Z^2$-valued function $(\ordfun_x\times \ordfun_y)^{-1} \circ \gr_j$.  
 
We leave to the reader the easy check that $M$ is a continuous extension of $H\Phi$ along $\G$.

\subsection{Computation and Storage of Anchors and Template Points}\label{Sec:Sparse_Data_Structure_For_T}
Recall from \cref{sec:arr_def} that an \emph{anchor} is the least upper bound of a weakly incomparable pair of points in $\XiSp$, and that the set of anchors determines the line arrangement $\cell(M)$.  In this section, we will let $A$ denote the set of anchors.  
To compute the line arrangement $\cell(M)$, we need to first compute a list $\anchors$
of all elements of $A$.

Moreover, our algorithm for computing the barcode templates, described in \cref{Sec:Computing_Additional_Data_at_Faces}, requires us to represent the set
\[\PCal:=\bigcup_{e\textup{ a 2-cell in }\A(M)}\PCal^e\]
\nomenclature[P]{$\PCal$}{set of all template points}
of all template points using a certain sparse matrix data structure.  It's easy to see that \[\PCal=\ancset \cup \XiSp.\]  We will see that because of this, it is convenient to compute the list $\anchors$ and the sparse matrix representation of $\PCal$ at the same time.

In this section, we specify our data structure for $\PCal$ and describe our algorithm for simultaneously computing both this data structure and the list $\anchors$.
 
\paragraph{Sparse Matrix Data Structure for \protect$\PCal$}
Note that $\XiSp\subseteq \im(\ordfun_x\times \ordfun_y)$.  Given this, it's easy to see that also $\PCal\subseteq \im (\ordfun_x\times \ordfun_y)$.  Thus, to store $\PCal$, it suffices to store $(\ordfun_x\times \ordfun_y)^{-1}(\PCal)$ and the maps $\ordfun_x$ and $\ordfun_y$.  We store $\ordfun_x$ and $\ordfun_y$ in two arrays of size $n_x$ and $n_y$, and we store $(\ordfun_x\times \ordfun_y)^{-1}(\PCal)$ in a sparse matrix $\xiSuppMat$ of size $n_x\times n_y$.  The triple ($\xiSuppMat$, $\ordfun_x$, $\ordfun_y$) is our data structure for $\PCal$.

Henceforth, to keep notation simple, we will assume that $\ordfun_x$ and $\ordfun_y$ are the identity maps on $[n_x]$ and $[n_y]$ respectively, so that $(\ordfun_x\times \ordfun_y)^{-1}(\PCal)=\PCal$.

Let us now describe $\xiSuppMat$ in detail.  An example $\xiSuppMat$ is shown in \cref{fig:xisupportmatrix}.  
Each element $u\in \PCal$ is represented in $\xiSuppMat$ by a quintuple 
\[(u,p_l,p_d,\partialLevelSet{1}{u},\partialLevelSet{2}{u}).\]
In this quintuple, $p_l$ and $p_d$ are pointers (possibly null); $p_l$ points to the element of $\PCal$ immediately to the left of $u$, and $p_d$ points to the element of $\PCal$ immediately below $u$. The objects $\partialLevelSet{1}{u}$, $\partialLevelSet{2}{u}$ are lists used in the computation of the barcode templates $\P^e$.  Initially, these lists are empty.  We discuss them further in \cref{Sec:Computing_Additional_Data_at_Faces}.  

The data structure $\xiSuppMat$ also contains an array \rows of pointers, of length $n_y$; the $i^{\mathrm{th}}$ entry of \rows points to the rightmost element of $\PCal$ with $y$-coordinate $i$.

\begin{figure}[ht]
  \begin{center}
    \begin{tikzpicture}[scale=0.55]
        % column numbers
        \foreach \i in {1,...,10}
        {
            \node at (\i-0.5,8) {\small{\sf \i}};
        }
    
        % rows array
        \draw (11,0) -- (11,7) -- (12,7) -- (12,0);
        \foreach \i in {1,...,7}
        {
            \draw (11,\i-1) -- (12,\i-1);
            \node[right] at (12,\i-0.5) {\small{\sf \i}};
        }
        \node[below] at (11.5,-0.1) {\sf rows};

        % interior nodes
        \foreach \p in {(0,4),(1,2),(1,5),(2,3),(3,5),(4,0),(4,6),(5,4),(7,3),(7,4),(8,0),(8,4),(9,2),(9,5)}
	    \fill[blue!35]  \p+(0.15,0.15) rectangle +(0.85,0.85);
	\foreach \p in {(1,4),(1,5),(2,4),(2,5),(3,5),(4,2),(4,3),(4,4),(4,5),(4,6),(5,4),(5,5),(5,6),(7,3),(7,4),(7,5),(7,6),(8,0),(8,2),(8,3),(8,4),(8,5),(8,6),(9,2),(9,3),(9,4),(9,5),(9,6)}
            \draw \p+(0.15,0.15) rectangle +(0.85,0.85);
            
        % vertical pointers
        \foreach \x/\a/\b in {1/5/4, 1/4/2, 2/5/4, 2/4/3, 4/6/5, 4/5/4, 4/4/3, 4/3/2, 4/2/0, 5/6/5, 5/5/4, 7/6/5, 7/5/4, 7/4/3, 8/6/5, 8/5/4, 8/4/3, 8/3/2, 8/2/0, 9/6/5, 9/5/4, 9/4/3, 9/3/2}
        {
            \draw[thick,->] ($(\x,\a)+(0.5,0.5)$) -- ($(\x,\b)+(0.5,0.85)$);
            \fill ($(\x,\a)+(0.5,0.5)$) circle (0.12);
        }

          % horizontal pointers
        \foreach \y/\a/\b in {0/11/8, 0/8/4, 2/11/9, 2/9/8, 2/8/4, 2/4/1, 3/11/9, 3/9/8, 3/8/7, 3/7/4, 3/4/2, 4/11/9, 4/9/8, 4/8/7, 4/7/5, 4/5/4, 4/4/2, 4/2/1, 4/1/0, 5/11/9, 5/9/8, 5/8/7, 5/7/5, 5/5/4, 5/4/3, 5/3/2, 5/2/1, 6/11/9, 6/9/8, 6/8/7, 6/7/5, 6/5/4}
        {
            \draw[thick,->] ($(\a,\y)+(0.5,0.5)$) -- ($(\b,\y)+(0.85,0.5)$);
            \fill ($(\a,\y)+(0.5,0.5)$) circle (0.12);
        }
    \end{tikzpicture}
  \end{center}
  \caption{Example of $\xiSuppMat$ for $n_x=10$ and $n_y=7$. Each element of $\PCal$ is represented by a square. Shaded squares represent elements of $\XiSp$, and squares with solid borders represent anchors. Each entry contains pointers to the next entries down and to the left; non-null pointers are illustrated by arrows. The lists $\partialLevelSet{1}{u}$ and $\partialLevelSet{2}{u}$ stored at
each $u \in \PCal$ are not shown.}
  \label{fig:xisupportmatrix}
\end{figure}
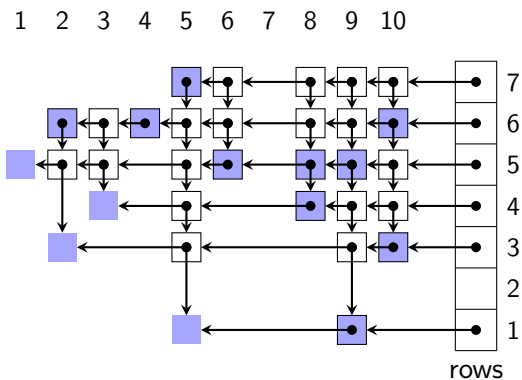

\paragraph{Computation of $\xiSuppMat$ and $\anchors$}
Our algorithm for computing Betti numbers, given in \cite{lesnick2014betti}, computes $\xi_0(M)(u)$ and $\xi_1(M)(u)$ at each bigrade $u\in [n_x]\times [n_y]$ by iterating through  $[n_x]\times [n_y]$ in lexicographical order.  As we iterate through $[n_x]\times [n_y]$, it is easy to also compute both $\xiSuppMat$ and the list $\anchors$.  Let us explain this in detail.  

Upon initialization, $\anchors$ is empty, $\xiSuppMat$ contains no entries, and all pointers in \rows are null.  We create a temporary array \columns of pointers, of length $n_x$, with each pointer initially null.
\begin{sloppypar}
At each $u\in [n_x]\times [n_y]$, once the Betti numbers at $u$ have been computed, if $u\in\ancset$, then we add $u$ to the list $\anchors$.  If $u \in \XiSp \union \ancset=\PCal$, then we add the quintuple $(u,\mathsf{rows}[u_2],\mathsf{columns}[u_1],[\,],[\,])$ to $\xiSuppMat$, and set $\mathsf{rows}[u_2]$ and $\mathsf{columns}[u_1]$ to both point to $u$.  These updates to \columns and \rows ensure that: \end{sloppypar}
\begin{itemize}
\item For each $i\in [n_y]$, $\mathsf{rows}[i]$ always points to the rightmost entry with $y$-coordinate $i$ added to $\xiSuppMat$ thus far;
\item For each $j\in [n_x]$, $\mathsf{columns}[j]$ always points to the topmost entry with $x$-coordinate $j$ added to $\xiSuppMat$ thus far.  
\end{itemize}

It remains to explain how we determine whether $u\in \ancset$.  Note that $u\in \ancset$ if and only if at least one of the following two conditions holds when we visit $u$:
\begin{enumerate}
\item Both $\mathsf{rows}[u_2]$ and $\mathsf{columns}[u_1]$ are not null.
\item $u \in S$ and \emph{either} $\mathsf{rows}[u_2]$ or $\mathsf{columns}[u_1]$ is not null.
\end{enumerate}
Using this fact, we can check whether $u\in \ancset$ in constant time.  

Beyond the $O(m^3)$ time required to compute $S$, this algorithm for computing $\xiSuppMat$ and $\anchors$ takes $O(n_x\times n_y)=O(m^2)$ time.

\subsection{Building the Line Arrangement}\label{sec:BuildingTheLineArrangement}
Recall that the anchors of $M$ correspond under point-line duality to the lines in the arrangement $\cell(M)$.  Thus, once the list of anchors has been determined, we are ready to build the DCEL representation of $\cell(M)$.  For this, our implementation of RIVET uses the well-known Bentley-Ottmann algorithm \cite{deberg2008computation}, which constructs the DCEL representation of a line arrangement with $n$ lines and $k$ vertices in time $O((n + k) \log n)$.  Since our arrangement contains $O(\kappa)$ lines and $O(\kappa^2)$ vertices, the algorithm requires $O(\kappa^2 \log \kappa)$ elementary operations.  

As explained in \cref{sec:arr_def}, the number of cells in $\cell(M)$ is $O(\kappa^2)$.  The size of the DCEL representation of any arrangement is of order the number of cells in the arrangement, so the size of the DCEL representation of $\cell(M)$ is also $O(\kappa^2)$.

\begin{remark}
The $\kappa^2\log \kappa$ term in the bound of \cref{SimpleAugArrComplexity}~(ii) arises from our use of the Bentley-Ottmann algorithm.  There are asymptotically faster algorithms for constructing line arrangements that would give a slightly smaller term in the bound of \cref{SimpleAugArrComplexity}~(ii)---in fact, we can remove the $\log \kappa$ factor.  However, the Bentley-Ottmann algorithm, which is relatively simple and performs well in practice, is a standard choice.
\end{remark}  

Each 1-cell $e$ in $\cell(M)$ lies on the line dual to some anchor $\alpha$.  In our DCEL representation of $\cell(M)$, we store a pointer at $e$ to the entry of $\xiSuppMat$ corresponding to $\alpha$.

\paragraph{Numerical Considerations}
Line arrangement computations, as with many computations in computational geometry, are notoriously sensitive to numerical errors that arise from floating-point arithmetic.  
Much effort has been invested in the development of smart arithmetic models for computational geometry which allow us to avoid the errors inexact arithmetic can produce, without giving up too much computational efficiency \cite{cgal:geometrykernel}.  Because exact arithmetic is generally far more computationally expensive than floating-point arithmetic, these models typically take a hybrid approach, relying on floating-point arithmetic in cases where the resulting errors are certain to not cause problems, and switching over to exact arithmetic for calculations where the errors could be problematic.  
Our implementation of RIVET relies on a simple such hybrid model specially tailored to the problem at hand.
 

%% file: VRI_Computing_Data_At_Faces.tex
\section{Computing the Barcode Templates}\label{Sec:Computing_Additional_Data_at_Faces}

Once we have found all anchors and constructed the line arrangement $\cell(M)$, we are ready to complete the computation of $\S(M)$ by computing the barcode templates $\P^e$ for all 2-cells $e$ of $\cell(M)$.  This section describes our core algorithm for this.  In \cref{Sec:Barcodes_Scratch}, we describe a refinement of the algorithm which performs significantly faster in practice. 

The input to our algorithm consists of three parts:
\begin{enumerate}
\item A \fir $\Phi$ of $M$, represented in the way described in \cref{Sec:FI_Reps},
\item Our sparse matrix representation $\xiSuppMat$ of the set $\PCal$ of all template points,
\item A DCEL representation of the line arrangement $\cell(M)$.
\end{enumerate}
Recall that $\Phi$ is given to us as input to our algorithm for computing $\S(M)$; the computation of $\xiSuppMat$ and $\cell(M)$ from $\Phi$ has already been described above.

Recall from \cref{Sec:Discrete_Barcodes} that 
\[\P^e=\{(\indx^e(y),\indx^e(z))\mid [y,z)\in \B{M^e}\}.\]
Thus, to compute $\P^e$ at each 2-cell $e$, it suffices to compute the pair $(\B{M^e},\indx^e)$ at each 2-cell $e$.
This is essentially what our algorithm does, though it turns out to be unnecessary to explicitly store either $\B{M^e}$ or $\indx^e$ at any point in the computation.

Note that $M\cong 0$ if and only if $S=\emptyset$, if and only if each $\P^e=\emptyset$.  Thus, we may assume that $M\not\cong 0$.

\subsection{Trimming the Free Implicit Representation}\label{Sec:Trimming}
Let \[\sbx=\{y\in \R^2\mid y\leq \lub S\}.\]   
\nomenclature[box]{$\sbx$}{quadrant in $\R^2$ with upper right corner $\lub S$}
We say a \fir $\Phi'=(\gr'_1,\gr'_2,D'_1,D'_2)$ for $M$ is \emph{trimmed} if \[\im \gr'_1\cup \im \gr'_2\in \sbx.\]

As a preliminary step in preparation for the computation of the barcode templates, if $\Phi$ is not already trimmed we replace $\Phi$ with a smaller trimmed \fir.  While it is possible to work directly with an untrimmed \fir to compute the barcode templates, it is more efficient to work with a trimmed one.  In addition, assuming that our \fir is trimmed allows for some simplifications in the description of our algorithm.

For $j=1,2$, let $l_j=|\gr_j^{-1}(\sbx)|$.  There is a unique order-preserving map \[\eta_j:[l_j]\to \gr_j^{-1}(\sbx).\]  We define $\gr'_j=\gr_j\comp \eta_j$.  We define $D'_1$ to be the submatrix of $D_1$ whose columns correspond to elements of $\im \eta_1$, and we define $D'_2$ to be the submatrix of $D_2$ whose rows and columns correspond to elements of $\im \eta_1$ and $\im \eta_2$, respectively.
Let \[\Phi'=(\gr'_1,\gr'_2,D'_1,D'_2).\]

\begin{proposition}\label{Trim:Prop}
$H(\Phi')\cong M$.
\end{proposition}

\begin{proof}
Associated to $\Phi$ and $\Phi'$ we have respective chain complexes of free modules 
\begin{align*}
&F_2\xrightarrow{\partial_2} F_1\xrightarrow{\partial_1} F_0\\ 
&F'_2\xrightarrow{\partial'_2} F'_1\xrightarrow{\partial_1'} F'_0,
\end{align*}
with $\ker \partial_1/\im \partial_2=H(\Phi)\cong M$ and $\ker \partial'_1/\im \partial'_2=H(\Phi')$, as in the definition of a \fir; further, since $\Phi'$ is a trimming of $\Phi$, we have obvious maps $\psi^j: F'_j\hookrightarrow F_j$ for $j=0,1,2$, making the following diagram commute:
\[
\begin{tikzcd}
F_2 \arrow{r}{\partial_2}		               &F_1 \arrow{r}{\partial_1}                         & F_0\\
F'_2 \arrow[swap]{r}{\partial'_2}\ar[hookrightarrow]{u}{\psi^2}    &F'_1 \arrow[swap]{r}{\partial'_1}\arrow[hookrightarrow]{u}{\psi^1}       & F'_0\arrow[hookrightarrow]{u}{\psi^0}
\end{tikzcd}
\]
The maps $\psi^i$ induce a map $\psi:H(\Phi')\to H(\Phi)\cong M$.  It's easy to see that $\psi_a$ is an isomorphism for $a\in \sbx$.  To finish the proof, it remains to check that $\psi_b$ is also an isomorphism for $b\not\in \sbx$.  For $b\not \in \sbx$, there is a unique element $a\in \sbx$ minimizing the distance to $b$.  Note that $a<b$.  By commutativity, it suffices to see that $M(a,b)$ is an isomorphism and $H(\Phi')(a,b)$ is an isomorphism.  \cref{Lem:GradesOfInfluenceAndIsomorphisms} gives that $M(a,b)$ is an isomorphism, and since $\im \gr'_1\cup \im \gr'_2\in \sbx$, it's easy to see directly that $H(\Phi')(a,b)$ is an isomorphism.
\end{proof}

Clearly, $\Phi'$ is trimmed.  Given $\xiSuppMat$ and our column-sparse representation of $\Phi$, we can compute $\Phi'$ from $\Phi$ in $O(m+d)$ time, where $d=O(m^2)$ is the number of non-zero entries of $D_2$.  

Henceforth we will assume that $\Phi$, the FI-rep of $M$ given as input to our algorithm, is trimmed.

\subsection{\texorpdfstring{$RU$}{RU}-Decompositions and Computation of Persistence Barcodes}\label{sec:ComputationPersistenceBarcodes}
To prepare for a description of our algorithm for computing the barcode templates, we begin with some preliminaries on the computation of persistence barcodes.  There is a large and growing body of work on this topic; see \cite{otter2015roadmap} for a recent overview  with an emphasis on publicly available software.  We restrict attention here to what is needed to explain our algorithm.  

The \emph{standard algorithm} for computing persistence barcodes was introduced in \cite{zomorodian2005computing}, building on ideas in \cite{edelsbrunner2002topological}; see also \cite{edelsbrunner2010computational} for a succinct description of the algorithm, together with implementation details.

The algorithm takes as input a \fir $\Phi=(\gr_1,\gr_2,\MatLeft,\MatRight)$ of a 1-D persistence module, with $\gr_1$ and $\gr_2$ non-decreasing, and returns $\B{H\Phi}$.  Of course, in applications $\Phi$ typically comes from the chain complex of a 1-D filtration, with the simplices in each dimension ordered according to their grade of appearance.

The algorithm is a variant of Gaussian elimination; it performs column additions to construct certain factorizations of $\MatLeft$ and $\MatRight$, from which the barcode $\B{M}$ can be read off directly.  Let us explain this in more detail, drawing on ideas introduced in \cite{cohen2006vines}: Let $R$ be an $m\times n$ matrix, and for $j$ the index of a non-zero column of $R$, let $\low(R,j)$ denote the maximum row index of a non-zero entry in column $j$ of $R$.  We say $R$ is \emph{reduced} if $\low(R,j)\ne \low(R,j')$ whenever $j\ne j'$ are the indices of non-zero columns in $R$.

The standard persistence algorithm yields a decomposition $D=R U$ of any matrix $D$ with coefficients in the field $K$, where $R$ is reduced, and $U$ is an upper-triangular matrix.  For $D$ an $r\times s$ matrix, the algorithm runs in time $O(rs^2)$.

We define an $RU$-decomposition of $\Phi$ simply to be a pair of $RU$-decompositions $D_1=R_1 U_1$ and $D_2=R_2 U_2$.  We can read $\B{M}$ off of $R_1$ and $R_2$.  To explain this, suppose $\Phi$ is of dimensions $(m_0,m_1,m_2)$, and 
let $R_i(*,j)$ denote the $j^{\mathrm{th}}$ column of $R_i$.  Define 
\begin{align*}
\pairs(\Phi)&:=\{(\low(R_2,j),j)\in [m_1]\times [m_2] \mid R_2(*,j)\ne 0\}\\
\ess(\Phi)&:=\{j\in [m_1] \mid R_1(*,j)=0\textup{ and }j\ne \low(R_2,k)\textup{ for any column $k$ of $R_2$}.\}
\end{align*}
While the $RU$-decomposition of a matrix is not unique, it is shown in \cite{cohen2006vines} that $\pairs(\Phi)$ and $\ess(\Phi)$ are independent of the choice of $RU$-decomposition of $\Phi$.  
\begin{theorem}[\cite{zomorodian2005computing,edelsbrunner2010computational}]\label{BarcodeFromReducedMat}
\[\B{M}=\big\{ [\gr_1(j),\gr_2(k)) \mid (j,k)\in \pairs(\Phi) \big\}\union \big\{ [\gr_1(j),\infty) \mid j\in \ess(\Phi)\big\}.\]
\end{theorem}

\paragraph{Vineyard Updates to Barcode Computations} 
Suppose that $D$ is a $d_1\times d_2$ matrix, and that $D'$ is obtained from $D$ by transposing either two adjacent rows or two adjacent columns of $D$. \cite{cohen2006vines} introduces an algorithm, known as the \emph{vineyard algorithm}, for updating an $RU$-decomposition of $D$ to obtain an $RU$-decomposition of $D'$ in time $O(d_1+d_2)$.  
This algorithm is an essential subroutine in our algorithm for computing the barcode templates.  

\subsection{Permutations of Free Implicit Representations}
As mentioned above, the standard persistence algorithm takes as input a \fir of a 1-D persistence module $\Phi=(\gr_1,\gr_2,\MatLeft,\MatRight)$, with $\gr_1$ and $\gr_2$ non-decreasing.  The reason we need $\gr_1$ and $\gr_2$ to be non-decreasing is that the formula \cref{BarcodeFromReducedMat} for reading the barcode off of the $RU$-decomposition holds only under this assumption on $\gr_1$ and $\gr_2$.  

Now suppose that we are given a 1-D \fir $\Phi=(\gr_1,\gr_2,\MatLeft,\MatRight)$ with either $\gr_1$ or $\gr_2$ \emph{not} non-decreasing.  How can we modify $\Phi$ to obtain a \fir $\Phi'$ of $H(\Phi)$ with non-decreasing grade functions, so that we can read the barcode of $\H(\Phi)$ off of an $RU$-decomposition of $\Phi'$?  We now answer this question. 

It is easy to check that the following lemma holds. 

\begin{lemma}\label{Permutation_Lemma}
Suppose $\Phi=(\gr_1,\gr_2,\MatLeft,\MatRight)$ is a \fir of an $n$-D persistence module $M$ of dimensions $(m_0,m_1,m_2)$.  Then for $\sigma_1$ and $\sigma_2$ any permutations on $[m_1]$ and $[m_2]$, respectively, and $P_1$ and $P_2$ the corresponding permutation matrices, we have that
\[\Phi':= (\gr_1\comp\sigma_1,\gr_2\comp\sigma_2,D_1 P_1,P^T_1 D_2 P_2)\]
is also a \fir of $M$.  
\end{lemma}

Thus, in the case that $M$ is 1-D, finding a \fir of $M$ with non-decreasing grade functions amounts to finding permutations $\sigma_1$ and $\sigma_2$ as above with $\gr_1\comp\sigma_1$ and $\gr_2\comp\sigma_2$ non-decreasing, and applying the corresponding permutations to the rows and columns of $P_1$ and $P_2$.  

\paragraph{Finding the Permutations $\sigma_j$ by Sorting}
For $j=1,2$, we may use a sorting algorithm to find a permutation $\sigma_j^{-1}$ which puts the list \[\gr_j(1),\gr_j(2),\ldots,\gr_j(m_j)\] in non-decreasing order.  The function $\gr_i\comp \sigma_j$ is then non-decreasing.  To take advantage of the vineyard algorithm in our main algorithm, we will want to work with a sorting algorithm which generates the permutation $\sigma^{-1}_j$ as a product of transpositions of adjacent elements in $[m_j]$.  For this, we use the well known \emph{insertion-sort} algorithm \cite{cormen2009introduction}.  This yields $\sigma_j$ as a minimum length product of adjacent transpositions. 
\nomenclature[sigma]{$\sigma_j$}{permutation of $[m_j]$}

\subsection{Induced Free Implicit Representations at Each 2-Cell}\label{sec:InducedFreeImplicitReps}
Using \cref{Permutation_Lemma} above, we next show that for any 2-cell $e$ in $\cell(M)$, $\Phi$ yields a \fir \[\Phi^e=(\gr_1^e,\gr_2^e,D^e_1,D^e_2)\] of the discrete persistence module $M^e$ introduced in \cref{Sec:Discrete_Barcodes}, with $\gr^e_1$ and $\gr^e_2$ non-decreasing.  Thus, we can compute $\B{M^e}$ by computing an $RU$-decomposition of $\Phi^e$.  

\paragraph{Lift Maps}
Recall the definition of the set of template points $\PCal^e$ from \cref{Sec:Discrete_Barcodes}.  Define a function \[\lift^e:\sbx\to \PCal^e\] by taking $\blift^e(a)=u$, for $u$ the minimum element of $\PCal^e$ such that $a\leq u$ (where as elsewhere, $<$ denotes the partial order order $\R^2$). \nomenclature[lift]{$\lift^e$}{lift map at cell $e$}
In \cref{fig:lift_map} we illustrate $\lift^e$ and $\lift^{e'}$ for a pair of adjacent 2-cells $e$ and $e'$.

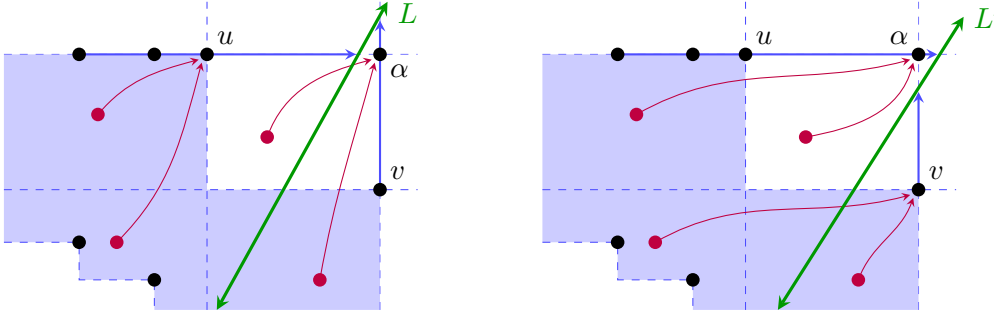
\begin{figure}[ht]
  \begin{center}
    \begin{tikzpicture}
      \fill[blue!20] (-5,-2.5) -- (-4,-2.5) -- (-4,-3) -- (-3,-3) -- (-3,-3.4) -- (0,-3.4) -- (0,-1.8) -- (-2.3,-1.8) -- (-2.3,0) -- (-5,0) -- cycle;
      
      \draw[dashed,color=blue!50] (-5,0) -- (0.5,0);
      \draw[->,thick,color=blue!70] (-4,0) -- (-0.32,0);
      \draw[dashed,color=blue!70] (-5,-1.8) -- (0.5,-1.8);
      \draw[dashed,color=blue!50] (0,-3.4) -- (0,0.7);
      \draw[->,thick,color=blue!70] (0,-1.8) -- (0,0.45);
      \draw[dashed,color=blue!70] (-2.3,-3.4) -- (-2.3,0.7);
      \draw[dashed,color=blue!70] (-5,-2.5) -- (-4,-2.5) -- (-4,-3) -- (-3,-3) -- (-3,-3.4);
      
      \foreach \p in {(0,0),(-2.3,0),(0,-1.8),(-4,-2.5),(-3,-3),(-3,0),(-4,0)}
        { \fill \p circle (0.09); }
          \node[below right] at (0,0) {$\alpha$};
          \node[above right] at (-2.3,0) {$u$};
          \node[above right] at (0,-1.8) {$v$};
     
      \foreach \p in {(-3.75,-0.8),(-1.5,-1.1),(-0.8,-3),(-3.5,-2.5)}
        { \fill[purple] \p circle (0.09); }
      \draw[purple,->] (-3.75,-0.8) to [out=50,in=195] (-2.41,-0.07);
      \draw[purple,->] (-1.5,-1.1) to [out=70,in=195] (-0.11,-0.07);
      \draw[purple,->] (-0.8,-3) to [out=80,in=255] (-0.07,-0.11);
      \draw[purple,->] (-3.5,-2.5) to [out=45,in=255] (-2.37,-0.11);
      
      \draw[<->,very thick,color=green!60!black] (-2.17,-3.4) -- (0.1,0.7);
      \node[right,color=green!60!black] at (0.1,0.55) {$L$};
    \end{tikzpicture}
    \hspace{0.5in}
    \begin{tikzpicture}
      \fill[blue!20] (-5,-2.5) -- (-4,-2.5) -- (-4,-3) -- (-3,-3) -- (-3,-3.4) -- (0,-3.4) -- (0,-1.8) -- (-2.3,-1.8) -- (-2.3,0) -- (-5,0) -- cycle;
      
      \draw[dashed,color=blue!50] (-5,0) -- (0.5,0);
      \draw[->,thick,color=blue!70] (-4,0) -- (0.25,0);
      \draw[dashed,color=blue!70] (-5,-1.8) -- (0.5,-1.8);
      \draw[dashed,color=blue!50] (0,-3.4) -- (0,0.7);
      \draw[->,thick,color=blue!70] (0,-1.8) -- (0,-0.5);
      \draw[dashed,color=blue!70] (-2.3,-3.4) -- (-2.3,0.7);
      \draw[dashed,color=blue!70] (-5,-2.5) -- (-4,-2.5) -- (-4,-3) -- (-3,-3) -- (-3,-3.4);
      
      \foreach \p in {(0,0),(-2.3,0),(0,-1.8),(-4,-2.5),(-3,-3),(-3,0),(-4,0)}
        { \fill \p circle (0.09); }
        \node[above left] at (0,0) {$\alpha$};
        \node[above right] at (-2.3,0) {$u$};
        \node[above right] at (0,-1.8) {$v$};
        
      \foreach \p in {(-3.75,-0.8),(-1.5,-1.1),(-0.8,-3),(-3.5,-2.5)}
        { \fill[purple] \p circle (0.09); }
      \draw[purple,->] (-3.75,-0.8) to [out=30,in=195] (-0.11,-0.07);
      \draw[purple,->] (-1.5,-1.1) to [out=10,in=255] (-0.07,-0.11);
      \draw[purple,->] (-0.8,-3) to [out=70,in=250] (-0.07,-1.91);
      \draw[purple,->] (-3.5,-2.5) to [out=30,in=195] (-0.11,-1.87);
      
      \draw[<->,very thick,color=green!60!black] (-1.87,-3.4) -- (0.6,0.5);
      \node[right,color=green!60!black] at (0.6,0.5) {$L'$};
    \end{tikzpicture}
  \end{center}
  \caption{Illustration of $\lift^e$ (left) and $\lift^{e'}$ (right) at two adjacent cells $e$ and $e'$, containing the duals of lines $L$ and $L'$, respectively.  The black dots represent points of $\PCal=\bigcup_{e} \PCal^e$.  Note that $u,\alpha\in \PCal^e$ and $v,\alpha\in \PCal^{e'}$, but $u\not\in \PCal^{e'}$ and $v\not\in \PCal^e$.  The maps $\lift^e$, $\lift^{e'}$ are illustrated by red arrows, for a few sample points (purple dots).  The shaded region in each figure is the subset of $\sbx$ on which $\lift^e\ne \lift^{e'}.$} 
  \label{fig:lift_map}
\end{figure}

Let \[\ord^e:\PCal^e\to \{1,2,\ldots,|\PCal^e|\}\] denote the unique order-preserving bijection, i.e., the inverse of the restriction of $\indx^e$ to $\{1,2,\ldots,|\PCal^e|\}$.
\nomenclature[ord]{$\ord^e$}{order map at cell $e$}

\paragraph{Free Implicit Representation of $M^e$}
Suppose that $\Phi=(\gr_1,\gr_2,D_1,D_2)$, our \fir of $M$, is of dimensions $(m_0,m_1,m_2)$.
For $j=1,2$, let $\sigma^e_j:[m_j]\to [m_j]$ be any permutation such that 
\[\gr_j^e:=\ord^e\comp \blift^e \comp \gr_j \comp \sigma^e_j: [m_j]\to \Z\]
is non-decreasing, and write $\sigma^e=(\sigma^e_1,\sigma^e_2)$.

Let $P_1$ and $P_2$ denote the permutation matrices corresponding to $\sigma^e_1$ and $\sigma^e_2$, respectively.  Let $D_1^e=D_1 P_1$ and $D_2^e=P_1^T D_2 P_2$.

\begin{proposition}\label{Prop:Induced_Presentations}
$\Phi^e:=(\gr^e_1,\gr_2^e,D_1^e,D_2^e)$ is a \fir of $M^e$. 
\end{proposition}
\begin{proof}
It's not hard to check that \[(\ord^e\comp \blift^e \comp \gr_1,\, \ord^e\comp \blift^e \comp \gr_2,D_1,D_2)\] is a \fir of $M^e$.  Given this, the result follows from \cref{Permutation_Lemma}.
\end{proof}

\paragraph{Non-Uniqueness of $\Phi^e$}
In general, $\Phi^e$ is not uniquely defined, because it depends on the pair of permutations $\sigma^e$, which needn't be unique. We will sometimes write $\Phi^e=\Phi(\sigma^e)$ to emphasize the dependence on $\sigma^e$.  We say any $\sigma^e$ chosen as above is \emph{valid}.  

\paragraph{Reading off $\P^e$ from an $RU$-Decomposition of $\Phi^e$}
For $j=1,2$, let 
\begin{equation*}
    f^e_j=\blift^e \comp \gr_j \comp \sigma^e_j={(\ord^e)}^{-1}\comp \gr_j^e,
\end{equation*}
and write $f^e=(f^e_1,f^e_2)$.  We call $f^e$ the \emph{template map} for $\Phi^e$.  \nomenclature[fe]{$f^e$}{template map}
Note that $f^e$ is independent of the choice of a valid $\sigma^e$. 

From \cref{BarcodeFromReducedMat}, \cref{Prop:Induced_Presentations}, and the definition of the barcode template $\P^e$ in \cref{Sec:Discrete_Barcodes}, we have the following relationship between $\Phi^e$, $f^e$, and $\P^e$:
\begin{equation}\label{BarcodeTemplateFormula}
\P^e=\big\{(f_1^e(j),f_2^e(k)) \mid (j,k)\in \pairs(\Phi^e)\big\} \cup \big\{(f_1^e(j),\infty) \mid j\in \ess(\Phi^e)\big\}.
\end{equation}

\subsection{Our Algorithm}\label{Sec:Algorithm_High_Level}
\cref{BarcodeTemplateFormula} tells us that for a 2-cell $e$ in $\cell(M)$, to compute the barcode template $\P^e$ it suffices to compute the template map $f^e$, along with the $RU$-decomposition of $\Phi^e=\Phi(\sigma^e)$, for some valid $\sigma^e$.  This is exactly what our algorithm does; we give a description of the algorithm here, deferring some details to later sections.  

We need to compute $\P^e$ at every 2-cell $e$.  A na\"ive approach would be to do the computation from scratch at each $2$-cell $e$.  However, we can do better, by leveraging the work done at one cell to expedite the computation at a neighboring cell.

We proceed as follows:  Let $G$ denote the \emph{dual graph} of $\cell(M)$; this is the undirected graph with a vertex for each 2-cell $e$ of $\cell(M)$, and an edge $[e,e']$ for each pair of adjacent 2-cells $e,e'\in \cell(M)$.  
The dual graph is illustrated in \cref{fig:dual_graph}. 
\nomenclature[G]{$G$}{dual graph of $\cell(M)$}
We compute a path  $\pth=e_0,e_1,\ldots e_w$ in $G$ which visits each 2-cell at least once. Our algorithm for computing $\pth$ is discussed below, in \cref{Sec:Computing_The_Path}.  
\nomenclature[Gamma]{$\Gamma$}{path through all 2-cells of $G$}

\begin{figure}[ht]
  \begin{center}
    \begin{tikzpicture}[dot/.style 2 args={circle,inner sep=1.5pt,fill,blue,label={#2:\textcolor{blue}{\footnotesize{$#1$}}},name=#1}]
        \draw[<->,thick,black!30!white] (0,0) -- (0,4);
        
        \draw[->,thick,black!30!white] (0,0.2) -- (5,4);
        \draw[->,thick,black!30!white] (0,1) -- (5,1.2);
        \draw[->,thick,black!30!white] (0,1.5) -- (5,3.5);
        \draw[->,thick,black!30!white] (0,2.1) -- (5,2.3);
        
        \node[dot={a}{above}] at (2,3.2) {};
        \node[dot={b}{right}] at (4.9,3.7) {};
        \node[dot={c}{right}] at (4.2,2.7) {};
        \node[dot={d}{above}] at (2.6,2.38) {};
        \node[dot={e}{below}] at (1.4,1.8) {};
        \node[dot={f}{left}] at (0.5,1.9) {};
        \node[dot={g}{left}] at (0.45,0.8) {};
        \node[dot={h}{below}] at (2.5,0.5) {};
        \node[dot={i}{right}] at (3.5,1.7) {};
        
        \draw[blue,thick] (a) -- (b) -- (c) -- (d) -- (e) -- (f) -- (a) -- (d);
        \draw[blue,thick] (c) -- (i) -- (h) -- (g) -- (e) -- (i);
    \end{tikzpicture}
  \end{center}
  \caption{The line arrangement $\cell(M)$ (in grey), together with its dual graph $G$ (in blue).  The path $\pth$ through $G$ might visit the vertices in the order $a, b, c, d, e, f, e, g, h, i$.}
  \label{fig:dual_graph}
\end{figure}
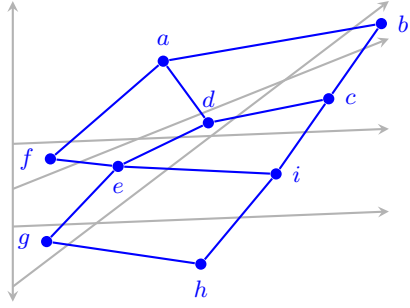

Let us adopt the convention of abbreviating an expression of the form $(-)^{e_i}$ by $(-)^i$.  For example, we write $f^{e_i}$ as $f^i$.

Once we have computed $\pth$, for each $i=1,\ldots,w$, we compute the template map $f^{i}$ and an $RU$-decomposition of $\Phi(\sigma^{i})=(\gr^i_1,\gr^i_2,D^i_1,D^i_2)$, for some valid choice of $\sigma^{i}$.  We proceed in order of increasing $i$.  

For $j=1,2$, we store $f^i_j$ in memory by separately storing its factors $\blift^{i-1}\comp \gr_j$ and $\sigma^{i-1}_j$; we discuss the data structures for this in \cref{Sec:Data_Structures}.  Thus, to compute $f^{i}_j$, we compute both $\blift^{i-1}\comp \gr_j$ and $\sigma^{i-1}_j$.  

The initial cell $e_0$ is chosen in a way that allows for a simple combinatorial algorithm to compute $\blift^0\comp \gr_j$ and  $\sigma^0_j$; see \cref{Sec:Initial_Cell}.  Letting $P_j$ denote the matrix representation of $\sigma^0_j$, we have that \[D^0_1=D_1 P_1,\qquad D^0_2=P_1^T D_2P_2.\]  Thus, $D^0_1$ and $D^0_2$ can be obtained from $D_1$ and $D_2$ by performing row and column permutations.  Given $D^0_1$ and $D^0_2$, we compute an $RU$-decomposition of $\Phi^0$ via an application of the standard persistence algorithm.

For $1\leq i\leq w$, we compute $f^i$ as an update of $f^{i-1}$, and we compute the $RU$-decomposition of 
$\Phi^i$ as an update of the $RU$-decomposition of $\Phi^{i-1}$.  Let us explain this in more detail (with yet more detail to come in later sections).  To update $f^{i-1}_j$, we update both $\blift^{i-1}\comp \gr_j$ and $\sigma^{i-1}_j$.   First, we update $\blift^{i-1}\comp \gr_j$ to obtain $\blift^i\comp \gr_j$; details of this computation are given in \cref{Sec:Computations_at_Cell_i}.     
Second, we update $\sigma^{i-1}_j$ to obtain $\sigma^i_j$ as follows.  Define 
\begin{equation*}
    \tilde \gr^i_j:=\ord_i \comp \blift^i \comp \gr_j \comp \sigma^{i-1}_j.
\end{equation*}
Note the distinction between $\tilde \gr^i_j$ and $\gr^i_j$: the former is defined in terms of $\sigma^{i-1}_j$, while the latter is defined in terms of $\sigma^i_j$.

Applying the insertion-sort algorithm to $\tilde\gr^i_j$, we compute   
a sequence of transpositions of adjacent elements in $[m_j]$ such that, for $\tau^{-1}_j:[m_j]\to [m_j]$ the composition of these transpositions, $\tilde \gr^i_j \comp \tau_j$ is non-decreasing.  We take $\sigma^i_j = \sigma^{i-1}_j \comp \tau_j$.  Clearly, then, \[\gr^i_j= \ord_i \comp \blift^i \comp \gr_j \comp \sigma^i_j\] is non-decreasing, so $\sigma^i$ is valid.  

Note that for $P_j$ the matrix representation of $\tau_j$, we have \[D^i_1=D^{i-1}_1 P_1,\qquad D^{i}_2=P_1^T D^{i-1}_2 P_2.\]

To compute an $RU$-decomposition of $\Phi^i(\sigma^e)$ from our $RU$-decomposition of $\Phi^{i-1}(\sigma^e)$, we exploit the decomposition of $\tau^{-1}_j$ as a sequence of transpositions provided by the insertion-sort algorithm, and apply the $RU$-update algorithm of \cite{cohen2006vines} repeatedly, performing an update of the $RU$-decomposition for each transposition in the sequence.  We note that neither $\tau$, nor $\tau^{-1}$, nor the decomposition of $\tau^{-1}_j$ into transpositions ever needs to be stored explicitly in memory.  Rather, we use each transposition to perform part of the update as soon as it is computed; after this, there is no need to store the transposition.  

It remains to explain how we compute $\pth$, $\blift^{0}\comp \gr_j$, and $\sigma^{0}$, and how we update $\blift^{i-1}\comp\gr_j$ to obtain $\blift^{i}\comp \gr_j$.  In what follows, we explain all of this, and also fill in some details about the data structures used by our algorithm.

\subsection{Computing the Path \texorpdfstring{$\Gamma$}{}}\label{Sec:Computing_The_Path}
We first explain how we choose the path $\pth=e_0,e_1,\ldots, e_w$.  
As above, let $G$ be the dual graph of the line arrangement $\cell(M)$.  To compute $\pth$, we first compute a weight for each edge of $G$.  The weight of $[e,e']$ is chosen to be an estimate of the amount of work our algorithm must do to pass between cells $e$ and $e'$, in either direction.  We defer the details of how we define and compute these edge weights until \cref{Sec:Choosing_Edge_Weights}.  While the choice of edge weights impacts our choice of $\pth$, and hence the speed of our algorithm for computing the barcode templates, our asymptotic complexity bounds for our algorithm are independent of the choice of edge weights.  

We take $e_0$ to be the topmost 2-cell in $\cell(M)$; the points of $e_0$ correspond under point-line duality to lines that pass to the right of all points in $\XiSp$.  

Call a path in $G$ starting at $e_0$ and visiting every vertex of $G$ a \emph{valid} path.  We'd like to choose $\pth$ to be a valid path of minimum length, but we do not have an efficient algorithm for computing such a path. (Indeed, we expect the problem is NP-hard).  Instead, we compute a path $\pth$ whose length is \emph{approximately} minimum.

Let $\pth^*$ be a valid path of minimum length.  It is straightforward to compute a valid path $\pth$ such that $\length(\pth)\leq 2\,\length(\pth^*)$: First we compute a minimum spanning tree $\M$ for $G$ via a standard algorithm such as Kruskal's Algorithm \cite{cormen2009introduction}. Via depth-first search of $\M$ starting at $e_0$, we can find a valid path $\pth$ in $\M$ which traverses each edge of $\M$ at most twice.  Since $\length(\M)\leq \length(\pth^*)$, we have that $\length(\pth)\leq 2\,\length(\pth^*)$.

In fact, an algorithm with a better approximation ratio is known:  \cite{hoogeveen1991analysis} shows that a variant of the Christofides $\frac{3}{2}$-approximation algorithm for the traveling salesman problem on a metric graph \cite{christofides1976} yields a valid path $\pth$ with $\length(\pth)\leq \frac{3}{2}\length(\pth^*)$.

\subsection{Data Structures}\label{Sec:Data_Structures}
Before completing the specification of our algorithm for computing the barcode templates $\P^e$, we need to describe the data structures used internally by the algorithm.

\paragraph{Persistent Homology and Vineyard Update Data Structures}
First, we mention that our algorithm uses each of the data structures specified in \cite{cohen2006vines} for computing and updating $RU$-decompositions. These consist consists of 4 sparse matrices---one to store each of $R_1,R_2$, $U_1$ and $U_2$---as well as several additional 1-D arrays which aid in performing the persistence and vineyard algorithms, and in reading barcodes off of the matrices $R_1$, $R_2$.  Since we use these data structures only in the way described in \cite{cohen2006vines}, we refer the reader to that paper for details.

\paragraph{Array Data Structures}
For $j=1,2$, we also maintain arrays $\g{j}{}$, $\u{j}{}$, $\sig{j}{}$, and $\sigInv{j}{}$, 
  each of length $[m_j]$:
\begin{itemize}
\item $\g{j}{}$ is a static array with $\g{j}{k}=\gr_j(k)$.
\item $\u{j}{k}$ is an array of pointers to entries of $\xiSuppMat$.
\item After our computations at cell $e_i$ are complete, 
\[\u{j}{k}=\lift^i\comp \gr_j(k),\quad \sig{j}{k}=\sigma^i_j(k),\quad \sigInv{j}{k}=(\sigma^i_j)^{-1}(k).\]
\end{itemize}

\begin{remark}\label{Rmk:Template_Map_Encoding}
Note that after all computations at cell $e_i$ are complete, we can use $\u{j}{}$ and $\sig{j}{}$ together to perform constant time evaluations of the template map \[f^i_j=\lift^i\comp \gr_j \comp \sigma^i_j.\]  Together with the $RU$-decomposition of $\Phi^i$, this allows us to efficiently read off $\P^e$ using \cref{BarcodeTemplateFormula}.
\end{remark}

\paragraph{The Lists $\partialLevelSet{j}{u}$}
We mentioned in \cref{Sec:Sparse_Data_Structure_For_T} that for each $u\in \PCal$, we store lists $\partialLevelSet{1}{u}$ and $\partialLevelSet{2}{u}$ at the entry of $\xiSuppMat$ corresponding to $u$.  We now specify what these lists store: After our computations at cell $e_i$ are complete, $\partialLevelSet{j}{u}$ stores $(\lift^i\comp \gr_j\comp \sigma^i_j)^{-1}(u)$ for $u\in \PCal^i$, and $\partialLevelSet{j}{u}$ is empty for $u\not\in \PCal^i$.  As we will see in \cref{Sec:Computations_at_Cell_i}, our algorithm uses the lists $\partialLevelSet{j}{u}$ to efficiently perform the required updates as we pass from cell $e_{i-1}$ to cell $e_i$.

\subsection{Computations at the Initial Cell \texorpdfstring{$e_0$}{}}\label{Sec:Initial_Cell}
We next describe in more detail the computations performed at the initial cell $e_0$, building on the explanation of \cref{Sec:Algorithm_High_Level}.  

To begin our computations at cell $e_0$, for each $u\in \PCal$, we initialize \[\partialLevelSet{j}{u}=(\lift^0\comp \gr_j)^{-1}(u).\]
Note that $(\lift^0\comp \gr_j)^{-1}(u)$ is nonempty only if $u$ is the rightmost element of $\PCal$ on the horizontal line passing through $u$.  Thus, the elements $u\in \PCal$ such that $\partialLevelSet{j}{u}$ is non-empty have unique $y$-coordinates.

To efficiently initialize the lists $\partialLevelSet{j}{u}$, we use an ${O(m \log m+ m\kappa)}$ time sweep algorithm, described in \cref{sec:alg_details}.  When we add $k\in [m_j]$ to the list $\partialLevelSet{j}{u}$, we also set $\u{j}{k}=u$.  

Next, we concatenate the non-empty lists $\partialLevelSet{j}{u}$ into single list of length $m_j$, in increasing order of the $y$-coordinate of $u$, and set $\sigInv{j}{}$ equal to this list.  For $j=1,2$, given $\sigInv{j}{}$, we construct $\sig{j}{}$ in the obvious way in time $O(m_j)$.  
We define $\sigma^{0}_j$ to be the permutation whose array representation is $\sig{j}{}$.   

Letting $P_j$ denote the matrix representation of $\sigma^0_j$, we use the arrays $\sigInv{1}{}$ and $\sigInv{2}{}$ to compute \[ D^0_1=D_1 P_1, \quad D^0_2=P_1^T D_2 P_2.\] Using the column sparse representation of $D_1$ and $D_2$ described in \cite{cohen2006vines}, which allows for implicit representations of row permutations, this takes $O(m)$ time.   

As already explained in \cref{Sec:Algorithm_High_Level}, we then apply the standard persistence algorithm to compute the RU-decompositions $D^0_1=R^0_1 U^0_1$, $ D^0_2=R^0_2U^0_2.$ 

This completes the work done by the algorithm at cell $e_0$.  

\subsection{Computations at Cell \texorpdfstring{$e_i$, for $i\geq 1$}{i}}\label{Sec:Computations_at_Cell_i}
In \cref{Sec:Algorithm_High_Level}, we outlined our algorithm for updating the template map $f^{i-1}$ and $RU$-decomposition $\Phi^{i-1}$, as we pass from cell $e_{i-1}$ to $e_1$.  We now give a more detailed account of this algorithm, filling in some details omitted earlier.  As explained in \cref{Sec:Algorithm_High_Level}, to update \[f^{i-1}_j=\lift^{i-1}\comp \gr_j\comp \sigma^{i-1},\] our algorithm separately updates the factors $\lift^{i-1}\comp \gr_j$ and $\sigma^{i-1}_j$.  

While it is possible to first completely update $\lift^{i-1}\comp \gr_j$ and then update $\sigma^{i-1}_j$ via an application of insertion sort, it is slightly more efficient to interleave the updates of the two factors, so that when we update the value of $\lift^{i-1}\comp \gr_j(k)$ for some $k\in [m_j]$, we immediately perform the insertion-sort transpositions necessitated by that update, along with the corresponding updates to the RU-decomposition.  This is the approach we take.  

We assume without loss of generality that $e_{i-1}$ lies below $e_i$.  The shared boundary of $e_{i-1}$ and $e_i$ lies on the line dual to some anchor $\alpha=(\alpha_1,\alpha_2)$.  $\xiSuppMat$ provides constant time access to the element $u=(u_1,u_2)\in \PCal$ immediately to the left of $\alpha$ and the element $v=(v_1,v_2)\in \PCal$ immediately below $\alpha$, if such $u$ and $v$ exist.  To keep our exposition simple, we will assume that $u$ and $v$ do both exist; the cases that either $u$ or $v$ do not exist are similar, but simpler.  
Note that 
\[u,\alpha\in \PCal^{i-1},\qquad v,\alpha\in \PCal^{i},\qquad u\not\in \PCal^i,\qquad  v\not\in \PCal^{i-1};\]
see \cref{fig:lift_map}.

Recall that for each $i$, $\lift^{i}\comp \gr_j$ is represented in memory using the data structures $\u{j}{}$ and $\partialLevelSet{j}{\cdot}$, whereas $\sigma^{i}_j$ is represented using $\sig{j}{}$ and $\sigInv{j}{}$.

To perform the required updates as we pass from cell $e_{i-1}$ to cell $e_i$, we first iterate through the list $\partialLevelSet{j}{\alpha}$ in decreasing order.  For each $k\in \partialLevelSet{j}{\alpha}$, if the $y$-coordinate of $\gr_j(k)$ is less than or equal to $v_2$, then $\lift^i\comp \gr_j(k)=v$; we remove $k$ from $\partialLevelSet{j}{\alpha}$, add $k$ to the beginning of the list $\partialLevelSet{j}{v}$, and set $\u{j}{k}=v$.  

If, on the other hand, the $y$-coordinate of $\gr_j(k)$ is greater than $v_2$, then $\lift^i\comp\gr_j(k)=\alpha$, and we do not perform any updates to $\u{j}{}$, $\partialLevelSet{j}{\alpha}$, or $\partialLevelSet{j}{v}$ for this value of $k$.  If in addition, $\u{j}{\protect\sigInv{j}{\protect\sig{j}{k}+1} } =v$, then we apply insertion sort to update $\sig{j}{}$, $\sigInv{j}{}$, and the $RU$-decomposition; specifically, we compute $\topsort(\sig{j}{k})$, for $\topsort$ the algorithm defined below.

\begin{algorithm}[ht]
	\caption{\protect$\topsort(w)$}\label{alg:topsort}
	\footnotesize
	\begin{algorithmic}[1]
		\Require $w\in \{1,\ldots, m_j-1\}$ such that $\u{j}{{\sigInv{j}{y}}} \leq\u{j}{\protect{\sigInv{j}{z}}}$ whenever $w<y<z$
		\Ensure Updated $\sig{j}{}$ and $\sigInv{j}{}$ such that $\u{j}{{\sigInv{j}{y}}} \leq\u{j}{\protect{\sigInv{j}{z}}}$ whenever ${w\leq y<z}$; correspondingly updated $RU$-decomposition
		\State $\curr\leftarrow w$;
	\While{$\curr<m_j$ and $\u{j}{{\sigInv{j}{\curr}}} >\u{j}{\protect{\sigInv{j}{\curr+1}}}$}
	        \State swap $\sigInv{j}{\curr}$ and $\sigInv{j}{\curr+1}$
	        \State $\sig{j}{\sigInv{j}{\curr}}\leftarrow \curr$
	        \State $\sig{j}{\sigInv{j}{\curr+1}}\leftarrow\curr+1$
               \State Perform the corresponding updates to the $RU$-decomposition, as described in \cref{Sec:Algorithm_High_Level}.
	       \State $\curr\leftarrow \curr+1$ 
	\EndWhile
	\end{algorithmic}
\end{algorithm}

Once we have finished iterating through the list $\partialLevelSet{j}{\alpha}$, we next iterate through the list $\partialLevelSet{j}{u}$ in decreasing order.  For each $k\in \partialLevelSet{j}{u}$, we perform updates exactly as we did above for elements of $\partialLevelSet{j}{\alpha}$, with one difference: If the second coordinate of $\gr_j(k)$ is greater than $v_2$, then we must remove $k$ from $\partialLevelSet{j}{u}$, add $k$ to the beginning of the list 
$\partialLevelSet{j}{\alpha}$, and set $\u{j}{k}=\alpha$.  

When we have finished iterating through the list $\partialLevelSet{j}{u}$, our updates at cell $e_i$ are complete.
 
\subsection{Choosing Edge Weights for \texorpdfstring{$G$}{G}}\label{Sec:Choosing_Edge_Weights}
We have seen in \cref{Sec:Computing_The_Path} that the path $\pth$ depends on our choice of weights $w(\e)$ on the edges $\e$ of $G$.  We now explain how we choose and compute these weights.  
 As we will explain in \cref{Sec:Speedups}, computing these weights is also the first step in two practical improvements to our algorithm.

In practice, the cost of our algorithm for computing the barcode templates is dominated by the cost of updating the $RU$-decompositions. On average, we expect the cost of updating the $RU$-decomposition as we traverse edge $E$ in $G$ to be roughly proportional to the total number of insertion-sort transpositions performed.  Thus, if it were the case that the average number $t(\e)$ of transpositions performed as we traverse $\e$ were independent of the choice of path $\pth$, then it would be reasonable to take $w(\e)=t(\e)$.  In fact, $t(\e)$ does depend on $\pth$.  Nevertheless, we can give a simple, computable estimate of $t(\e)$ which is independent of $\pth$.  We choose $w(\e)$ to be this estimate.  For $\e=[e_1,e_2]$, our definition of $w(\e)$ in fact depends only on the anchor line $L$ containing the common boundary of $e_1,e_2$, so that we may write $w(\e)=w(L)$.

To prepare for the definition of $w(L)$, we introduce some terminology, which we will also need in \cref{Sec:TotalTranspositions}.

\paragraph{Switches and Separations}
For $e$ and $e'$ adjacent 2-cells of $\cell(M)$, we say $r,s\in \sbx$ \emph{switch at $(e,e')$} if either \[\lift^e(r)<\lift^e(s)\quad \text{and} \quad\lift^{e'}(s)< \lift^{e'}(r),\] or \[\lift^{e'}(r)< \lift^{e'}(s)\quad \text{and}\quad\lift^{e}(s)< \lift^{e}(r).\]  Similarly, for $r,s$ incomparable, we say \emph{$r,s$ separate} at $(e,e')$ if either \[\lift^e(r)\ne \lift^e(s)\quad\text{and}\quad \lift^{e'}(r)=\lift^{e'}(s),\] or \[\lift^e(r)=\lift^e(s)\quad\text{and}\quad\lift^{e'}(r)\ne\lift^{e'}(s).\]

We omit the straightforward proof of the following:

\begin{lemma}\label{Face_Independence_Of_Switches}
Suppose $(e,e')$ and $(f,f')$ are pairs of adjacent 2-cells of $\cell(M)$, with the shared boundary of each pair lying on the same anchor line $L$.  Then 
\begin{enumerate}[(i)]
\item $r,s\in \R^2$ switch at $(e,e')$ if and only if $r,s$ switch at $(f,f')$.
\item $r,s$ separate at $(e,e')$ if and only if $r,s$ separate at $(f,f')$.  
\end{enumerate}
\end{lemma}

In view of \cref{Face_Independence_Of_Switches}, for $L$ an anchor line, we say that $a,b\in \R^2$ \emph{switch at $L$} if $a,b$ switch at $e,e'$ for any adjacent 2-cells $e,e'$ whose boundary lies on $L$.  Analogously, we speak of $a,b\in \R^2$ \emph{separating at $L$}.  

\begin{lemma}
If $r,s$ switch at any anchor line $L$, then $r,s$ are incomparable.  
\end{lemma}

\begin{proof}
Suppose $L=\dual_p(\alpha)$ for some anchor $\alpha$.  If $r,s$ switch at $L$, then there exist $u$, $v$ as in \cref{fig:lift_map}, and exchanging and $r$ and $s$ if necessary, we have that 
\[r_1\leq u_1<s_1,\quad s_2\leq v_2<r_2.\qedhere\]
\end{proof}

\begin{remark}\label{switch_remark}
Every time we cross an anchor line $L$, our algorithm performs one insertion-sort transposition for each pair $k,l\in [m_j]$ such that $\gr_j(k),\gr_j(l)$ switch at $L$.  For each pair $k,l\in [m_j]$ such that $\gr_j(k),\gr_j(l)$ separate at $L$, the algorithm may only perform a corresponding insertion-sort transposition when crossing $L$ in one direction---e.g., from above to below---and may sometimes not perform such a transposition even when crossing $L$ in this direction.  

It is reasonable to estimate, then, that for each pair $k,l\in [m_j]$ such that $\gr_j(k),\gr_j(l)$ separate at $L$, the algorithm performs a corresponding transposition roughly $\frac{1}{4}$ of the time.
\end{remark}

\paragraph{Definition of $w(L)$}
For an anchor line $L$, a finite set $Y$, and a function $f:Y\to \R^2$, define $\switch_L(f)$ (respectively $\sep_L(f)$) to be the number of unordered pairs $a,b\in Y$ such that $f(a),f(b)$ switch (respectively, separate) along $L$.  Motivated by \cref{switch_remark}, we define \[w(L)=\switch_L(\gr_1)+\switch_L(\gr_2)+\frac{1}{4}\sep_L(\gr_1)+\frac{1}{4}\sep_L(\gr_2).\]

\paragraph{Computing the weights $w(L)$}\
The weights $w(L)$ can be be computed using a simplified version of our main algorithm for computing all barcode templates: First, we choose a path $Q$ through the 2-cells of $\cell(M)$  starting at $e_0$ and crossing every anchor line $L$ once; for example, we can choose $Q$ to be a path through the rightmost cells of $\cell(M)$.  We then run a variant of the algorithm for computing the barcode templates described above, using the path $Q$ in place of $P$, and omitting all of the steps involving matrices and updates of $RU$-decompositions.  For $e$, $e'$ adjacent 2-cells in $Q$ with shared boundary on the anchor line $L=\dual_p(\alpha)$, we compute $w(L)$ as we pass from cell $e$ to cell $e'$.

To explain how this works, let $u,v$ be as in \cref{Sec:Computations_at_Cell_i}; for simplicity, assume that $u$ and $v$ exist, as we did there.  For any pair of elements $r,s$ that switch or separate at $L$, $\lift^e(r),\lift^e(s)\in \{u,\alpha\}$, so to compute $w(L)$ we only need to consider pairs whose elements lie in the lists 
$\partialLevelSet{j}{u}$ and $\partialLevelSet{j}{\alpha}$.   Further, the lines $x=u_1$ and $y=v_2$ determine a decomposition of the plane into four quadrants, and whether $r,s$ switch or separate is completely determined by which of these quadrants contain $r$ and $s$; see \cref{fig:lift_map}.  Using these observations, we can easily extend the update procedure described in \cref{Sec:Computations_at_Cell_i} to compute the weight $w(L)$ as we cross from $e$ into $e'$.

%% file: VRI_Time_Complexity.tex
\section{Cost of Computing and Storing the Augmented Arrangement}\label{Sec:Complexity_Results} 
In this section, we prove \cref{SimpleAugArrComplexity}, which bounds the cost of computing and storing $\S(M)$.  Recall that \cref{SimpleAugArrComplexity} is stated for persistence modules arising as the $i^{\mathrm{th}}$ persistent homology of a bifiltration.  Using language of \firs, we may state the result in a more general algebraic form:

\begin{proposition}\label{Prop:Algebraic_Complexity}
Let $M$ be 2-D persistence module of coarseness $\kappa$, and let $\Phi$ be a \fir of $M$ of dimensions $(m_0,m_1,m_2)$.  Letting $m=m_0+m_1+m_2$, we have that
\begin{enumerate}
\item $\S(M)$ is of size $O(m\kappa^2)$,
\item Our algorithm computes $\S(M)$ using $O(m^3 \kappa+m\kappa^2\log \kappa)$ elementary operations, and requires $O(m^2+m\kappa^2)$ storage.%.
\end{enumerate}
\end{proposition}

To see that \cref{SimpleAugArrComplexity} follows from \cref{Prop:Algebraic_Complexity}, let $\F$ be a (1-critical or multi-critical) bifiltration of size $l$, and recall from section \cref{Sec:Motivativing_Free_Implicit_Reps} that, using the construction of \cite{chacholski2012combinatorial}, in $O(l \log l)$ time we can compute a \fir of $H_i(\F)$ of dimensions $(m_0,m_1,m_2)$, with $m_0+m_1+m_2=O(l)$.

\subsection{Size of the Augmented Arrangement}
We prove \cref{Prop:Algebraic_Complexity}~(i) first.  As noted in \cref{sec:BuildingTheLineArrangement}, the DCEL representation of $\cell(M)$ is of size $O(\kappa^2)$.

At each 2-cell of $\cell(M)$, we store the barcode template $\P{}^e$. 
By considering the $RU$-decomposition, we see that if \[\Phi=(\gr_1,\gr_2,\MatLeft, \MatRight)\] is a \fir of a 1-D persistence module $N$ of dimensions $(m_0,m_1,m_2)$, then $|\B{N}|\leq m_1$.  
Hence, \cref{Prop:Induced_Presentations} implies that $|\P{}^e|\leq m_1$ for all 2-cells $e$.  Therefore, our representation of $\S(M)$ in memory is of total size $O(m_1\kappa^2)=O(m\kappa^2)$, as claimed.  

\subsection{Cost of Computing the Augmented Arrangement}\label{Sec:CostOfComputingBarcodeTemplates}
We now turn to the proof of \cref{Prop:Algebraic_Complexity}~(ii).  
As we have seen, our algorithm for computing $\S(M)$ involves several sub-computations.
For each sub-computation, a row in \cref{Table:Arrangement_Cost} lists the data computed by this sub-computation, a bound on the number of elementary operation required, and the sections in this paper where the details were discussed.

\begin{table}[ht]\footnotesize
	\begin{center}
	\caption{Cost of augmented arrangement sub-computations}
	\label{Table:Arrangement_Cost}
	\begin{tabular}{M{5cm} m{3.3cm}  l}
		\toprule
		\textbf{Data} & \textbf{$\#$ Elem. Operations} & \textbf{Details in}  \\
		\midrule
		Set $\XiSp=\supp \xi_0(M) \union \supp\xi_1(M)$  & $O(m^3)$ & \cref{sec:ComputingBettiNumbers} \\
		\midrule[0.3pt]
		Template points $\PCal$ (stored in $\xiSuppMat$) and list $\anchors$ & $O(m^2)$ & \cref{Sec:Sparse_Data_Structure_For_T}\\
		\midrule[0.3pt]
		Arrangement $\cell(M)$ (constructed via Bentley-Ottmann algorithm) & $O(\kappa^2 \log \kappa)$ & \cref{sec:BuildingTheLineArrangement} \\
		\midrule[0.3pt]
		Data structures for point location & $O(\kappa^2 \log \kappa)$ & \cref{Sec:Data_Structure} \\
		\midrule[0.3pt]
		Barcode templates $\P{}^e$ & $O(m^3 \kappa+(m+\log \kappa)\kappa^2)$ & \cref{Sec:Computing_Additional_Data_at_Faces} \\
		\bottomrule
	\end{tabular}
	\end{center}
\end{table}

The bounds in the first four rows of \cref{Table:Arrangement_Cost} were explained earlier, so it remains  to analyze the cost of our algorithm for computing the barcode templates $\P^e$.  The computation of $\P^e$ itself involves a number of steps, whose individual time complexities we again list in \cref{Table:Barcode_Cost}.

\begin{table}[ht]\footnotesize
	\begin{center}
	\caption{Cost of barcode template sub-computations}
	\label{Table:Barcode_Cost}
	\begin{tabular}{M{6.4cm} m{3.3cm} l}
		\toprule
		\textbf{Data} & \textbf{$\#$ Elem. Operations} & \textbf{Details in}  \\
		\midrule
		Trimming the \fir & $O(m^2)$ & \cref{Sec:Trimming} \\
		\midrule[0.3pt]
		Path $\pth$ (found via the 2-approximation algorithm for the optimal path, using Kruskal's MST algorithm) & $O(\kappa^2 \log \kappa)$ & \cref{Sec:Computing_The_Path} \\
		\midrule[0.3pt]
		$\partialLevelSet{j}{u}$, $\u{j}{}$, $\sig{j}{}$, $\sigInv{j}{}$ at cell $e_0$ & $O(m \log m+ m\kappa)$ & \cref{Sec:Initial_Cell}, \cref{sec:frontier_sweep_details} \\
		\midrule[0.3pt]
		RU-decomposition of $\Phi^0$ & $O(m^3)$ & \cref{sec:ComputationPersistenceBarcodes} \\ 
		\midrule[0.3pt]
		Reading the barcode template $\P^e$ off of the $RU$-decomposition of $\Phi^i$, for all $i\geq 0$ & $O(m\kappa^2)$ & \cref{sec:InducedFreeImplicitReps}, \cref{Sec:Data_Structures}\\
		\midrule[0.3pt]\midrule[0.3pt]
		$\partialLevelSet{j}{u}$ and $\u{j}{}$ at all cells $e_i$, $i\geq 1$ & $O(m\kappa^2)$ & \cref{Sec:Computations_at_Cell_i} \\
		\midrule[0.3pt]
		$\sig{j}{}$, $\sigInv{j}{}$ at all $e_i$, $i\geq 1$ & $O(m^2\kappa)$ & \cref{Sec:Computations_at_Cell_i} \\
		\midrule[0.3pt]
		RU-decompositions at all $e_i$, $i\geq 1$ & $ O(m^3 \kappa+\kappa^2)$ & \cref{Sec:Computations_at_Cell_i} \\
		\midrule[0.3pt]
		Weights $w(L)$ for all anchor lines $L$ & $O(m^2+m\kappa)$ & \cref{Sec:Choosing_Edge_Weights} \\
		\bottomrule
	\end{tabular}
	\end{center}
\end{table}

The bounds in all but the last four rows in \cref{Table:Barcode_Cost} (above the double horizontal line) were either explained earlier, or are clear from the discussion presented.  In the remainder of this section, we verify the last four bounds.

\paragraph{Cost of Updates of the $\partialLevelSet{j}{\cdot}$ and $\u{j}{}$ at $e_i$, $i\geq 1$}
In the notation of \cref{Sec:Computations_at_Cell_i}, to update the lists $\partialLevelSet{j}{\cdot}$ to their proper values at $e_i$, our algorithm considers performing an update for each element $k$ in the lists $\partialLevelSet{j}{u}$ and $\partialLevelSet{j}{\alpha}$.  In the worst case, for each cell $e_i$ there are $O(m)$ such elements $k$ to consider in total.  For each $k$, the updates of $\partialLevelSet{j}{\cdot}$ and $\u{j}{}$ take constant time.  Thus the total amount of work we need to do at cell $e_i$ is $O(m)$.  The path $P$ contains $O(\kappa^2)$ 2-cells, so the total work to perform the updates over all 2-cells is $O(m\kappa^2)$, as claimed.

\paragraph{A Bound on the Total Number of Transpositions}
To establish the next two bounds, we take advantage of the following result, which we prove in \cref{Sec:TotalTranspositions}:

\begin{proposition}\label{Prop:TotalTranspositions}
Our algorithm for computing all barcode templates performs a total of $O(m^2 \kappa)$ insertion-sort transpositions on 
$\sigInv{1}{}$ and $\sigInv{2}{}$.
\end{proposition}

\paragraph{Cost of Updates of $ \protect \sig{j}{}$, $ \protect \sigInv{j}{}$ at cells $e_i$, $i\geq 1$}
The cost of updating $\sig{}{}$ is proportional to the cost of updating $\sigInv{j}{}$.  Thus it is immediate from \cref{Prop:TotalTranspositions} that the total cost of updating these arrays is $O(m^2\kappa)$. 

\paragraph{Cost of Updates of $RU$-decomposition at cells $e_i$, $i\geq 1$}
There are $O(\kappa^2)$ cells to consider; this gives us the $O(\kappa^2)$ term.  For each transposition performed on $\sig{j}{}$, we call the vineyard algorithm, described in \cref{sec:ComputationPersistenceBarcodes}, at most twice.  Each call to the vineyard algorithm takes time $O(m)$.  By \cref{Prop:TotalTranspositions}, then, the total cost of all vineyard updates performed is $O(m^3 \kappa)$.  This gives the desired bound.

\paragraph{Cost of Computing Weights $w(L)$}
As explained in \cref{Sec:Choosing_Edge_Weights}, we compute the edge weights $w(L)$ using a stripped-down variant of our algorithm for computing the barcode templates $\P^e$.  Using \cref{Lem:Switches_And_Separations} below, it can be checked that computing all edge weights takes time $O(m^2+m \kappa)$.

\paragraph{Storage Requirements}
By \cref{Prop:Algebraic_Complexity}~(i), $\S(M)$ itself is of size $O(m\kappa^2)$, so our algorithm for computing $\S(M)$ requires at least this much storage.  Our algorithm for computing the Betti numbers requires $O(m^2)$ storage, as do the persistence algorithm and the vineyard updates to $RU$-decompositions.  The Bentley-Ottmann algorithm requires $O(\kappa^2)$ storage, as does Kruskal's algorithm.  Constructing the search data structures used for queries of $\S(M)$ also requires $O(\kappa^2)$ storage. From our descriptions of the data structures used in the remaining parts of our algorithm, it is clear that other steps of our algorithm for computing $\S(M)$ do not require more than $O(m^2+m\kappa^2)$ storage.  The bound of \cref{Prop:Algebraic_Complexity}~(ii) on the storage requirements of our algorithm follows. 

\subsection{Bounding Total Number of Transpositions Required to Compute all Barcode Templates}\label{Sec:TotalTranspositions}
To complete our proof of \cref{Prop:Algebraic_Complexity}~(ii), it remains to prove \cref{Prop:TotalTranspositions}.
To this end, for $r\in \sbx$, let 
\begin{align*}
\Pgeqr &:= \{u\in \PCal\mid  u\geq r\},\\
\lift(r)&:=\{u\in \PCal \mid u=\lift^e(r)\textup{ for some 2-cell }e\}.
\end{align*}
We leave to the reader the proofs of the following two lemmas. 

\begin{lemma}\label{lem:lifts_one_cell}
For $e$ a 2-cell, $L$ a line with $\dual_\ell(L)\in e$, and $u\in \Pgeqr$, $\lift^e(r)=u$ if and only if the following two conditions hold:
\begin{enumerate}
\item $\push_L(u)$ is the minimum element of $\push_L(\Pgeqr)$,
\item for all $u'\in \Pgeqr$, $u' \ne u$, with $\push_L(u')=\push_L(u)$, we have $u'<u$.
\end{enumerate}
\end{lemma}

\begin{lemma}\label{lem:possible_lifts}
For %$r\in \sbx$ and 
$u\in \Pgeqr$, we have that $u\in \lift(r)$ if and only if the following two conditions hold:
\begin{enumerate}
\item there exists no $w\in \Pgeqr$ with 
\[w_1< u_1 \quad \text{and} \quad w_2<u_2.\]
\item there exist no pair $v,w\in \Pgeqr$ with
\[v_1< u_1,\quad v_2=u_2,\quad w_1=u_1,\quad \text{and}\quad w_2<w_2.\]
\end{enumerate}
\end{lemma}
\cref{fig:lift_r} illustrates the shape of the set $\lift(r)$, as described by \cref{lem:possible_lifts}.

\begin{figure}[ht]
  \begin{center}
    \begin{tikzpicture}
      \draw[purple,dotted] (0,5) -- (0,1.4) -- (6.2,1.4);
      \fill[purple] (0,1.4) circle (0.09);
      \node[purple, below left] at (0,1.4) {$r$};   
      
      \foreach \p in {(0.2,4.4),(.9,4),(.9,3.3),(1.6,3.3),(2.3,3.3),(4,2.75),(4,2.3),(5,2.3),(5.8,1.9)}
        { \fill \p circle (0.09); }
    \end{tikzpicture} 
    \end{center}
  \caption{The shape of a set $\lift(r)$, as described by \cref{lem:possible_lifts}.}
  \label{fig:lift_r}
\end{figure}
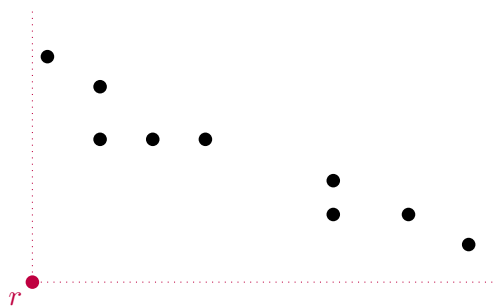

The next lemma shows that the number of anchor lines at which a given pair of points in $\sbx$ can switch or separate is at most two; it is the key step in our proof of \cref{Prop:TotalTranspositions}.

\begin{lemma}\label{Lem:Switches_And_Separations}
For $r,s\in \sbx$ incomparable,
\begin{enumerate}[(i)]
\item there is at most one anchor line $L$ at which $r$ and $s$ switch, and if such $L$ exists, then there is no anchor line at which $r$ and $s$ separate. 
\item there are at most two anchor lines at which $r$ and $s$ separate.
\end{enumerate}
\end{lemma}

\begin{proof}
Assume without loss of generality that $r_1 < s_1$.  Then, since $r$ and $s$ are incomparable, $r_2 > s_2$.  Let 
\begin{align*}
\Rc &= \{ (x,y)\in \lift(r) \mid x < s_1 \},\\
\Sc &= \{ (x,y) \in \lift(s) \mid y < r_2 \},\\
\Qc &=  \{(x,y)=\lift^e(r)=\lift^e(s) \text{ for some } e \in \cell(M)\}.
\end{align*}
The following observations, illustrated in \cref{fig:switches}, follow from  \cref{lem:possible_lifts}:
\begin{itemize}
\item If $\Rc$ is nonempty, there is an element $r'\in \Rc$ such that for all  $u\in \Rc$, $u_1\leq r'_1$ and $u_2\geq r'_2$.
\item Symmetrically, if $\Sc$ is nonempty, there is an element $s'\in \Rc$ such that for all  $u\in \Sc$, $u_1\geq s'_1$ and $u_2 \leq s'_2$.
\item If $\Qc$ is nonempty, there is an element $q^x\in \Qc$ such that for all $u\in \Qc$, $u_1\geq q^x_1$ and $u_2 \leq q^x_2$.
\item Symmetrically, if $\Qc$ is nonempty, there is an element $q^y\in \Qc$ such that for all $u\in \Qc$, $u_1\leq q^y_1$ and $u_2 \geq r_2$.
\end{itemize}
Clearly, $r',s',q^x$, and $q^y$ are unique, when they exist.  
Using \cref{lem:lifts_one_cell}, it is straightforward to check that $u\in \Qc$ if and only if $u\in \lift(r)\intrs\lift(s)$ and one of the following is true:
\begin{enumerate}
\item $u$ is incomparable to every element of $\Rc\union\Sc$.  
\item $\Rc$ is non-empty, $u_2=r'_2$, and there is no $v\in \lift(s)$ with $v_1=u_1$ and $v_2<u_2$, 
\item $\Sc$ is non-empty, $u_1=s'_1$, and there is no $v\in \lift(r)$ with $v_1<u_1$ and $v_2=u_2.$
\end{enumerate}
See \cref{fig:switches} for an illustration of $\Rc$, $\Sc$, and $\Qc$.

To finish the proof of \cref{Lem:Switches_And_Separations}, we consider seven cases.  For each, we explicitly describe the lines where $r$ and $s$ either switch or separate.  The verification of the claimed behavior in each case, which uses \cref{lem:lifts_one_cell} and the observations above, is left to the reader.  \cref{fig:switches} illustrates case 7.

\begin{sloppypar}
\begin{enumerate}
	\item $\Rc$ and $\Sc$ empty.  $\lift^e(r)=\lift^e(s)$ for every 2-cell $e\in\cell(M)$.  Therefore, $r$ and $s$ never switch or separate.
        \item $\Rc$ nonempty, $\Sc$ and $\Qc$ empty.  $\lift^e(r)< \lift^e(s)$ for every 2-cell $e\in\cell(M)$.  Again, no switches or separations.
        \item $\Sc$ nonempty, $\Rc$ and $\Qc$ empty.  Symmetric to the above, no switches or separations.
	\item $\Rc$ and $\Sc$ nonempty, $\Qc$ empty.  $\lift^e(r)< \lift^e(s)$ whenever $e$ lies below $L=\dual_p(\lub(r',s'))$, and $\lift^e(s)< \lift^e(r)$ whenever $e$ lies above $L$.  Hence $r$ and $s$ switch at $L$.
	\item $\Rc$ and $\Qc$ nonempty, $\Sc$ empty.  $\lift^e(r)< \lift^e(s)$ whenever $e$ lies below $L=\dual_p(\lub(r',q_x))$, and $\lift^e(r)=\lift^e(s)$ whenever $e$ lies above $L$.  Hence $r$ and $s$ separate at $L$.
	\item $\Sc$ and $\Qc$ nonempty, $\Rc$ empty.  Symmetric to the above, $r$ and $s$ separate at $\dual_p(\lub(q_y,s'))$.
	\item $\Rc$, $\Sc$, and $\Qc$ all nonempty.  $\lift^e(r)< \lift^e(s)$ whenever $e$ lies below $L=\dual_p(\lub(r',q_x))$; $\lift^e(r)=\lift^e(s)$ whenever $e$ lies above $L$ and below $L'=\dual_p(\lub(q_y,s'))$; and $\lift^e(s)<\lift^e(r)$ whenever $e$ lies above $L'$.  Hence $r$ and $s$ separate at $L$ and $L'$.\qedhere
\end{enumerate}
\end{sloppypar}
\end{proof}	
	
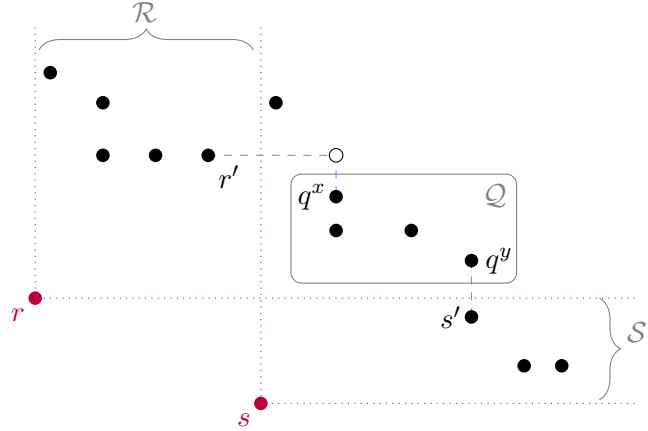
\begin{figure}[ht]
  \begin{center}
    \begin{tikzpicture}
      \draw[purple,dotted] (0,5) -- (0,1.4) -- (8,1.4);
      \fill[purple] (0,1.4) circle (0.09);
      \node[purple, below left] at (0,1.4) {$r$};   
      
      \draw[purple,dotted] (3,5) -- (3,0) -- (8,0);
      \fill[purple] (3,0) circle (0.09);
      \node[purple, below left] at (3,0) {$s$}; 
      
      \draw[black!50,decorate,decoration={brace,amplitude=8pt}] (7.5,1.4) -- (7.5,0.05);
      \node[black!50] at (8,0.95) {$\Sc$};
      \draw[black!50,decorate,decoration={brace,amplitude=8pt}] (0.05,4.7) -- (2.9,4.7);
      \node[black!50] at (1.45,5.2) {$\Rc$};
      \node[below left,black!50] at (6.4,3.05) {$\Qc$};
      
      \draw[rounded corners,black!50](6.4,1.60) -- (6.4,3.05) -- (3.4,3.05)--(3.4,1.60)--cycle;
      \draw[dashed,blue!50] (2.3,3.3) -- (4,3.3) -- (4,2.75);
      \draw[dashed,blue!50] (5.8,1.9) -- (5.8,1.15);
      
      \foreach \p in {(0.2,4.4),(.9,4),(.9,3.3),(1.6,3.3),(2.3,3.3),(3.2,4),(4,2.75),(4,2.3),(5,2.3),(5.8,1.9),(5.8,1.15),(6.5,0.5),(7,0.5)}
        { \fill \p circle (0.09); }
      \draw[fill=white] (4,3.3) circle (0.09);
      \node[below right] at (2.3,3.3) {$r'$};
      \node[left] at (5.8,1.15) {$s'$};
      \node[left] at (4,2.75) {$q^x$};
       \node[right] at (5.85,1.9) {$q^y$};
    \end{tikzpicture} 
    \end{center}
  \caption{An illustration of the case where $\Rc$, $\Sc$, and $\Qc$ are each non-empty (case 7).  Points in $\lift(r)\union \lift(s)$ are drawn as black dots. Observe that $r$ and $s$ separate at $\dual_p(\lub(r',q_x))$ and $\dual_p(\lub(q_y,s'))$, and at no other anchor line do $r$ and $s$ either switch or separate.}
  \label{fig:switches}
\end{figure}

\begin{proof}[Proof of \cref{Prop:TotalTranspositions}]
Fix $j\in \{1,2\}$, and let $k\ne k'\in [m_j]$.  First, we note that if $\gr_j(k)$ and $\gr_j(k')$ are comparable, then as we pass from cell $e_{i-1}$ to  cell $e_i$, our algorithm for computing barcode templates never performs an insertion-sort transposition of the values $k$ and $k'$ in $\sigInv{j}{}$.  This is because our initialization procedure at cell $e_0$, described in \cref{Sec:Initial_Cell} and \cref{sec:frontier_sweep_details}, chooses $\sigInv{j}{}$ such that if $\gr_j(k)<\gr_j(k')$, then $\sig{j}{k}<\sig{j}{k'}$.  Since $\lift^e(k)\leq \lift^e(k')$ for all 2-cells $e$, there thus is never any need to swap $k$ and $k'$.  

Therefore, as we pass from cell $e_{i-1}$ to  cell $e_i$, our algorithm performs an insertion-sort transposition of the values $k$ and $k'$ in $\sigInv{j}{}$ only if $\gr_j(k)$ and $\gr_j(k')$ either switch or separate at $(e_{i-1},e_i)$.  

Clearly, the number of pairs $k,k'\in[m_j]$ such that $\gr_j(k)$ and $\gr_j(k')$ either switch or separate at \emph{any} anchor line is less than $m_j^2/2$, the total number of pairs.  The path $\pth$, constructed via the minimum spanning tree construction in \cref{Sec:Computing_The_Path}, crosses each anchor line at most $2\kappa$ times.  By \cref{Lem:Switches_And_Separations} then, for each pair $k,k'\in [m_j]$ , the insertion-sort component of our algorithm performs a total of at most $4\kappa$ transpositions of that pair in $\sigInv{j}{}$.  Hence, the total number of transpositions performed by the algorithm altogether is at most $2(m_1^2+m_2^2)\kappa<2m^2 \kappa $.
\end{proof}

%% file: VRI_Optimizations.tex
\section{Speeding up the Computation of the Augmented Arrangement}\label{Sec:Speedups}
In this section, we describe several simple, practical strategies to speed up the runtime of our computation of $\S(M)$.  Used together, these strategies allow us to the compute augmented arrangements of the persistent homology modules of much larger datasets than would otherwise be possible.   

\subsection{Persistence Computation from Scratch When \texorpdfstring{$RU$}{RU}-Updates are too Slow}\label{Sec:Barcodes_Scratch}

\paragraph{Three Options for Computing a Barcode Template}
While an update to an $RU$-decomposition involving few transpositions is very fast in practice, an update to an $RU$-decomposition requiring many transpositions can be quite slow; when many transpositions are required, it is sometimes much faster to simply recompute the $RU$-decomposition from scratch using the standard persistence algorithm.  In our setting, the practical performance of our algorithm can be greatly improved if for consecutive 2-cells $e_{i-1},e_{i}\in \Gamma$ with the edge weight $w(e_{i-1},e_{i})$ greater than some suitably chosen threshold $t$, we simply compute the $RU$-decomposition of $\Phi^i$ from scratch, directly from $D^i_1$ and $D^i_2$.

Moreover, we can obtain significant additional speedups by avoiding the computation of the full $RU$-decomposition of $\Phi^i$ altogether at some cells $e_i$.  To explain this, we first note that to obtain $\P^{e_i}$ via \cref{BarcodeTemplateFormula}, we do not need the full $RU$-decomposition $\Phi^i$, but only $\PE(\Phi^i)=(\pairs(\Phi^i),\ess(\Phi^i))$; in particular we do not need $U^i_1$ and $U^i_2$.  
The algorithm for computing barcode templates described in \cref{Sec:Computing_Additional_Data_at_Faces} maintains the full $RU$-decomposition of each $\Phi^i$ because the vineyard algorithm requires this.  But if we are willing to compute $\PE(\Phi^{i+1})$ from scratch, then it is not necessary to compute the full $RU$-decomposition of $\Phi^{i}$; it suffices to compute $\PE(\Phi^{i})$.  Further, if in this case we have that $e_{i}=e_j$ for some $j<i$, then we do not even need to compute $\PE(\Phi^{i})$ at all at cell $e_i$, since we have already done so at an earlier step.  

In recent years, several algorithms for computing barcodes have been introduced which are much faster than the standard persistence algorithm \cite{otter2015roadmap,bauer2014phat,bauer2014clear}.  For example, a few such algorithms are implemented in the software library PHAT \cite{bauer2014phat}.  % and GUDHI \cite{boissonnat2013Dey,maria2014gudhi}.  
Given a \fir $\Phi$ as input, these algorithms compute $\PE(\Phi)$, but do not compute the full $RU$-decomposition of $\Phi$.\footnote{However, some fast algorithms for persistence computation can be readily adapted to compute $R_1$, $R_2$, $U_1^{-1}$, and $U_2^{-1}$.  For example, as explained to us by Ulrich Bauer, this is true for the ``twist" variant of the standard persistence algorithm \cite{chen2011persistent}.}

Let us restrict our attention a single such algorithm, say, the \emph{clear and compress algorithm} implemented in PHAT \cite{bauer2014clear}.

To compute $\PE(\Phi^i)$, then, we have three options available to us:
\begin{enumerate}[(A)]
\item Use the clear and compress algorithm if $e_i\ne e_j$ for all $j<i$; do nothing if $e_i=e_j$ for some $j<i$.
\item Compute the full $RU$-decomposition of $\Phi^i$ from scratch using the standard persistence algorithm.
\item Use vineyard updates.  This option is only available if we chose option $B$ or $C$ at cell $e_{i-1}$, so that the full $RU$-decomposition of $\Phi^{i-1}$ was computed.
\end{enumerate}

Clearly, there is a tradeoff between options $\Arm$ and $\Brm$: Option $\Arm$ is much faster, but choosing option $\Arm$ at cell $e_i$ precludes the use of option $\Crm$ at cell $e_{i+1}$.  How, then,  do we choose between these three options at each cell $e_i$?  We formulate this problem as a discrete optimization problem, which can solved efficiently by reduction to a min-cut problem.

\paragraph{Estimating Runtimes of the Different Options}
Our formulation of the problem requires us to first estimate the respective runtimes $c_i(\Arm)$, $c_i(\Brm)$, and $c_i(\Crm)$ of options $\Arm$, $\Brm$, and $\Crm$ at each cell $e_i$ in the path $\Gamma=e_0,e_1,\ldots,e_w$.  We will describe a simple strategy for this here, and then explain below how to modify our approach to correct for a drawback of the strategy.  

We take $c_i(\Arm)=0$ if $e_i=e_j$ for some $j<i$, and otherwise we take $c_i(\Arm)$ to be some constant $c(\Arm)$ independent of $i$.  Similarly, we take $c_i(\Brm)=c(\Brm)$ to be independent of $i$.    
To compute $c(\Arm)$, we compute $\PE(\Phi^0)$ using option $\Arm$, and set $c(\Arm)$ to be the runtime of this computation.  Similarly, to compute $c(\Brm)$, we compute the the full $RU$-decomposition of $\Phi^0$ from scratch and take $c(\Brm)$ to be the runtime.  

We set $c_0(\Crm)$ arbitrarily, say $c_0(\Crm)=0$. To compute $c_i(\Arm)$ for each $i\geq 1$, we perform several thousand random vineyard updates to the $RU$-decomposition of $\Phi^0$.  Using timing data from these computations, we compute, for $j=1,2$, the average runtime $c^{\vine}_j$ of an update to the $RU$-decomposition corresponding to a transposition of adjacent elements in $[m_j]$.  Letting $L$ denote the anchor line containing the shared boundary of $e_{i-1}$ and $e_i$, and recalling the notation of \cref{Sec:Choosing_Edge_Weights}, we take 
\[c_i:=c^{\vine}_1(\switch_L(\gr_1)+\frac{1}{4}\sep_L(\gr_1))+c^{\vine}_2(\switch_L(\gr_2)+\frac{1}{4}\sep_L(\gr_2)).\]  
Some motivation for this choice of $c_i$ is provided by \cref{switch_remark}.

\paragraph{The Optimization Problem}
To decide between options $\Arm$, $\Brm$ and $\Crm$ at each cell $e_i\in \Gamma$, we solve the following optimization problem:
\begin{align*}
   \text{minimize: } & \sum_{i=0}^w c_i(X_i) \\
   \text{subject to: } 
                              &X_i\in \{\Arm,\Brm,\Crm\},\\
                              &\text{$X_i=\Arm \implies X_{i+1}\ne \Crm$},\\
                              &X_0\ne C.
\end{align*}
Clearly, this problem is equivalent to the integer linear program (ILP):
\begin{align*}
   \text{minimize: } &\sum_{i=0}^w \big(c_i(\Arm)x_{i}+c_i(\Brm)y_{i}+ c_i(\Crm)z_{i}\big) \\
   \text{subject to: } 
                  & x_i,y_i,z_i  \in \{0,1\}\\
                  &x_i+y_i+z_i=1\\
                   & x_i+z_{i+1}\leq 1,\\
                  & z_0=0.
\end{align*}
Using the constraints $x_i+y_i+z_i=1$, we can eliminate the variables $y_i$ from this ILP to obtain an equivalent ILP with a simpler set of constraints:
\begin{align*}
\text{minimize: } &\sum_{i=0}^w \big((c_i(\Arm)-c_i(\Brm))x_{i}+(c_i(\Crm)-c_i(\Brm))z_{i}\big) \\
\text{subject to: }  &x_i,z_i  \in \{0,1\}\\ 
                            &x_i+z_{i+1}\leq 1,\\                          
                            &z_0=0.                            
\end{align*}
The constraint matrix associated to the latter ILP is of a standard form, well known to be totally modular \cite{schrijver1998theory}.  While an ILP with totally unimodular constraint matrix can always be solved directly via linear programming relaxation, it is often the case that such an ILP can be cast as a network flow problem, in which case we can take advantage of very efficient specialized algorithms.  

In fact, as explained to us by John Carlsson, the simplified ILP above can be cast as the problem of finding a minimum cut in a network: First, the ILP can be cast as a maximum-weight independent set problem in a bipartite graph with non-negative vertex weights.  In any graph, the complement of an independent set is a vertex cover and vice versa, so the latter problem is in turn equivalent to a minimum-cost vertex cover problem in a bipartite graph.  It is well known that such a problem can be solved by computing a minimum cut in a flow network \cite[Section 8.4.3.1]{atallah1998algorithms}.

\paragraph{Dynamic Updates to our Estimates of Runtime Cost}
As mentioned above, there is a drawback to the approach to barcode template computation we have proposed here: $c(\Arm)$ and $c(\Brm)$ may not be very good estimates of the respective average costs of option $\Arm$ and option $\Brm$; after all, our estimates $c(A)$ and $c(B)$ are computed using very little data.  

Here is one way to correct for this: We first solve the optimization problem above for just the first few cells (say $5\%$) in the path $\Gamma$.  Using the solution, we then compute the barcode template for each of these cells.  As we do this, we record the runtime of the computation at each cell.  We next update the value of $c(\Arm)$ to the average run time of all computations performed using option $\Arm$ thus far.  We also update $c(\Brm)$, $c_1^{\vine}$, and $c_2^{\vine}$ in the analogous way.  We then use these updated values as input to the optimization problem for the next $5\%$ of cells in $\Gamma$.  We continue in this way until the estimates of $c(\Arm)$, $c(\Brm)$, $c_1^{\vine}$, and $c_2^{\vine}$ have stabilized.  Finally, we solve the optimization problem for all of the remaining cells in $\Gamma$ and use the solution to compute the remaining barcode templates.

\subsection{Coarsening of Persistence Modules}\label{Sec:Coarsening}
We have seen that size of the augmented arrangement $\S(M)$ depends quadratically on the coarseness $\kappa$ of $M$, and that computing $\S(M)$ requires $O(m^3 \kappa+(m+\log \kappa)\kappa^2)$ elementary operations.

Thus, to keep our computations small, we typically want to limit the size of $\kappa$ by coarsening our module.  As we now explain, doing this quite simple.  A similar coarsening scheme is mentioned in \cite{carlsson2009computing}.

Let $\G=(\G_1,\G_2,\ldots,\G_n):\Z^2\to \R^2$ be any grid function, as defined in \cref{Sec:Continuous_Extensions}.  
$\G$ extends to a functor $\ZCat^n\to \RCat^n$, which we also denote by $\G$.  For $M$ a 2-D persistence module, let $M^\G$ be a continuous extension of $M\circ \G$.  In view of \cref{Prop:BettiNumsContinuousExtensions}, the coarseness of the grid $\G$ controls the coarseness $\kappa$ of the module $M^\G$.

Let $d_I$ denote the multidimensional interleaving distance, as defined in \cite{lesnick2014theory}.  As explained there, $d_I$ is a particularly well-behaved metric on persistence modules.  The following proposition, whose easy proof we omit, makes precise the intuitive idea that a small amount of coarsening leads to a small change in our persistence module:

\begin{proposition}
If $|\G_j(z)-\G_j(z+1)|\leq \delta$ for $i=1,2$ and all $z\in \Z$, then $d_I(M,M^\G)\leq \delta$.   
\end{proposition}

The ``external stability theorem" of Landi \cite{landi2014rank}, mentioned earlier in \cref{sec:rank_inv_fibered_barcodes},  shows that if 2-D persistence modules $M$ and $N$ are close in the interleaving distance, then the fibered barcodes $\B{M}$ and $\B{N}$ will be close, in a precise sense.  This justifies the use of coarsening in conjunction with our visualization paradigm.  

\paragraph{Coarsening Free Implicit Representations}
As explained in \cref{Sec:Motivativing_Free_Implicit_Reps}, in practice, we typically have access to $M$ via a \fir $\Phi=(\gr_1,\gr_2,\MatLeft,\MatRight)$ for $M$.  Since our algorithm for computing an augmented arrangement takes a \fir as input, to compute $\S(M^\G)$, we want to first construct a \fir $\Phi^\G$ of $M^\G$ from $\Phi$.  

Let $\ceil^\G:\R^n\to \G$ be the function which takes each $a\in \R^n$ to the minimal $z\in \G$ with $a\leq z$.
 
Define $\Phi^\G=(\ceil^\G\circ \gr_1,\ceil^\G\circ \gr_2,\MatLeft,\MatRight)$.  We leave the proof of the following to the reader:

\begin{proposition}  
$H(\Phi_G)\simeq M^\G$.
\end{proposition}

Thus, to obtain a \fir of a coarsening of $M$, it suffices to simply coarsen the grade functions in our \fir of $M$.

\subsection{Parallelization}\label{Sec:Parallel_Comp}
The problem of computing the barcode templates $\P^e$ is embarrassingly parallelizable.  Here is one very simple parallelization scheme: Given $l^2$ processors, where $l$ is less than or equal to the number of anchors, we can choose the $l$ anchor lines $L$ with the largest values of $w(L)$.  These lines divide $\cell(M)$ into at most $l^2$ polygonal cells $\{C_k\}$ with disjoint interiors, and the remaining anchor lines induce a line arrangement $\cell(M)_k$ on each $C_k$.  On processor $k$, we can run our main serial algorithm, described above, to compute the barcode templates at each 2-cell of $\cell(M)_k$.  

For this, we need to make just one modification to the algorithm: When choosing our path $\pth$ through the 2-cells of $\cell(M)_k$, we generally cannot choose our initial cell $e$ to be the cell $e_0$ described in \cref{Sec:Computing_The_Path}, since $e_0$ may not be be contained in $\cell(M)_k$.  Instead, we choose the initial cell $e$ arbitrarily.

This means we cannot use the approach of \cref{Sec:Initial_Cell} to initialize the data structures $\sig{j}{}$, $\sigInv{j}{}$, $\u{j}{}$, and $\partialLevelSet{j}{\cdot}$ at cell $e$.  One way to initialize these data structures is to chose an arbitrary affine line $L$ with $\dual_\ell(L)\in e$, and consider the behavior of the map $\push_L$ on $\XiSp\cup \im(\gr_1) \cup \im(\gr_2)$, using exact arithmetic where necessary.

%% file: VRI_Runtimes.tex
\section{Preliminary Runtime Results}\label{Sec:Runtimes}
We now present runtimes for the computation of augmented arrangements arising from synthetic data.  

We emphasize that these computational results are preliminary, as our implementation of RIVET does not yet take advantage of some key optimizations.  First, the code that produced these results employs a highly simplified variant of the scheme detailed in \cref{Sec:Barcodes_Scratch} for computing barcode templates; this variant chooses only between options (B) and (C) at each edge crossing, which is less efficient than what is proposed in \cref{Sec:Barcodes_Scratch}.  Secondly, this code stores the columns of our sparse matrices using linked lists; it is known that a smarter choice of data structure for storing columns can lead to major speedups in persistence computation \cite{bauer2014phat}.  Finally, as mentioned in \cref{Sec:Intro_Computation}, our current implementation runs only on a single processing core.  We expect to see substantial speedups after parallelizing the computation of barcode templates as proposed in \cref{Sec:Parallel_Comp}.

Our computations were run on a single (slow) 800 MHz core in a 32-core server with 96 GB of RAM. 
However, for the computations reported here, only a fraction of the memory was required.
For example, RIVET used approximately 9.3 GB of RAM for our largest computation.

\paragraph{Noisy Circle Data}
Data sets we consider are point clouds sampled with noise from an annulus, such as the center point cloud in \cref{fig:barcode_problems}.
Specifically, 90\% of the points in each data set $X$ are sampled randomly from a thick annulus in a plane, and 10\% are sampled randomly from a square containing the annulus.
We define a codensity function $\gamma:X\to \R$ by taking $\gamma(p)$ to be equal to the number of points of $X$ within some fixed distance to $p$.  We then construct the Vietoris-Rips bifiltration $\F:=\Rips(\gamma)$, described in \cref{Sec:Multi_D_PH_Intro}, taking the metric on $X$ to be the Euclidean distance, with the scale parameter for the Vietoris-Rips complexes capped at a value slightly larger than the inner diameter of the annulus.

\paragraph{Computing the Graded Betti Numbers}
\Cref{Table:Circle_Betti} displays the average runtimes for computing the graded Betti numbers $\xi_0$, $\xi_1$, and $\xi_2$ of $H_i(\F)$, for $i\in\{1,2\}$.
Each row gives the averages from three point clouds of the specified size.
For example, we generated three point clouds of 100 points, and the average number of 1-simplices in the resulting bifiltrations was 3,820, so computing $H_0$ homology required working with a bifiltration of average size 3,920 simplices.

\begin{table}[ht]\footnotesize
	\begin{center}
	\caption{Average runtimes for computing the bigraded Betti numbers of the noisy circle data}
	\label{Table:Circle_Betti}
	\begin{tabular}{ccrr}
		\toprule
		& Points & Simplices & Runtime (sec.) \\
		\midrule
		\multirow{4}{*}{$H_0$} & 100 & 3,920 & 0.11 \\
		& 200 &	15,719 & 0.70 \\
		& 300 &	35,324 & 2.38 \\
		& 400 &	62,415 & 3.99 \\
		\midrule
		\multirow{4}{*}{$H_1$} & 100 & 91,876 & 4.01 \\
		& 200 &	755,211 & 59.4 \\
		& 300 &	2,560,718 & 264 \\
		& 400 &	6,052,584 & 790 \\
		\bottomrule
	\end{tabular}
	\end{center}
\end{table}

\paragraph{Building the Augmented Arrangement}
\Cref{Table:Circle_H0} displays the average runtimes (in seconds) to build $\S(H_0(\F))$.
As before, each row gives the averages from three point clouds of the specified size.
The average runtimes for computing $\S(H_0(\F))$ for each of four different coarsenings are displayed in the table.
Similarly, \cref{Table:Circle_H1} displays the average runtimes (in seconds) to build $\S(H_1(\F))$.

\begin{table}[ht]\footnotesize
	\begin{center}
	\caption{Runtimes for computing the augmented arrangement for $0^{\mathrm{th}}$ homology of the noisy circle data}
	\label{Table:Circle_H0}
	\begin{tabular}{crrrrr}
		\toprule
		& & \multicolumn{4}{c}{Runtimes (seconds)} \\
		\cmidrule(l){3-6}
		Points & Simplices & Bins: $5 \times 5$ & $10 \times 10$ & $15 \times 15$ & $20 \times 20$ \\
		\midrule
		100 &	3,920 &	0.17 &	0.39 &	0.97 &	1.51 \\
		200 &	15,719 &	0.40 &	7.90 &	21.9 &	46.1 \\
		300 &	35,324 &	0.46 &	10.3 &	46.7 &	113 \\
		400 &	62,415 &	1.25 &	22.0 &	180 &	637 \\
		\bottomrule
	\end{tabular}
	\end{center}
\end{table}

\begin{table}[H]\footnotesize
	\begin{center}
	\caption{Runtimes for computing the augmented arrangement for $1^{\mathrm{st}}$ homology of the noisy circle data}
	\label{Table:Circle_H1}
	\begin{tabular}{crrrrr}
		\toprule
		& & \multicolumn{4}{c}{Runtimes (seconds)} \\
		\cmidrule(l){3-6}
		Points & Simplices & Bins: $5 \times 5$ & $10 \times 10$ & $15 \times 15$ & $20 \times 20$ \\
		\midrule
		100 &	91,876 &	0.60 &	1.33 &	2.35 &	3.61 \\
		200 &	755,211 &	13.1 &	211 &	1,012 &	3,368 \\
		300 &	2,560,718 &	39.9 &	374 &	3,364 &	7,510 \\
		400 &	6,052,584 &	228 &	6,893 &	18,346 &	59,026 \\
		\bottomrule
	\end{tabular}
	\end{center}
\end{table}

%% file: Interface_Details.tex
\subsection{Details of the RIVET Interface}\label{Sec:Interface_Details}

Expanding on \cref{Sec:Intro_Visualization}, we now provide some more detail about RIVET's graphical interface. 

As discussed in \cref{Sec:Computing_Line_Arrangement}, the module $M$ is input to RIVET as a free implicit representation \[\Phi=(\gr_1,\gr_2,D_1,D_2).\]  RIVET uses $\gr_1$ and $\gr_2$ to choose the bounds for the Line Selection Window and Persistence Diagram Window.  To explain, let $A$ and $B$ denote the greatest lower bound and least upper bound, respectively, of $\im \gr_1\cup \im \gr_2$.  In order to avoid discussion of uninteresting edge cases, we will assume that $A_1<B_1$ and $A_2<B_2$.

\paragraph{Choice of Bounds for the Line Selection Window}
We take the lower left corner and upper right corner of the Line Selection Window to be $A$ and $B$, respectively.\footnote{For a more intrinsic choice of bounds for the Line Selection Window, one could instead take the lower left and upper right corners of the window to be the greatest lower bound and least upper bound, respectively, of the set $\{a\in \R^2\mid \xi_i(M)(a)> 0\textup{ for some } i\in \{0,1,2\}\}$.  However, we feel that for a typical TDA application, the extrinsic bounds for the Line Selection Window we have proposed provide a more intuitive choice of scale.}
By default, the Line Selection Window is drawn to scale.  A toggle switch rescales (i.e., normalizes) the window so that it is drawn as a square on the screen.

\paragraph{Parameterization of Lines for Plotting Persistence Diagrams}
We next explain how, given a line $L\in \bar\L$, RIVET represents $\B{M^L}$ as a persistence diagram.  Let us first assume that the Line Selection Window is unnormalized; we treat the case where it is normalized at the end of this section.  Further, by translating the indices of $M$ if necessary, we may assume without loss of generality that $A=0$.

As noted in \cref{Sec:Intro_Visualization},  to plot $\B{M^L}$ as a persistence diagram, we need to first choose a parameterization $\gamma_L:\R\to L$ of $L$.  We choose $\gamma_L$ to be the unique order-preserving isometry such that:
\begin{enumerate}
\item If $L$ has finite positive slope, then $\gamma_L(0)$ is the unique point in the intersection of $L$ with the union of the non-negative portions of the coordinate axes.
\item If $L$ is the line $x=a$, then $\gamma_L(0)=(a,0)$.
\item If $L$ is the line $y=a$, then $\gamma_L(0)=(0,a)$.
\end{enumerate}

\paragraph{The Choice of Bounds for the Persistence Diagram Window}
The bounds for the Persistence Diagram Window are are chosen statically, depending on $M$ but not on the choice of $L$.  RIVET chooses the viewable region of the persistence diagram to be $[0,|B|]\times [0,|B|]$.

\paragraph{Representation of Points Outside of the Viewable Region of the Persistence Diagram}
It may be that for some choices of $L$, $\B{M^L}$ contains intervals $[\alpha,\beta)$, for $\alpha \geq |B|$ or $\beta \geq |B|$, so that $[\alpha,\beta)$ falls outside of the viewable region of the persistence diagram; 
indeed, the coordinates of some points in the persistence diagram can become huge (but finite), as the slope of the $L$ approaches 0 or $\infty$.  Thus, our persistence diagrams include some information on the top of the diagram, not found in typical persistence diagram visualizations, to record the points in the persistence diagram which fall outside of the viewable region:
\begin{itemize}
\item Above the main square region of the persistence diagram are two narrow horizontal strips, separated by a dashed horizontal line.  The upper strip is labeled ``inf," while the lower is labeled ``$<$ inf".  In the higher strip, we plot a point with $x$-coordinate $\alpha$ for each interval $[\alpha, \infty)$ with $\alpha\leq |B|$.  In the lower strip, we plot a point with $x$-coordinate $\alpha$ for each interval $[\alpha, \beta)\in \B{M^L}$ with $\alpha\leq |B|$ and $|B| < \beta < \infty$. 
\item Just to the right of each of the two horizontal strips is a number, separated from the strip by a vertical dashed line.  The upper number is the count of intervals $[\alpha,\infty)\in \B{M^L}$ with $|B|<\alpha$; the lower number is the count of intervals $[\alpha,\beta)\in \B{M^L}$ with $|B|<\alpha,\beta<\infty$.
\end{itemize}

\paragraph{Persistence Diagrams Under Rescaling}
If we choose to normalize the Line Selection Window, then RIVET also normalizes the persistence diagrams correspondingly.  To do this, it computes respective affine normalizations $\gr'_1$, $\gr'_2$ of $\gr_1$, $\gr_2$, so that after normalization, $A=(0,0)$ and $B=(1,1)$.  RIVET then chooses the parameterizations of lines $L\in \bar \L$ and computes bounds on the Persistence Diagram Window exactly as described above in the unnormalized case, but taking the input FI-rep to be $(\gr'_1,\gr'_2,D_1,D_2)$.

%% file: VRI_Algorithmic_Details.tex
\subsection{\texorpdfstring{Computing the Lists $\partialLevelSet{j}{u}$ at Cell $e_0$}{Computing the Lists LevelSet}} 
\label{sec:alg_details}\label{sec:frontier_sweep_details}

As mentioned in \cref{Sec:Initial_Cell}, to efficiently compute the lists $\partialLevelSet{j}{u}$ at cell $e_0$, we use a sweep algorithm.  As we did in \cref{Sec:Sparse_Data_Structure_For_T}, to keep notation simple we assume that $\ordfun_x$ and $\ordfun_y$ are the identity maps on $[n_x]$ and $[n_y]$ respectively, so that $(\ordfun_x\times \ordfun_y)^{-1}(\PCal)=\PCal$.
We assume that to start, $\gr_1$ and $\gr_2$ are both increasing with respect to colexicographical order; if this is not the case, then by \cref{Permutation_Lemma}, we can apply a sorting algorithm to modify $\Phi$ so that the assumption does hold. This sorting can be done in $O(m \log m)$ time.

Our sweep algorithm maintains a linked list \frontier of pointers to elements of $\PCal$. (These elements are stored in $\xiSuppMat$.)  Both the $x$- and $y$-coordinates of the entries of \frontier are always strictly decreasing. 

The algorithm iterates through the rows of the grid $[n_x]\times [n_y]$ from top to bottom.
The list \frontier is initially empty, and is updated at each row $s$, with the help of $\xiSuppMat$.
If row $s$ is empty, then no update is necessary at row $s$.
Otherwise, let $u$ be the rightmost entry of $\PCal$ in row $s$, and let $r$ be the column containing $u$.
If the last element of \frontier is also in column $r$, then this element is removed from \frontier and replaced by $u$.
Otherwise, we append $u$ to the end of \frontier. 

The algorithm then inserts each $k \in [m_j]$ with $\gr_j(k)=(r,s)$ for some $r$ into the appropriate list $\partialLevelSet{j}{-}$.  
(Since $\gr_j$ is assumed to be increasing with respect to colexicographical order, we have immediate access to all such $k$.)
Specifically, $k$ is added to the list $\partialLevelSet{j}{u}$, for $u$ the leftmost element of \frontier with $r \leq u_1$.  
The lists $\partialLevelSet{j}{u}$ are maintained in lexicographical order.
It is easy to check that $u=\lift^0 \circ \gr_j(k)$, as desired.  

The frontier-sweep algorithm is stated in pseudocode in \cref{alg:frontier}.

Updating \frontier at each row of $\xiSuppMat$ takes constant time, so the total cost of updating \frontier is $O(n_y)$.  
For each row $s$, we must iterate over \frontier once to identify the lists $\partialLevelSet{j}{-}$ into which we insert elements of $[m_j]$.  
There are $n_y$ rows, and the length of \frontier is $O(\kappa_x)$, so the total cost of these iterations over \frontier is $O(n_y \kappa_x)$.  
Inserting each $k\in [m_j]$ into the appropriate list $\partialLevelSet{j}{-}$ takes constant time, so the total cost of such insertions is $O(m)$.  
Thus, the total number of operations for the frontier-sweep algorithm (including the cost of the initial sorting to put $\gr_1$ and $\gr_2$ in the right form), is $O(m \log m+ n_y \kappa_x)=O(m \log m+ m\kappa)$.

\begin{algorithm}[h]
	\caption{Frontier-sweep algorithm for building the lists $\partialLevelSet{j}{u}$}\label{alg:frontier}
	\footnotesize
	\begin{algorithmic}[1]
		\Require $\gr_j$ (represented as a list in colexicographical order), $\xiSuppMat$  
		\Ensure{$\xiSuppMat$, updated so that for each $u\in \PCal^{e_0}$, $\partialLevelSet{j}{u}=(\lift^0\circ \gr_j)^{-1}(u)$, with each list sorted in lexicographical order}% and for $u\in \PCal\setminus\PCal^{e_0}$, $\partialLevelSet{j}{u}$ is empty}
		\State initialize \frontier as an empty linked list
		\For{$s = n_y$ to $1$} 
			\If{$\xiSuppMat$ has an entry in row $s$} \Comment{update \frontier\ for row $s$}
				\State let $r$ be the column containing rightmost element of $\PCal$ in row $s$
				\If{last element of \frontier is in column $r$}
					\State remove the last entry from \frontier
				\EndIf
				\State append the entry of $\PCal$ at row $s$, column $r$  to the end of \frontier 
			\EndIf
			\ForAll{entry $u$ in \frontier} \Comment{add elements with grades in row $s$ to the lists $\partialLevelSet{j}{-}$}
				\If{$u$ is the last element of \frontier}
					\State add all $k\in [m_j]$ such that $\gr_j(k)=(r,s)$ with $r\leq u_1$ to $\partialLevelSet{j}{u}$
				\Else
					\State let $v$ be the element after $u$ in \frontier
					\State add all $k\in [m_j]$ such that $\gr_j(k)=(r,s)$ with $v_1 < r\leq u_1$ to $\partialLevelSet{j}{u}$
				\EndIf
			\EndFor
		\EndFor
	\end{algorithmic}
\end{algorithm}